\documentclass[11pt,reqno]{amsart}
\usepackage{xifthen}
\usepackage{subfig}
\usepackage{pkgfile}
\usepackage{tikz}

\newtheorem{ass}[thm]{Assumption}
\newcommand{\Graph}{{\sf G}}
\newcommand{\Good}{\mathrm{Good}}
\newcommand{\Bad}{\mathrm{Bad}}
\renewcommand{\nu}{v}
\newcommand{\vx}{\bm x}
\newcommand{\sss}{\scriptscriptstyle}
\newcommand{\mvbeta}{\bm \beta}
\newcommand{\vX}{\bm X}
\newcommand{\erdos}{Erd\H{o}s-R\'enyi }
\newcommand{\inp}[2]{\langle #1 , #2 \rangle}

\usetikzlibrary{positioning}
\renewcommand{\leq}{\leqslant}
\renewcommand{\geq}{\geqslant}

\newcommand{\eps}{\varepsilon}

\newcommand{\norm}[1]{\left\Vert#1\right\Vert}

\newcommand{\vI}{\mathbf{I}}

\newcommand{\vP}{\mathbf{P}}

\newcommand{\ve}{\mathbf{e}}

\newcommand{\mvzero}{\boldsymbol{0}}

\newcommand\pgfmathsinandcos[3]{%
  \pgfmathsetmacro#1{sin(#3)}%
  \pgfmathsetmacro#2{cos(#3)}%
}
\newcommand\LongitudePlane[3][current plane]{%
  \pgfmathsinandcos\sinEl\cosEl{#2} 
  \pgfmathsinandcos\sint\cost{#3} 
  \tikzset{#1/.style={cm={\cost,\sint*\sinEl,0,\cosEl,(0,0)}}}
}
\newcommand\LatitudePlane[3][current plane]{%
  \pgfmathsinandcos\sinEl\cosEl{#2} 
  \pgfmathsinandcos\sint\cost{#3} 
  \pgfmathsetmacro\yshift{\cosEl*\sint}
  \tikzset{#1/.style={cm={\cost,0,0,\cost*\sinEl,(0,\yshift)}}} %
}

\newcommand\DrawLatitudeCircle[2][1]{
  \LatitudePlane{\angEl}{#2}
  \tikzset{current plane/.prefix style={scale=#1}}
  \pgfmathsetmacro\sinVis{sin(#2)/cos(#2)*sin(\angEl)/cos(\angEl)}
  \pgfmathsetmacro\angVis{asin(min(1,max(\sinVis,-1)))}
  \draw[current plane,thin,black] (\angVis:1) arc (\angVis:-\angVis-180:1);
  \draw[current plane,thin,dashed] (180-\angVis:1) arc (180-\angVis:\angVis:1);
}

\tikzset{%
  >=latex,
  inner sep=0pt,%
  outer sep=2pt,%
  mark coordinate/.style={inner sep=0pt,outer sep=0pt,minimum size=3pt,
    fill=black,circle}%
}

\usetikzlibrary{arrows}
\usepackage{pgfplots}
\usetikzlibrary{calc,fadings,decorations.pathreplacing}

\begin{document}
\title[Large subgraphs in pseudo-random graphs]{Large subgraphs in pseudo-random graphs}

\author[A.\ Basak]{Anirban Basak$^*$}
 \author[S.\ Bhamidi]{Shankar Bhamidi$^\dagger$}
 \author[S.\ Chakraborty]{Suman Chakraborty$^\ddagger$}
 \author[A.\ Nobel]{Andrew Nobel$^\mathsection$}
 
 
 \address{$^*$Department of Mathematics,
Weizmann Institute of Science, 
\newline\indent POB 26, Rehovot 76100, Israel}

\address{$^{\dagger\ddagger\mathsection}$Department of Statistics \& Operations Research, 
\newline\indent University of North Carolina at Chapel Hill,
\newline\indent 318 Hanes Hall, CB\# 3260, Chapel Hill, NC 27599, USA}

\date{\today}

\maketitle

\begin{abstract}
We consider classes of pseudo-random graphs on $n$ vertices for which the degree of every vertex and the co-degree between every pair of vertices are in the intervals $(np - Cn^\delta,np+Cn^\delta)$ and $(np^2- C n^\delta, np^2 +C n^\delta)$ respectively, for some absolute constant $C$, and  $p, \delta \in (0,1)$. We show that for such pseudo-random graphs the number of induced isomorphic copies of subgraphs of size $s$ are approximately same as that of an Erd\H{o}s-R\'{e}yni random graph with edge connectivity probability $p$ as long as $s \le (((1-\delta)\wedge \f{1}{2})-o(1))\log n/\log (1/p)$, when $p \in (0,1/2]$. When $p \in (1/2,1)$ we obtain a similar result. Our result is applicable for a large class of random and deterministic graphs including exponential random graph models (ERGMs), thresholded graphs from high-dimensional correlation networks,  Erd\H{o}s-R\'{e}yni random graphs conditioned on large cliques, random $d$-regular graphs and graphs obtained from vector spaces over binary fields. In the context of the last example, the results obtained are optimal.  Straight-forward extensions using the proof techniques in this paper imply strengthening of the above results in the context of larger motifs if a model allows  control over higher co-degree type functionals.  
\end{abstract}

\section{Introduction}\label{sec:introduction}
In the context of probabilistic combinatorics, the origin of random graphs dates back to \cite{E} where Erd\H{o}s used probabilistic methods to show the existence of a graph with certain Ramsey property. Soon after some of the foundational work on random graphs was established in \cite{ER1, ER2, ER3}. The simplest model of random graph is as follows: Fix $n\ge1 $ and  start with vertex set $\{1,2,\ldots,n\}$. For ease of notation hereafter, we denote $[n]:=\{1,2,\ldots,n\}$. Fix $p\in (0,1)$. The graph is formed by randomly drawing an edge between every pair of vertices $i \ne j \in [n]$, with probability $p$. This random graph is known as {\em Erd\H{o}s-R\'{e}yni random graph} with edge connectivity probability $p$. In the sequel we use $\Graph(n,p)$ to denote this graph (also known as {\em binomial random graph} in the literature). This model has stimulated an enormous amount of work over the last fifty years aimed at understanding properties of $\Graph(n,p)$ (see \cite{bollobas, JLR} and the references therein) in particular in the large network $n\to\infty$ limit.  

In the last decade there has been an explosion in the amount of empirical data on real world networks in a wide array of fields including statistics, machine learning, computer science, sociology, epidemiology. This has stimulated the development of a multitude of models to explain many of the properties of real world networks including high clustering coefficient, heavy tailed degree distribution and small world properties; see e.g.~\cite{newman2003structure,albert2002statistical} for wide-ranging surveys; see \cite{durrett2007random,van2009random,chung2006complex} for a survey of rigorous results on the formulated models.    Although such network models are not as simple as an \erdos graph, one would still like to investigate whether these posses properties similar to an \erdos graph. Therefore, it is natural to ask about how similar/dissimilar such network models  are compared to an \erdos graph.
Many researchers have delved deep into these questions and have found various conditions under which a given graph $\Graph_n:=([n], E_n)$, with vertex set $[n]$ and edge set $E_n$ looks similar to an \erdos random graph. In literature these graphs are known by various names. Following Krivelevich and Sudakov \cite{prg}, we call them pseudo-random graphs.  

One key property of an \erdos random graph is the following: For any two subsets of vertices $U,W \subset [n]$ the number of edges in a $\Graph(n,p)$, whose one end point is in $U$ and the other is in $W$, is roughly equal to $p |U||W|$, where $|\cdot|$ denotes the cardinality of a set. Hence, if a pseudo-random graph $\Graph_n$ is similar to $\Graph(n,p)$, then it must also satisfy similar properties. This motivated researchers to consider graphs which possess the above property. 

Foundational work for these sorts of questions began with Thomason \cite{thom2, thom1} in the mid-eighties where he used the term  {\em Jumbled graph} (see Definition \ref{dfn:jumbled}) to describe such graphs.  
Roughly it provided quantitative bounds on similarity between pseudo-random graphs and a $\Graph(n,p)$. This study provided some examples of jumbled graphs and explored various properties of them. It exposed a whole new interesting research area and numerous results were obtained afterwards. Further fundamental work in this area was established by Chung, Graham, and Wilson in \cite{cgw}. They coined the term {\em quasi-random} graphs to describe their models of pseudo-random graphs. They provided several equivalent conditions of pseudo-randomness; one of their major results is described in Section \ref{sec:disc}, Theorem \ref{prgcgw} . Paraphrasing these results, loosely they state the following:

\begin{quote}
	\emph{For a graph ${\sf G}_n$ to be pseudo-random one must have that the number of induced isomorphic copies of any subgraph ${\sf H}$ of \textbf{fixed size} (e.g.~${\sf H}$ is a triangle) must be roughly equal to that of a $\Graph(n,p)$.}
\end{quote}
  Another class of pseudo-random graphs is Alon's $(n,d,\lambda)$-graph (see \cite{alon_reg_graph}). These graphs are $d$-regular graphs on $n$ vertices such that the second largest absolute eigenvalue of its adjacency matrix is less than or equal to $\lambda$. Various graph properties are known for $(n,d,\lambda)$-graphs (see \cite{prg} and the references therein).
  
  As described above, the availability of data on real-world networks has stimulated a host of questions in an array of field, in particular in finding \emph{large} motifs in observed networks such as large cliques or communities (representing interesting patterns in biological or social networks) or understanding the limits of search algorithms in cryptology. Many of these questions are computationally hard and one is left with brute search algorithms over all possible subgraphs of a fixed size to check existence of such motifs.  Thus a natural question (again loosely stated) is as follows:

  \begin{quote}
  \emph{Can we find simple conditions on a sequence of graphs $\set{\Graph_n: n\geq 1}$ such that the number of induced isomorphic copies of any subgraph ${\sf H}$ of \textbf{growing size} (e.g. ${\sf H}$ is a clique of size $\log{n}$) must be roughly equal to that in  $\Graph(n,p)$? What are the fundamental limits of such conditions?  }	
  \end{quote}

Here we consider a class of pesudo-random graphs and study the number of induced copies of large subgraphs in those. More precisely, we will assume that our graph $\Graph_n$ satisfies the following two assumptions for some absolute constant $C$, and $p, \delta \in (0,1)$:

\noindent
{\bf Assumption A1.}
$$\max_{v \in [n]}{\left|\left|\sN_\nu\right| - {n}p\right|} < Cn^{\delta}.$$

\noindent
{\bf Assumption A2.} 
$$
\max_{ \nu \ne \nu' \in  [n]}{\left|\left|\sN_\nu \cap \sN_{\nu'}\right| - {n}p^2\right|} < Cn^{\delta}. $$
Here for any $v \in [n]$, $\sN_v:=\{u \in [n]: u \sim v\}$, where $u \sim v$ means that $u$ is connected to $v$ by an edge in $\Graph_n$. These two conditions are very natural to assume. For example, using Hoeffding's ineqaulity it is easy to check that for  $\Graph(n, p)$ assumptions {\bf (A1)} and {\bf (A2)} are satisfied for any $\delta >1/2$ with super-polynomially high probability. As we will see below, besides $\Graph(n,p)$ there are many examples of graph ensembles which satisfy assumptions {\bf (A1)}-{\bf (A2)}.  Further, these two specifications are quite basic and often very easy to check for a given graph (random or deterministic). 

In our main theorem below we show that for any such graph sequence the number of induced isomorphic copies of any slowly growing sub-graph is approximately same as that of an \erdos random graph with edge connectivity $p$. Before stating our main theorem let us introduce some notation: For any $r \in \N$, let us denote $\cG(r)$ be the collection of all graph with vertex set $[r]$. Further given any graph ${\sf H}$, denote $n_{\Graph_n}({\sf H})$ to be the number of induced isomorphic copies of ${\sf H}$ in $\Graph_n$ and $E({\sf H})$ to be edge-set of ${\sf H}$. Next, for $p\in (0,1)$, define $\gamma_p:= \max\{p^{-1}, (1-p)^{-1}\}$. We also write $\log(\cdot)$ to denote the natural logarithm, i.e.~logarithm with respect to the base $e$. When needed, we will specify the base $b$ and write as $\log_b(\cdot)$. For any two positive integers $s \le n$, let us denote $(n)_s:=n(n-1)\cdots (n-s+1)$. Now we are ready to state our main theorem.
\begin{thm}\label{thm:main}
Let $\{\Graph_n\}_{n \in \N}$ be a sequence of graphs satisfying assumptions {\bf (A1)} and {\bf (A2)}. Then, there exists a positive constant $C_0'$, depending on $p$, such that 
\beq\label{eq:thm_main_display}
\max_{s \le ((1-\delta)\wedge \f{1}{2}) \f{\log n}{\log \gamma_p}- C_0' \log \log n}\max_{{\sf H}_s \in \cG(s)}\left|\f{n_{\Graph_n}({\sf H}_s)}{\frac{(n)_s}{|\mathrm{Aut}({\sf H}_s)|} \left(\f{p}{1-p}\right)^{|E({\sf H}_s)|} (1-p)^{{s \choose 2}}}-1 \right|  \ra 0,
\eeq
as $n \ra \infty$.
\end{thm}

\bigskip

\begin{rmk}
	Note that 
	\[
	\E(n_{\Graph(n,p)}({\sf H}_s))= \frac{(n)_s}{|\mathrm{Aut}({\sf H})|} \left(\f{p}{1-p}\right)^{|E({\sf H})|} (1-p)^{{s \choose 2}}.
	\]
	Thus Theorem \ref{thm:main} establishes that the number of induced isomorphic copies of large graphs are approximately same as that of an \erdos graph. In the proof of Theorem  \ref{thm:main} we actually obtain bounds on the rate of convergence to zero of the \abbr{LHS} of \eqref{eq:thm_main_display}. Using these rates,  one can allow $\min \{p, 1-p\}$ and $1-\delta$ to go to zero and obtain modified versions of \eqref{eq:thm_main_display} in those cases. See Remark \ref{rmk:delt-p} for more details. For clarity of presentation we work with fixed values of $p$ and $\delta$  in Theorem \ref{thm:main}. In Section \ref{sec:binary-graph} we give an example of a graph model where the result Theorem \ref{thm:main} is optimal.
	
\end{rmk}

\begin{rmk}[Existence of large motifs]
	\label{rmk:exs-motifs}
	The above result shows that under assumptions {\bf (A1)} and {\bf (A2)}, the associated sequences of graphs are {\bf strongly} pseudo-regular in the sense that the count of large motifs in the graph is approximately the same as in \erdos random graph. A weaker question is asking {\bf existence} of large motifs; to fix ideas let $r = c\log{n}$ for a constant $c> 0$ and ${\sf C}_r$ denote a clique on $r$ vertices. Then one could ask the import of Assumptions {\bf (A1)} and {\bf (A2)} on the {\bf existence} of ${\sf C}_r$; this corresponds to 
	\[\liminf_{n\to\infty} n_{\Graph_n}({\sf C}_r) > 0. \]
	We study these questions in work in progress. 
\end{rmk}

We will see in Section \ref{sec:binary-graph} that the conclusion of Theorem \ref{thm:main} cannot be improved unless one has more assumptions on the graph sequence. Below we consider one such direction to understand the implications of our proof technique if one were to assume bounds on the number of common neighbors of three and four vertices.

\noindent
{\bf Assumption A3.}
$$\max_{v_1\ne v_2\ne v_3 \in [n]}{\left|\left|\sN_{\nu_1} \cap \sN_{\nu_3} \cap\sN_{\nu_3}  \right| - {n}p^3\right|} < Cn^{\delta}.$$

\noindent
{\bf Assumption A4.} 
$$
\max_{v_1\ne v_2\ne v_3 \ne v_4 \in [n]}{\left|\left|\sN_{\nu_1} \cap \sN_{\nu_3} \cap\sN_{\nu_3}   \cap\sN_{\nu_4}\right| - {n}p^4\right|} < Cn^{\delta}. $$

\medskip
Under these above two assumptions we obtain the following improvement of Theorem \ref{thm:main}.

\begin{thm}\label{thm:main_improve}
Let $\{\Graph_n\}_{n \in \N}$ be a sequence of graphs satisfying assumptions {\bf (A1)}-{\bf (A4)}. Then, there exists a positive constant $C_5'$, depending on $p$, such that 
\beq\label{eq:thm_main_display_improve}
\max_{s \le ((1-\delta)\wedge \f{2}{3}) \f{\log n}{\log \gamma_p}- C_5' \log \log n}\max_{{\sf H}_s \in \cG(s)}\left|\f{n_{\Graph_n}({\sf H}_s)}{\frac{(n)_s}{|\mathrm{Aut}({\sf H})|} \left(\f{p}{1-p}\right)^{|E({\sf H})|} (1-p)^{{s \choose 2}}}-1 \right|  \ra 0,
\eeq
as $n \ra \infty$.
\end{thm}

\bigskip

From the proof of Theorem \ref{thm:main_improve} it will be clear that adding more assumptions on the common neighbors of a larger collection of vertices may further improve Theorem \ref{thm:main}. To keep the presentation of current work to manageable length we restrict ourselves only to the above extension. Proof of Theorem \ref{thm:main_improve} can be found in Section \ref{sec:proof_thm}.

 We defer a full discussion of related work to Section \ref{sec:disc}. The rest of this paper is organized as follows:

\subsection*{Outline of the paper.} 
\begin{enumeratei}
\item In Section \ref{sec:aotm} we apply Theorem \ref{thm:main} to {\bf exponential random graph models} (\abbr{ERGM}s) in the high temperature regime (Section \ref{sec:ergm}), {\bf random geometric graphs} (Section \ref{sec:rgg}), {\bf conditioned \erdos}random graphs (Section \ref{sec:erdos-renyi-lc}), and {\bf random regular graphs} (Section \ref{sec:rrg}). 
Section \ref{sec:binary-graph} describes a network model where Theorem \ref{thm:main} is optimal. Proofs of all these results can be found in Section \ref{sec:potaotm}.

\item Section \ref{sec:disc} discusses the relevance of our results, connections to existing literature as well as possible extensions and future directions to pursue. 

\item In Section \ref{sec:proof_outline} we provide an outline of our proof technique. We start this section by stating Proposition \ref{prop:main}, which is the sub-case of Theorem \ref{thm:main} for $p=1/2$. For clarity of presentation, we only explain the idea behind the proof of Proposition \ref{prop:main}. Along with the proof idea we introduce necessary notation and definitions.


\item We prove Proposition \ref{prop:main} in Section \ref{sec:proof_prop}. We state the lemmas needed to prove Proposition \ref{prop:main}. Proofs of these lemmas are deferred to Section \ref{sec:proof_lemma}.

\item In Section \ref{sec:proof_thm} we provide a detailed outline about how one can extend the ideas from the proof of Proposition \ref{prop:main} to prove Theorem \ref{thm:main}. Proof of Theorem \ref{thm:main_improve} can also be found in Section \ref{sec:proof_thm}. Section \ref{sec:potaotm} contains proofs for the applications our main result. 


\end{enumeratei}

\section{Applications of Theorem \ref{thm:main}}
\label{sec:aotm}

In this section we consider four different random graph ensembles: exponential random graph models, random geometric graphs, \erdos random graphs conditioned on a large clique, and random $d$-regular graphs. We show that for these four random graph ensembles Theorem \ref{thm:main} can be applied (and adapted) to show that the number of induced isomorphic slowly growing subgraphs are close to that of an \erdos random graph. In Section \ref{sec:binary-graph} considering an example of a sequence of deterministic graphs we show that one may not extend the conclusion of Theorem \ref{thm:main} without any additional assumptions on the graph sequence.

\subsection{Exponential random graph models (\abbr{ERGM})}
\label{sec:ergm}
    The exponential random graph model (\abbr{ERGM}) is one of the most widely used  models in areas such as sociology and political science since it gives a flexible method for capturing reciprocity (or clustering behavior) observed in most real world social networks. The last few years have seen major breakthroughs in rigorously understanding the behavior of \abbr{ERGM} in the large network $n\to\infty$ limt (see \cite{BBS-11, chaergm} and the references therein). It has been shown that in the {\em high temperature} regime (which we precisely define in Assumption \ref{ass:ergm}), these models converge in the so-called {\em cut-distance}, a notion of distance between graphs established in \cite{gl2}, \cite{gl3}, to the same limit as an \erdos random graph as the number of vertices increase, where the edge connectivity probability of the \erdos random graph is determined explicitly by a function of the parameters (see \cite{chaergm}). In particular this implies that the number of induced isomorphic copies of any subgraph on fixed number of vertices are asymptotically same as in an \erdos random graph in that high temperature regime (see \cite[Theorem 2.6]{gl2}). 
In this section we strengthen the above result to show that  in the high temperature regime the number of induced isomorphic copies of subgraphs of growing size in an \abbr{ERGM} are approximately same with that of an Erd\H{o}s-R\'{e}nyi graph, with appropriate edge connectivity probability. 

We start with a precise formulation of the model. We will stick to the simplest specification of this model which incorporates edges and triangles although we believe that the results below carry over for any model in the ferromagnetic case as long as one is in the high temperature regime. 
Let $\Omega := \set{0,1}^{{n\choose 2}}$ and note that any simple graph on $n$ vertices can be represented by a tuple $\vx := (x_{ij})_{1\leq i< j\leq n}\in \Omega$. Here $x_{ij} = 1$ if vertices $i,j$ are connected by an edge and $x_{ij}= 0$ otherwise.   On this space,  we define the Hamiltonian $H$ as follows.
\begin{equation}
\label{eqn:hamil-def}
	H(\vx) := \beta \sum_{1\leq i < j\leq n} x_{ij} + \frac{\gamma}{n} \sum_{1\leq i < j< k\leq n} x_{ij}x_{jk} x_{ik}.
\end{equation} 
Note that $E(\vx) := \sum_{1\leq i < j\leq n} x_{ij} $ is just the number of edges in $\vx$ whilst $T(\vx) := \sum_{1\leq i < j< k\leq n} x_{ij}x_{jk} x_{ik} $ is the number of triangles. Thus the above Hamiltonian has a simpler interpretation as 
\[H(\vx)=\beta E(\vx) + \frac{\gamma}{n} T(\vx).\]
For simplicity we write $\mvbeta := (\beta,\gamma) $ for the parameters of the model, and  consider the probability measure, 
\begin{equation}
\label{eqn:ergm-measure}
	p_{\mvbeta}^{\sss(n)}(\vx)\propto \exp(H(\vx)), \qquad \vx\in \Omega. 
\end{equation}
In the sequel, to ease notation,  we will suppress $n$ and write the above, often  as $p_{\mvbeta}$. Now for later use we define the function $\varphi_{\mvbeta}:[0,1]\to[0,1]$ via,
\begin{equation}
\label{eqn:phi-beta-def}
	\varphi_{\mvbeta}(x) := \frac{e^{\beta +\gamma x}}{1+e^{\beta +\gamma x}}. 
\end{equation}
Below we state the assumption on the \abbr{ERGM} with which we work in this paper.
\begin{ass}
	\label{ass:ergm}
	We assume that we are in the \emph{ferromagnetic} regime, namely $\gamma > 0$. Further we assume that the parameters $\mvbeta$ are in the high temperature regime (see also \cite{BBS-11}). That is,  there exists a unique solution $0 < p^{*} < 1$ to the equation 
\begin{equation}
\label{eqn:ergm-fixed-point}
	p = \frac{e^{\beta +\gamma p^2}}{1+e^{\beta +\gamma p^2}}
\end{equation} 
and further 
\begin{equation}
\label{eqn:ergm-slope-cond}
	\left.\frac{d}{dp} \frac{e^{\beta +\gamma p^2}}{1+e^{\beta +\gamma p^2}} \right|_{p=p^{*}} < 1. 
\end{equation}
See Figure \ref{fig:phip} below for a graphical description of our setting. In passing we note that the above two conditions can be succinctly expressed using the function $\varphi_{\mvbeta}$ as defined in \eqref{eqn:phi-beta-def} as,
\begin{equation}
\label{eqn:cond-exp-varphi}
	p^* = \varphi_{\mvbeta}((p^{*})^2), \qquad \left.\frac{d}{dp}\varphi_{\mvbeta}(p^2)\right|_{p=p^*} < 1.
\end{equation}	 
\end{ass}
\begin{figure}[htbp]
	\centering
		\includegraphics[scale=.7]{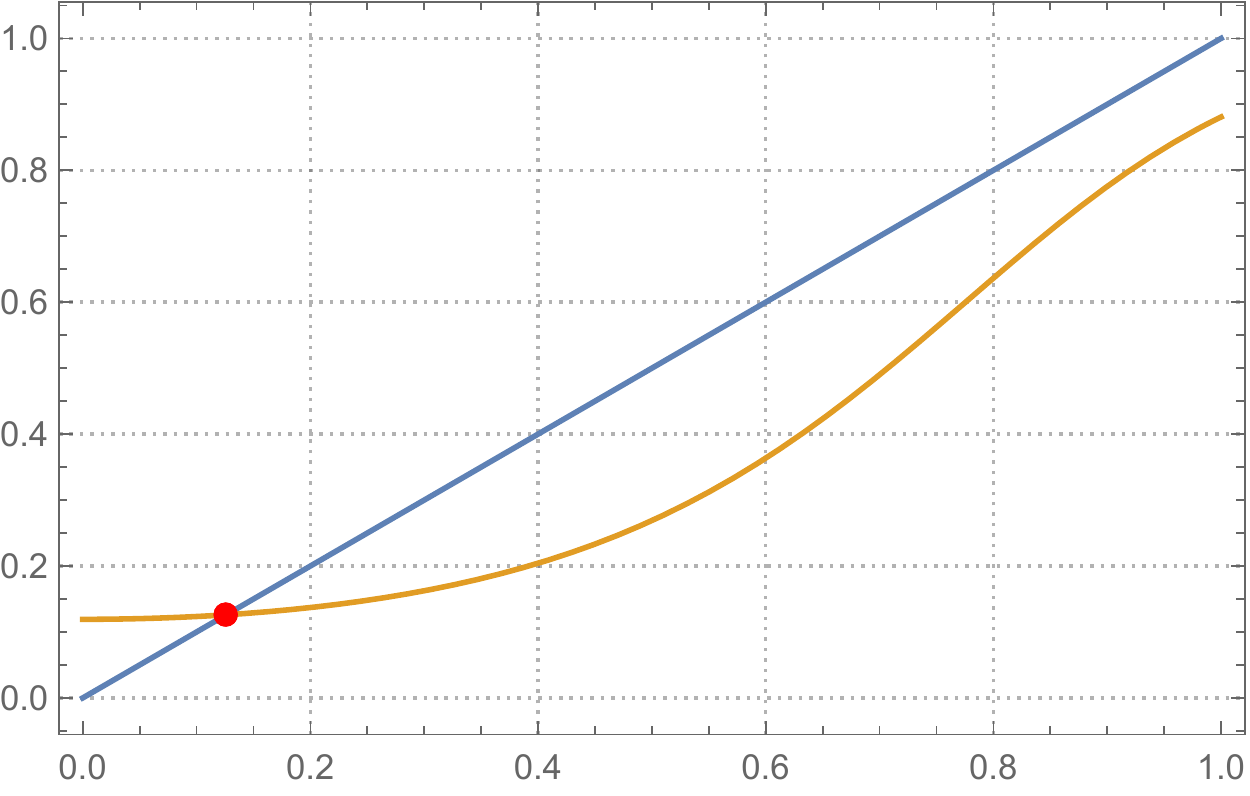}
	\caption{The functions $f(p) =p$ and $g(p) = \frac{e^{\beta +\gamma p^2}}{1+e^{\beta +\gamma p^2}}$ with $\beta = -2$ and $\gamma =4$. The red point corresponds to $p^*$. }
	\label{fig:phip}
\end{figure}
Now we are ready to state the main result regarding \abbr{ERGM}.

\begin{thm}\label{thm:ergm-impli}
Let $\Graph(n, p^*)$ denote an \erdos random graph with edge connectivity probability $p^*$. Let $\bm{X}^{\sss(n)}\sim p_{\mvbeta}^{\sss(n)}$ satisfying assumption \ref{ass:ergm}. Then as $n\to\infty$, we have, 
\[\max_{s \le \f{1}{2}\f{\log n}{ \log \gamma_{p^*}}- C_1' \log \log \hspace{-2pt}n}\max_{{\sf H}_s\in \cG(s)} \left|\frac{n_{\bm{X}^{\sss(n)}}({\sf H}_s)}{\E(n_{\Graph(n, p^*)}({\sf H}_s))} \; \;- \;\; 1  \right| \ra 0, \text{ almost surely},\]
where $C_1'$ is some positive constant depending only on $p^*$.
\end{thm}

\bigskip
Theorem \ref{thm:ergm-impli} follows from Theorem \ref{thm:main} once we establish that assumptions {\bf (A1)}-{\bf (A2)} are satisfied by ${\bm{X}^{\sss(n)}}$ almost surely. We defer the proof to Section \ref{sec:potaotm}.

\subsection{Thresholded graphs from high-dimensional correlation networks}\label{sec:rgg}
 In this section we consider thresholded graphs obtained from high-dimensional random geometric graphs. Roughly in a geometric graph the vertices are points in some metric space and two vertices are connected if they are within a specified threshold distance. These models have drawn wide interests from different branches of science such as machine learning, statistics etc. It is of interest to study these graphs when the dimension of the underlying metric space is growing with the number of vertices. In the context of applications, one major are of applications of this model is neuroscience, while it is impossible to give even a representative sample of related papers, a starting point would be \cite{bullmore2009complex,stam2007graph,eguiluz2005scale} and the references therein. Paraphrasing from \cite[Page 187]{bullmore2009complex} (with parts (c) and (d) most relevant for this paper): 
 
 \begin{quote}
	 \emph{
	``Structural and functional brain networks can be explored using graph theory through the following four steps: ''
	\begin{enumeratea}
		\item Define the network nodes. These could be defined as electroencephalography or multielectrode-array electrodes, or as anatomically defined regions of histological, MRI or diffusion tensor imaging data.
   \item Estimate a continuous measure of association between nodes. This could be the spectral coherence or Granger causality measures between two magnetoencephalography sensors, $\ldots$ or the inter-regional correlations in cortical thickness or volume MRI measurements estimated in groups of subjects.
   \item Generate an association matrix by compiling all pairwise associations between nodes and (usually) apply a threshold to each element of this matrix to produce a binary adjacency matrix or undirected graph.
   \item Calculate the network parameters of interest in this graphical model of a brain network and compare them to the equivalent parameters of a population of random networks.'' 
	\end{enumeratea} 
		} 
 \end{quote}
As a test-bed, we study the simplest possible setting for such questions, first studied rigorously in \cite{devroye2011high} and \cite{dingrgeom16}, and study the behavior of the graph as dimension is growing with the number of vertices.  
To describe the model, we closely follow \cite{devroye2011high}. Write $\sS_{d-1} := \set{\vx \in \R^d: \|\vx\|_2 =1}$ for the unit sphere in $\R^d$, where $\|\cdot\|_2$ denotes the usual Euclidean metric. Suppose $\set{\bm{X}_i: i\in [n]}$ be $n$ points chosen independently and uniformly distributed on $\sS_{d-1}$. Fix $p \in (0,1)$. We will use these points to construct a graph with vertex set $[n]$ as follows: For $i\neq j\in [n]$ say that vertex $i$ is connected to $j$ if and only if 
\[\inp{\bm{X}_i}{\bm{X}_j} \geq t_{p,d}.\]
Here $\inp{\cdot}{\cdot}$ is the usual inner product operation on $\R^d$ and $t_{p,d}$ is a constant chosen such that 
\begin{equation}
\label{eqn:tpd-def}
	\P(\inp{\bm{X}_i}{\bm{X}_j} \geq t_{p,d}) = p
\end{equation}
Call the resulting graph $\Graph(n,d,p)$. See Figure \ref{fig:sphere} for a pictorial representation. 

\begin{figure}[ht!]
	\begin{tikzpicture}[scale=.8,every node/.style={minimum size=1cm}]
	
	\def\R{4} 
	
	\def\angEl{25} 
	\def\angAz{-100} 
	\def\angPhiOne{-50} 
	\def\angPhiTwo{-35} 
	\def\angBeta{30} 
	
	
	\pgfmathsetmacro\H{\R*cos(\angEl)} 
	\LongitudePlane[xzplane]{\angEl}{\angAz}
	\LongitudePlane[pzplane]{\angEl}{\angPhiOne}
	\LongitudePlane[qzplane]{\angEl}{\angPhiTwo}
	\LatitudePlane[equator]{\angEl}{0}
	\fill[ball color=white!10] (0,0) circle (\R); 
	\coordinate (O) at (0,0);
	\coordinate[mark coordinate] (N) at (0,\H);
	\coordinate[mark coordinate] (S) at (0,-\H);
	\path[xzplane] (\R,0) coordinate (XE);
	
	\path[qzplane] (\angBeta:\R) coordinate (XEd);
	\path[pzplane] (\angBeta:\R) coordinate (P);
	\path[pzplane] (\R,0) coordinate (PE);
    \path[pzplane] (\R,0) coordinate (PEd);
	\path[qzplane] (\angBeta:\R) coordinate (Q);
	\path[qzplane] (\angBeta:\R) coordinate (Qd);
	
	\path[qzplane] (\R,0) coordinate (QE);

    \DrawLatitudeCircle[\R]{\angBeta}
    \DrawLatitudeCircle[\R]{0} 
	\node[above=8pt] at (N) {};
	\node[below=8pt] at (S) {};
	
	\draw[-, thick] (N) -- (S);
	\draw[->] (O) -- (P);
	\draw[dashed] (O) -- (QE);
	\draw[-,dashed,black,very thick] (O) -- (PEd);
    \draw[-,dashed,black,very thick] (O) -- (XEd);
    \draw[dashed] (XE) -- (O) -- (PE);
    \draw[-,ultra thick,black] (O) -- (PEd) node[below, right] {$\mathbf{e_1}$};
    \draw[-,ultra thick,black] (O)-- (XEd)node[above, right] {$\mathbf{X_i}$};
    		
	\draw[pzplane,->,thin] (0:0.5*\R) to[bend right=15]
	    node[midway,right] {$\theta$} (\angBeta:0.5*\R);
%
	
%
%
%
\end{tikzpicture}
	\caption{A vertex $i$ is connected $1$ in $\Graph(n,d,p)$ if and only if $\cos(\theta) \geq t_{p,d}$. Here we have rotated the original points via an orthogonal transformation so that the co-ordinates of vertex $1$ correspond to $\mathbf{e_1} = (1,0,\ldots,0)$. Picture modified from template on \url{https://github.com/MartinThoma/LaTeX-examples}.   }
	\label{fig:sphere}
\end{figure}
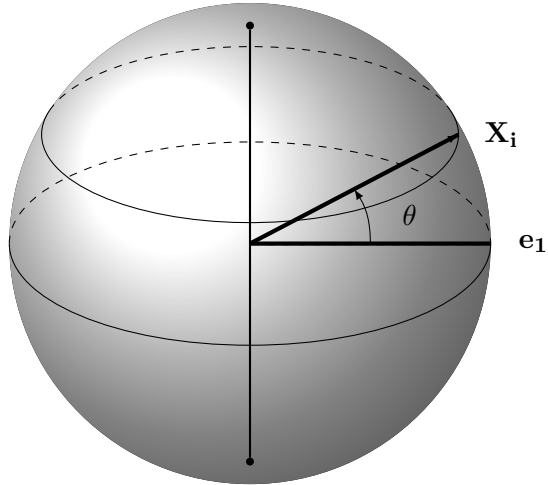
Maximal clique behavior in this model was  studied in \cite{devroye2011high}. Further it was shown in \cite{devroye2011high} that the total variation distance between $\Graph(n,d,p)$ and $\Graph(n, p)$ goes to zero as $d \rightarrow \infty$ while $n$ is fixed. It was further shown in \cite{dingrgeom16}  that the total variation distance between $\Graph(n,d,p)$ and $\Graph(n, p)$ goes to zero if $d$ grows faster than $n^3$ while the distance goes to $1$ if $d/n^3 \rightarrow 0$.  In this section we will consider weaker notions of convergence of graphs and establish that  random geometric graphs look similar to \erdos when $d$ is growing as slow as $(\log n)^c$ for some $c>1$. Below we work with following assumption on random geometric graphs.
\begin{ass}\label{ass:rgeom} We consider random geometric graphs $\Graph(d,n,p)$ with $0 < p\leq 1/2$ and $d/\log{n}\to \infty$ as $n \ra \infty$. 
\end{ass}

We split our result on $\Graph(d,n,p)$ into two parts. In first part we consider induced isomorphic copies of any fixed graph.
\begin{thm}\label{thm:rgg}
If $(\log n)^c/d =o(1)$ for some $c>1$, then as $n\to\infty$ and we have, 
\[ \left|\frac{n_{\Graph(n,d,p)}({\sf H})}{\E(n_{\Graph(n, p)}({\sf H}))} \; \;- \;\; 1  \right| \ra 0,\text{ almost surely},\]
for all finite subgraph ${\sf H}$. 
\end{thm}

Theorem \ref{thm:rgg} shows that the number of induced isomorphic copies of any fixed graph is asymptotically same as that of an \erdos random graph. Therefore, from \cite[Theorem 2.6]{gl2} it follows that $\Graph(d,n,p)$ and $\Graph(n,p)$ are arbitrarily close in cut-metric as soon as $d$ is poly-logarithmic in $n$. Theorem \ref{thm:rgg} is actually a consequence of the following stronger result. This result shows that 
the number of copies of ``large subgraphs'' are asymptotically same as that of \erdos random graph. However, the size of the subgraphs one can consider depends on how fast $d$ ``grows'' compared to  $n$. More precisely, we have the following result.

\begin{thm}\label{thm:rgeom-impli}
Let $\Graph(d,n, p)$ be a random geometric graph on $n$ vertices satisfying Assumption \ref{ass:rgeom}. Then, we have,
\[\max_{s \le {\log \tau_n^{-1}}/{\log \gamma_p}- C_2' \log \log \tau_n^{-1}}\max_{{\sf H}_s\in \cG(s)} \left|\frac{n_{\Graph(n,d,p)}({\sf H}_s)}{\E(n_{\Graph(n, p)}({\sf H}_s))} \; \;- \;\; 1  \right| \ra 0, \text{ almost surely},\]
where
\[
\tau_n :=  \kappa_p\left[\sqrt{\frac{\max\{\log{n},\log{d}\}}{d}} + \frac{1}{n^2}\right],
\]
for some large positive constants $\kappa_p$ and $C_2'$.
\end{thm}

\bigskip
\begin{rmk}
	Note that if $d$ is poly-logarithmic in $n$ then from Theorem \ref{thm:rgeom-impli} it follows that the number of induced isomorphic copies of subgraphs of size upto $O(\log \log n)$ are roughly same as that an \erdos random graph. Proof of Theorem \ref{thm:rgeom-impli} can be found in Section \ref{sec:potaotm}. As described in Remark \ref{rmk:exs-motifs} the question considered in this paper is refined in the sense that we want the number of large motifs in the model under consideration to closely match the corresponding number in an \erdos random graph; this second object grows exponentially in the size of the motif being considered. Thus small errors build up as will be seen in the proofs. In future work we will consider the problem described in  Remark \ref{rmk:exs-motifs} where one is only interested in the existence of large motifs. 
\end{rmk}

\subsection{\erdos random graphs conditioned on large cliques}
\label{sec:erdos-renyi-lc}
Finding maximal cliques in graphs is a fundamental problem in computational complexity. One natural direction of research is average case analysis; more precisely studying this problem in the context of the \erdos random graph. Here it is known that the maximal clique in $\Graph(n, 1/2)$ is $(2+o(1))\log_2 \hspace{-2pt} n$. Simple greedy algorithms have been formulated to find cliques of size $\log_2\hspace{-2pt} n$ in polynomial time but no algorithms are known for finding cliques closer to the maximal size of $(2-\eps)\log_2 \hspace{-2pt} n$ with $\eps$ small, in polynomial time. The situation is different when one places (hides) a clique of large size in the random graph, the so-called \emph{hidden clique} problem especially when the size of the hidden clique is of size $\kappa\sqrt{n}$ for some absolute constant $\kappa$. Here a number  of polynomial algorithms have been formulated to discover the maximal clique. For example  \cite{conjecture} proposed a spectral algorithm that finds a hidden clique of size $\sqrt{cn\log_2 n}$ in polynomial time. In \cite{montalgo} the authors devised an almost linear time non-spectral algorithm to find a clique of size $\sqrt{n/e}$. Dekel, Gurel-Gurevich, and Peres in \cite{pereslargeclique} describe the ``most important" open problem in this area as devising an algorithm (or proving the lack of existence thereof) that finds a hidden clique of size $o(\sqrt{n})$ in polynomial time. Motivated by these questions we investigate the following question,
\begin{quote}
	\emph{How does an \erdos random graph look like when the largest clique is of size ${n^{1/2-\vep}}$? }
\end{quote}
We check that the graph is strongly pseudo random graph even when the largest clique is of size  ${n^{1/2-\vep}}$ for some $\eps>0$.

\begin{thm}\label{lem:large_clique}
Let $\Graph(n,1/2)$ be the Ed\H{o}s-R\'{e}nyi random graph with connectivity probability $\f{1}{2}$. Fix $\delta \in (1/2,1)$, and define
\[
\sA_{\mathrm{d}}^\delta:= \left\{ \exists \,  \nu \in [n]: \left|\left|\sN_\nu\right| - \frac{n}{2}\right| \ge  Cn^{\delta}\right\},
\]
and
\[
\sA_{\mathrm{cod}}^\delta:=\left\{ \exists \,  \nu \ne \nu' \in [n]: \left|\left|\sN_\nu \cap \sN_{\nu'}\right| - \frac{n}{4}\right| \ge Cn^{\delta}\right\}.
\]
Let $\sC_r$ denote the event that the maximal clique size in $\Graph(n,1/2)$ is greater than or equal to $r$. If $r=c n^{\f{1}{2}-\vep}$, for some positive constants $c$, and $\vep$ such that $\delta+\vep>1$, then
\beq
\P\left(\sA_{\mathrm{d}}^\delta \cup \sA_{\mathrm{cod}}^\delta \Big| \sC_r \right) \ra 0,
\eeq
as $n \ra \infty$. In particular by Theorem \ref{thm:main}, the assertion in  \eqref{eq:thm_main_display} holds with high probability as $n\to\infty$ with $p=1/2$. 
\end{thm}

\begin{rmk}
	It would be interesting to further explore the implications of the above result and in particular see if this result negates the existence of polynomial time algorithms for finding hidden cliques in $\Graph(n,1/2)$ if the size of the hidden clique is $o(\sqrt{n})$. We defer this to future work. 
\end{rmk}
\begin{rmk}
For clarity of the presentation of the proof of Theorem \ref{lem:large_clique} we consider only $p=1/2$. We believe it can be extended for any $p \in (0,1)$. 
\end{rmk}

\subsection{Random $d$-regular graphs}\label{sec:rrg} 
For positive integers $n$ and $d :=d(n)$,  the random $d$-regular graph $\Graph(n,d)$ is a random graph which is chosen uniformly among all regular graphs on $n$ vertices of degree $d$. One key difference between $\Graph(n,d)$ and $\Graph(n,p)$ is that the edges in $\Graph(n,d)$ are not independent of each other.  However, using different techniques researchers have been able to study different properties of $\Graph(n,d)$. For more details, we refer the reader to \cite{bollobas, JLR}. Here we study the number of induced isomorphic copies of large subgraph in $\Graph(n,d)$. We obtain the following result. Before stating the result let us recall $a_n= \Theta(b_n)$ implies that theres exists constant $c_1$ and $c_2$ such that $c_1 a_n \le b_n \le c_2 a_n$, for all large values of $n$.

\begin{thm}\label{thm:rrd}
Let $\Graph(n,d)$ be a random regular graph with $d=\Theta(n)$. Then, there exists a positive constant such that
\beq\label{eq:thm_rrd_display}
\max_{s \le  \f{1}{2} \f{\log n}{\log \gamma_{d/n}}- C_3' \log \log n}\max_{{\sf H}_s \in \cG(s)}\left|\f{n_{\Graph(n,d)}({\sf H}_s)}{\frac{(n)_s}{|\mathrm{Aut}({\sf H}_s)|} \left(\f{d/n}{1-(d/n)}\right)^{|E({\sf H}_s)|} (1-\f{d}{n})^{{s \choose 2}}}-1 \right|  \ra 0, 
\eeq
as $n \ra \infty$, in probability.
\end{thm}

\bigskip
The proof of Theorem \ref{thm:rrd} is a direct application of Theorem \ref{thm:main} once we establish that $\Graph(n,d)$ satisfies conditions {\bf (A1)}-{\bf (A2)}. However, this is already known in the literature.
\begin{thm}[{\cite[Theorem 2.1(i)]{KSVW}}]
\label{thm:co-deg_rrg}
Let $\Graph(n,d)$ be a random regular graph with $d=\Theta(n)$. Then, 
\[
\P\left(\max_{v \ne v' \in [n]} {\left|\left|\sN_\nu \cap \sN_{\nu'}\right| - \f{d^2}{n}\right|} \le \f{6 d\sqrt{\log n}}{\sqrt{n}}\right) \ra 1, \text{ as } n \ra \infty.
\]
\end{thm}

\bigskip
Proof of Theorem \ref{thm:rrd} is immediate from Theorem \ref{thm:co-deg_rrg}.

\begin{rmk}
Note that  Theorem \ref{thm:co-deg_rrg} implies that {\bf (A1)}-{\bf (A2)} (assumption {\bf (A1)} is automatically satisfied as the graphs are $d$-regular) are satisfied with $\delta:=\delta_n= \f{1}{2}+\f{1}{2}\f{\log \log n}{\log n}$. One can absorb the quantity $\f{1}{2}\f{\log \log n}{\log n}$ in the constant $C_0'$ appearing in Theorem \ref{thm:main}, and can essentially assume that $\Graph(n,d)$ satisfies assumptions {\bf (A1)}-{\bf (A2)} with $\delta=\f{1}{2}$. 
\end{rmk}

\begin{rmk}\label{rmk:new_preprint_random_d}
Very recently, during the final stages of writing of this paper, Konstantin and Youssef \cite{KY_rrg} improved a previously existing bound on the second largest absolute value of the adjacency matrix of a random $d$-regular graph and showed that it is $O(\sqrt{d})$ with probability tending to one. Hence, using \cite{KY_rrg} one can re-derive Theorem \ref{thm:rrd} from \cite[Theorem 4.10]{prg}. 
\end{rmk}


\subsection{Optimality of Theorem \ref{thm:main}: Binary graph}
\label{sec:binary-graph}

Now as mentioned above we provide the example which shows that the result Theorem \ref{thm:main} is optimal. This corresponds to a graph on a vector space over the binary field. For any odd integer $k$ this graph $\Graph_n^b$  is a graph on $n$ vertices where $n=2^{k-1}-1$ and ``$b$'' is to emphasize ``binary'' which we now elucidate. The vertices of $\Graph_n^b$ are binary tuples of length $k$ with odd number of ones except the vector of all ones. Note that every vertex $v$ has also a vector representation. We draw an edge between two vertices $u$ and $v$ if and only if $\langle u , v \rangle =1$, where the multiplication and the addition are the corresponding operations in the binary field. Now we are ready to state the main result on $\Graph_n^b$.

\begin{thm}\label{thm:binary_graph}
Let $\Graph_n^b$ be the graph described above. It has the following properties:

\begin{enumeratei}
\item There exists an absolute constant $C_4'$ such that
\[\max_{s \le \f{1}{2}{\log_2\hspace{-2pt} n}- C_4' \log \log \hspace{-2pt}n}\max_{{\sf H}_s\in \cG(s)} \left|\frac{n_{\Graph_n^b}({\sf H}_s)}{\E(n_{\Graph(n, \f{1}{2})}({\sf H}_s))} \; \;- \;\; 1  \right| \ra 0,\]
as $n \ra \infty$. 

\item Let ${\sf I}_s$ denote the empty graph with vertex set $[s]$. Then for any $\eta >0$,
\[\max_{s \le (1-\eta){\log_2\hspace{-2pt} n}} \left|\frac{n_{\Graph_n^b}({\sf I}_s)}{\E(n_{\Graph(n, \f{1}{2})}({\sf I}_s))} \; \;- \;\; 1  \right| \ra 0,\]
as $n \ra \infty$. Moreover the size of the largest independent set in $\Graph_n^b$ is $\log_2(n+1) +1$. 

\item Recall that ${\sf C}_s$ is the complete graph with vertex set $[s]$. Then for any $\eta>0$, setting $\bar{s}:=(\f{1}{2}+\eta) \log_2\hspace{-2pt}n$ we have
\[
 \f{n_{\Graph_n^b}({\sf C}_{\bar{s}})}{\E(n_{\Graph(n, \f{1}{2})}({\sf C}_{\bar{s}}))} \ge \exp(-2/3) 2^{{\eta \log_2 \hspace{-2pt}n \choose 2}},
\]
for all large $n$.
\end{enumeratei}
\end{thm}

Proof of Theorem \ref{thm:binary_graph}(i) is a direct application of Theorem \ref{thm:main}. We only need to verify that assumptions {\bf (A1)}-{\bf (A2)} are satisfied by the graph sequence $\{\Graph_n^b\}$ with $p=1/2$ and $\delta \le 1/2$. Theorem \ref{thm:binary_graph}(ii)-(iii) have deeper implications in the context of pseudo-random graphs. Theorem \ref{thm:binary_graph}(ii) shows that if in \eqref{eq:thm_main_display} we restrict ourselves to some subsets of $\cG(r)$ then  one may be able to improve the conclusion of Theorem \ref{thm:main}. Whereas, Theorem \ref{thm:binary_graph}(iii) shows that one cannot expect any such improvement to hold uniformly for all subgraphs in $\cG(r)$ beyond the $\f{1}{2}\log_2\hspace{-2pt}n$ barrier, as predicted by Theorem \ref{thm:main}. Therefore, Theorem \ref{thm:main} gives us an optimal result under the assumptions {\bf (A1)}-{\bf (A2)}. Theorem \ref{thm:binary_graph}(ii) further shows that even if we restrict our attention to some subset of $\cG(r)$ any improvement must not go beyond $r=\log_2\hspace{-2pt}n$.

\section{Discussion and related results}
\label{sec:disc}
In this Section we provide a wide-ranging discussion both of related work as well as future directions.  We begin our discussion with Thomason's Jumbled graphs.

\subsection{Jumbled graphs}
In the context of pseudo-random graphs results similar to \eqref{eq:thm_main_display} first appeared in the work of Thomason \cite{thom2}. To explain his result, let us define his notion of {Jumbled graph}. 
\begin{dfn}\label{dfn:jumbled}
A graph $\Graph_n$ is called $(p,\alpha)$-jumbled if for every subset of vertices $U \subset [n]$, one has
\beq\label{eq:jumbled}
\left| e(U) - p {|U| \choose 2}\right| \le \alpha |U|,
\eeq 
where $e(U)$ denotes the number of edges whose both end points are in $U$.
\end{dfn}

\medskip

In \cite[Theorem 2.10]{thom2} it was shown that for $(p,\alpha)$-jumbled graphs, with $p \le 1/2$, \eqref{eq:thm_main_display} holds whenever $s$ is such that $p^s n \gg 42 \alpha s^2$. In \cite{ES} it was established that $\alpha$ must be $\Omega(\sqrt{np})$ (recall $a_n=\Omega(b_n)$ implies that $\liminf_n a_n/b_n \ge c >0$). When $\alpha=\Theta(\sqrt{np})$, then we must have $p^s \sqrt{n} \gg 1$ for \eqref{eq:thm_main_display} to hold for $(p,\alpha)$-jumbled graphs. Note that, when $\delta, p \le 1/2$, under the same assumption on $s$ we establish Theorem \ref{thm:main}.
 
 Pseudo-random graphs satisfying assumptions {\bf (A1)}-{\bf (A2)} and jumbled graphs are not very far from each other. For example, if $\Graph_n$ satisfies {\bf (A1)}-{\bf (A2)}, then one can check that $\Graph_n$ is a $(p, O(n^{({1+\delta})/{2}}))$-jumbled graph (see \cite[Theorem 1.1]{thom2}). On the other hand, if $\Graph_n$ is a $(p, n^{\delta'})$-jumbled graph then all but few vertices satisfy assumption {\bf (A1)} for any $\delta >\delta'$ (see \cite[Lemma 2.1]{thom2}). So our assumptions {\bf (A1)}-{\bf (A2)} and the condition \eqref{eq:jumbled} are somewhat comparable. 
 
The advantage with {\bf (A1)}-{\bf (A2)} is that given any graph one can easily check those two conditions. On other hand, given any large graph, it is almost impossible to check the condition \eqref{eq:jumbled} for all $U \subset [n]$. For example, condition \eqref{eq:jumbled} would be very hard to establish for \abbr{ERGM}s, {random geometric graphs}, \erdos random graphs conditioned on the existence of a large clique, and even random $d$-regular graphs. We have already seen in Section \ref{sec:aotm} that for these graphs one can use Theorem \ref{thm:main} and deduce that the number of slowly growing subgraphs in those graphs are approximately same as that of an \erdos random graph with an appropriate edge connectivity probability.
 
 As mentioned above, if {\bf (A1)}-{\bf (A2)} are satisfied by graph $\Graph_n$, then it is also jumbled graph with certain parameters. Therefore. one can try to apply \cite[Theorem 2.10]{thom2} to obtain a result similar to ours. However, the application of \cite[Theorem 2.10]{thom2} yields a sub-optimal result in the following sense. Note that both random $d$-regular graphs and exponential random graph models satisfy assumptions {\bf (A1)}-{\bf (A2)} with $\delta=1/2$ (see Theorem \ref{thm:co-deg_rrg} and Theorem \ref{thm:ergm-cond}). Therefore, applying Theorem \ref{thm:main} one would obtain that the number of induced isomorphic copies of subgraphs of size $\f{1}{2}\log_2\hspace{-2pt}n$ (for ease of explanation, let us consider only the case $d/n=p^*=1/2$ in Theorem \ref{thm:rrd} and Theorem \ref{thm:ergm-impli}) are asymptotically same as that of an $\Graph(n,\f{1}{2})$. However, application of \cite[Theorem 1.1]{thom2} and \cite[Theorem 2.10]{thom2} implies that the same holds for $s \le \f{1}{4}\log_2\hspace{-2pt}n$. Therefore Theorem \ref{thm:main} improves the result of \cite{thom2}.
 
 We now turn our attention to the class of pseudo-random graphs defined by Chung, Graham, and Wilson \cite{cgw}.

\subsection{Quasi-random graphs}
One of our main motivations for this work was the notion introduced by Chung, Graham, and Wilson \cite{cgw}. We state the main theorem in their paper that relates various graph parameters and any of them can be taken as a definition of pseudo-randomness. Before stating that result recall that for a sequence of reals $\{a_n\}$ one writes $a_n=o(1)$ to imply that $a_n \ra 0$ as $n \ra \infty$. We only state parts of the main theorem of \cite{cgw} that are relevant in the context of our result. 
\begin{thm}[{\cite{cgw}}]
	\label{prgcgw}
Let $p \in(0,1)$ be fixed. Then for any graph sequence $\Graph_n$, the following are equivalent: 
\begin{itemize}

\item[$P_1(r)$:] For any  ${\sf H} \in \cG(r)$, where $r$ is any arbitrary fixed positive integer  
\begin{equation}\label{subgraphcopies}
n_{\Graph_n}^*({\sf H}) = ({1+o(1)}) n^r \left(\f{p}{1-p}\right)^{|E({\sf H})|}(1-p)^{{r \choose 2}},
\end{equation} 
where $n_{\Graph_n}^*({\sf H})$ denotes the number of labelled occurrences of the subgraph ${\sf H}$ in $\Graph_n$.

\item[$P_4$:]
For each subset of vertices $U \subset [n]$
\begin{equation}\label{edgedistnprg}
 e(U)=\f{p}{2}|U|^2+o(n^2).
 \end{equation}

\item[$P_7$:] 
\begin{equation}\label{avgcodegree}
\sum_{\nu\ne  \nu'=1}^n{\left|\left|\sN_\nu \cap \sN_{\nu'}\right| - p^2{n}\right|} = o(n^3).
\end{equation}

\end{itemize}
\end{thm}

\bigskip
Although \cite[Theorem 1]{cgw} was proved only for $p=1/2$ one can easily adapt their technieues to easily extend the result for any $p \in (0,1)$.  In  this paper one of our main aims is to understand the role of \emph{concentration inequalities} and rates of convergence in the above conditions. That is, if we specify the rates of convergence in \eqref{edgedistnprg} and/or in \eqref{avgcodegree}, then a natural question is whether one can accommodate a slowly growing $r$ in \eqref{subgraphcopies}. Motivated by this observation we began this work. Note that our assumptions {\bf (A1)}-{\bf (A2)} are similar to $P_7$ above. 

Investigating the proof of Theorem \ref{thm:main} one can easily convince oneself that we do not need the maximum over all vertices (and pair of vertices) in {\bf (A1)}-{\bf (A2)}. The proof goes through as long as the assertions {\bf (A1)}-{\bf (A2)} hold on the average. Indeed, using  Markov's inequality, and increasing $\delta$ by some multiple of $\f{\log n}{\log \log n}$, one can establish that {\bf (A1)}-{\bf (A2)} holds for all but few vertices. This does not change the conclusion Theorem \ref{thm:main} and it only increases the constant $C_0'$. For clarity of presentation we do not pursue this generalization and work with the assumption that {\bf (A1)}-{\bf (A2)} holds for all vertices (and pair of vertices).
 
In the context of quasi-random graphs attempts have been made to prove that \eqref{subgraphcopies} holds even one allows $r$ to go to infinity. In \cite{rootlogn} Chung and Graham showed that if a subgraph ${\sf H}_s$ is not an induced subgraph of $\Graph_n$ then one must have a set $S$ of size $n/2$ such that $e(S)$ deviates from $\f{n^2}{16}$ by at least $2^{-2(s^2+ 27)n^2}$. This implies that in a quasi-random graph $\Graph_n$ at least one induced copy of every subgraph of size $O(\sqrt{\log n})$ must be present. 

One may also want to check upon plugging in the rates of convergence from {\bf (A1)}-{\bf (A2)} and following the steps in the proof of \cite[Theorem 1]{cgw} whether it is possible to show  \eqref{subgraphcopies} for subgraphs of growing size. It can indeed be done. However, the arguments break down at $r=O(\sqrt{ \log n})$. Note that this is much weaker than Theorem \ref{thm:main}. To understand why the argument stops at $O(\sqrt{\log n})$ we need discuss the key idea behind the proof of \cite[Theorem 1]{cgw} which is deferred to Section \ref{sec:proof_outline}.

Next we direct our attention to Alon's $(n,d,\lambda)$ graph.

\subsection{Alon's $(n,d,\lambda)$ graphs} 
A graph is called an $(n,d,\lambda)$-graph if it is a $d$-regular graph on $n$ vertices whose second largest absolute eigenvalue is $\lambda$. A number of properties of such graphs have been studied (see \cite{prg} and the references therein). For such graphs the number of induced isomorphic copies of large subgraphs is also well understood (see \cite[Section 4.4]{prg}). From \cite[Theorem 4.1]{prg} it follows that $(n,d,\lambda)$-graphs contain the correct number of cliques (and independent sets) of size $s$ if $s$ satisfies the condition $n \gg \lambda (n/d)^s$. Thus, to apply \cite[Theorem 4.1]{prg} one needs a good bound on the second largest absolute eigenvalue $\lambda$. For nice graph ensembles it is believed that one should have $\lambda=\Theta(\sqrt{d})$. However, this has been established only in few examples. For example, when $d=\Theta(n)$ for a random $d$-regular 
this has been established very recently in \cite{KY_rrg}. 
Any bound on $\lambda$ of the form $O(d^{\f{1}{2}+\eta})$, for some $\eta >0$, yield a sub-optimal  result when applied to \cite[Theorem 4.10]{prg}. 
On the other hand, our method being a non-spectral technique, does not require any bound on $\lambda$ and the conditions {\bf (A1)}-{\bf (A2)} are much easier to check.

\subsection{Optimality, limitations and future directions}

We have already seen that Theorem \ref{thm:main} is optimal in the sense that one cannot improve the conclusion of Theorem \ref{thm:main} without adding further assumptions. 
One can consider two possible directions for improvements. The first of those is to have more conditions on the graph sequence. For example, one may assume controls on the number of common neighbors of three and four vertices. This indeed improves the conclusion of Theorem \ref{thm:main} when $\delta \le 1/2$. This is seen in Theorem \ref{thm:main_improve}.  

Another direction we are currently pursuing is extending Theorem \ref{thm:main} to the setting of models which incorporate \emph{community structure}. Here the comparative model is not the \erdos random graph but the {\em stochastic block model}.  Community  detection and clustering on networks have spawned an enormous literature over the last decade, see the survey \cite{fortunato2010community} and the references therein. One can easily extend the ideas of Theorem \ref{thm:main} to count the number of induced isomorphic copies of large subgraphs in a stochastic block model. To keep the clarity of the presentation of the current work we defer this generalization to a separate paper.

\section{Notation and Proof Outline}
\label{sec:proof_outline}
In this section our goal is to discuss the idea behind the proof of our main result. For clarity of presentation we consider the sub-case $p=1/2$ separately. That is we establish the following result first. 
\begin{prop}\label{prop:main}
Let $\{\Graph_n\}_{n \in \N}$ be a sequence of graphs satisfying the following two conditions:
\beq
\sup_{1 \leq \nu \leq n}{\left|\left|\sN_\nu\right| - \f{n}{2}\right|} < Cn^{\delta}, \label{eq:A1_1/2}
\eeq
\beq
\sup_{1\leq \nu \ne \nu' \leq n}{\left|\left|\sN_\nu \cap \sN_{\nu'}\right| - \f{n}{4}\right|} < Cn^{\delta}. \label{eq:A2_1/2}
\eeq
There exists a positive constant $C_0$, such that
\beq
\max_{s \le ((1-\delta)\wedge \f{1}{2}) \log_{2} \hspace{-2pt}n - C_0' \log_2 \log_2 \hspace{-2pt}n}\max_{{\sf H}_s \in \cG(s)}\left|\f{n_\Graph({\sf H}_s)}{\frac{(n)_s}{|\mathrm{Aut}({\sf H}_s)|} \left(\f{1}{2}\right)^{{s \choose 2}}}-1 \right|  \ra 0,
\eeq
as $n \ra \infty$.
\end{prop}

\medskip
Once we prove Proposition \ref{prop:main} the proof of Theorem \ref{thm:main} is an easy adaptation. Proof of Proposition can be found in Section \ref{sec:proof_prop} and the proof of Theorem \ref{thm:main} is deferred to Section \ref{sec:proof_thm}. This section is devoted to outline the ideas of the proof of Theorem \ref{thm:main}. For clarity once again we focus on the sub-case $p=1/2$. 

For any graph $\sf{H}_r$ on $r$ vertices, let us define
\beq
\cE_n({\sf{H}}_r):=\left|\f{n_\Graph({\sf H}_r)}{\f{(n)_r}{|\mathrm{Aut}(\sf{H_r})|}} - \left(\f{1}{2}\right)^{{r \choose 2}}\right|,
\eeq
and let $\cE_n(r):=\max_{{\sf H}_r \in \cG(r)}\cE_n({\sf H}_r)$.

Note that, in order to prove Proposition \ref{prop:main}, it is enough to show that $2^{{r \choose 2}}\cE_n(r)\ra 0$ as $n \ra \infty$, uniformly for all graphs $\sf{H}_r \in \cG(r)$ and all $r \le ((1-\delta) \wedge \f{1}{2}) \log_2\hspace{-2pt}n - C_0' \log_2 \log_2\hspace{-2pt}n$, where $C_0'$ is some finite positive constant. If $\Graph_n$ satisfies assumptions \eqref{eq:A1_1/2}-\eqref{eq:A2_1/2}, then one can easily check that the required result holds for $r=2$. So, one natural approach to extend the same for larger values of $r$ is to use an induction-type argument. The key to such an argument is a recursive estimate of the errors $\cE_n(r)$ (see Lemma \ref{lem:error_inductive_bd}). 

Therefore the main idea behind the proof of Proposition \ref{prop:main} lies in obtaining a bound on $\cE_n(r+1)$ in terms of $\cE_n(r)$. To explain the idea more elaborately let us first introduce few notation. Our first definition will be a notion about generalized neighborhood.

\begin{dfn}[Generalized neighborhood]\label{nbrdfn}
Let $(a_{\nu \nu'})_{\nu,\nu'=1}^n$ denote the adjacency matrix of the graph $\Graph_n$. 
\begin{enumeratea}
\item For any $\xi \in \{0,1\}$, and $v \in [n]$ define 
\[
\sN_v^\xi:=\left\{v' \in [n]: a_{vv'}=\xi\right\}.
\]
That is, for $\xi=1$ and $0$ the $\sN_v^\xi$ denote the collection of neighbors and non-neighbors of $v$ respectively. By a slight abuse of notation $\sN_v^\xi$ will also denote cardinality of the set in context. For later use, we also denote $a_{uv}^\xi:= \bI(a_{uv}=\xi)$.
\item For any set of vertices $B \subset [n]$ we denote $\sN_\nu^{B,\xi} := \sN_\nu^\xi \cap B$. That is, $\sN_\nu^{B,\xi}$ denotes the collection of neighbors (non-neighbors, if $\xi=0$) of $\nu$ in $B$. Further, for ease of writing we will continue to use the notation $\sN_\nu^{B,\xi}$ for the cardinality of the set in context. 	
\end{enumeratea} 
  
\end{dfn}

The parameter $\xi$ in Definition \ref{nbrdfn} is required to consider all subgraphs of any given size. For example, if we are interested in finding the number of cliques in $\Graph_n$ then it is enough to consider $\xi=1$, whereas for independent sets we will need $\xi=0$.

With the help of this definition let us now again continue explaining the idea of the proof.  For ease of explanation let us assume that we are interested in proving  the main result only for cliques. Suppose we already have the desired conclusion for $\cE_n({\sf{C}_k})$. Note that a collection of $(k+1)$ vertices forms a clique of size $(k+1)$ if and only if any $k$ of them form a clique of size $k$, and the $(k+1)$-th vertex is the common neighbor of the first $k$ vertices. We use this simple idea to propagate our estimate in the error $\cE_n(\cdot)$.

If we have an Erd\H{o}s-R\'{e}yni random graph with edge connectivity probability $1/2$, then we see that given any $r$ vertices the expected number of common neighbors of those $r$ vertices is $n/2^r$. This implies that the number of cliques of size $(r+1)$ is approximately same with the number of the cliques of size of $r$ multiplied by $n/2^r$. 

To formalize this idea for graphs satisfying \eqref{eq:A1_1/2}-\eqref{eq:A2_1/2} we need to show that for any collection of vertices $v_1,v_2,\ldots,v_r\in [n]$ and $\xi_1,\xi_2,\ldots,\xi_r \in \{0,1\}^r$ we must have that $|\sN_{v_1}^{\xi_1} \cap \sN_{v_2}^{\xi_2} \cap \cdots \sN_{v_r}^{\xi_r} | \approx \f{n}{2^r}$. For ease of writing we have the next definition.

\begin{dfn}\label{dfn:f_r_bar_xi}
For $\ul{\nu}:=\{\nu_1,\nu_2,\ldots,\nu_r\}$ and $\ul{\xi}:=\{\xi_1,\xi_2,\ldots,\xi_r\} \in \{0,1\}^r$, let us denote
\beq
{f_r(\underline{\nu},\ul{\xi})} := {\left|\sN_{\nu_1}^{\xi_1} \cap \sN_{\nu_2}^{\xi_2} \cap \cdots \cap \sN_{\nu_r}^{\xi_r}\right|}.\notag
\eeq
Further denote 
\begin{equation}\label{avgno}
\bar{f}_{r,\ul{\xi}}:= \f{1}{(n)_r} \sum_{\ul{\nu}} f_r(\ul{\nu},\ul{\xi}).
\end{equation}
\end{dfn}

\medskip

Equipped with the above definition, as noted earlier, our goal is to show that $f_r(\ul{v},\ul{\xi}) \approx \f{n}{2^r}$ for graphs satisfying \eqref{eq:A1_1/2}-\eqref{eq:A2_1/2}. However, such an estimate is not feasible for all choices of vertices $v_1,v_2,\ldots, v_r$. Instead we show that such an estimate can be obtained for most of the vertices which we call ``good'' vertices, whereas the vertices for which the above does not hold are termed as ``bad'' vertices. More precisely we have the following definition.
\begin{dfn}\label{dfn:good_dfn}
Fix $\vep >0$. For any given set $B \subset [n]$, and $\xi \in \{0,1\}$, define
\[
\mathrm{Good}^\xi_\vep(B):= \left\{ \nu \in [n]: \left|\sN_\nu^{B,\xi} - \frac{|B|}{2}\right| \le \wt{C} |B| n^{\vep(\delta-1)/2}\right\}.
\]
Throughout this paper, we fix an \begin{equation}
\label{eqn:vep-def}
	\vep := \frac{\bar{C}_0 \log_2 \log_2\hspace{-2pt}n}{(1-\delta)\log_2\hspace{-2pt}n},
\end{equation} 
for some appropriately chosen large constant $C_0$, which we determine later during the proof. Since we work with this fixed choice of $\vep$ through out the paper (except in the proof of Theorem \ref{thm:rgeom-impli}), for brevity of notation we drop the subscript $\vep$ from in the definition of $\mathrm{Good}^\xi_\vep(\cdot)$ and write $\mathrm{Good}^\xi(\cdot)$ instead. Next for a given collection of $m$ vertices $\{v_1,v_2,\ldots, v_m\}$ and $\ul{\xi}:=\{\xi_1,\xi_2,\ldots,\xi_m,\xi_{m+1}\} \in \{0,1\}^{m+1}$, we set 
$$\mathrm{Good}^{\ul{\xi}}(v_1,v_2,\ldots,v_m):= \mathrm{Good}^{\xi_{m+1}}(B), \qquad \text{ where }B= \sN_{\nu_1}^{\xi_1} \cap \sN_{\nu_2}^{\xi_2} \cap \cdots \cap \sN_{\nu_m}^{\xi_m}.$$ 
Moreover, letting $\ul{\wt{\xi}}:=\{\xi_1,\xi_2,\ldots,\xi_m\}$, define 
$$\mathrm{Bad}^{\ul{\wt{\xi}}}(v_1,v_2,\ldots,v_m):=\cup_{\xi_{m+1}=0}^1 \left(\mathrm{Good}^{\ul{\xi}}(v_1,v_2,\ldots,v_m)\right)^c.$$
\end{dfn} 

\medskip

Equipped with above definition it is easy to show by induction that for any $\ul{v} \in [n]^r$ and $\ul{\xi} \in \{0,1\}^r$ such that $\nu_j \in \Good^{\ul{\xi}^{j}}(\nu_1, \nu_2,...,\nu_{j-1})$,  for all $j =3,4,\ldots,r$, we have $f_r(\underline{\nu},\ul{\xi}) \approx {n}/{2^r}$ (see Lemma \ref{common-nbr-allgood}).  However, our notion of ``good'' vertices is useful only when the cardinality of collection of ``bad'' vertices is small compared to that of ``good'' vertices. To show this we work with the following definition.

\begin{dfn}\label{dfn:bad_define}
Fix two positive integers $m'< m$. Let ${\sf H}_m$ be a graph on $m$ vertices, and ${\sf H}_{m'}$ be any one of the sub-graphs of ${\sf H}_m$ induced by $m'$ vertices. Define $\cH_m$ to be the (unordered) collection of vertices $\ul{v} :=\{v_1,v_2,\ldots,v_m\} \in [n]^m$ such that the sub-graph induced by them is isomorphic to ${\sf H}_m$. Similarly define $\cH_{m'}$. Further given any $\ul{v} \in [n]^m$, denote $\ul{v}^{m'}:=\{v_1,v_2,\ldots,v_{m'}\}$, and for any  $\ul{\xi}=\{\xi_1,\xi_2,\ldots,\xi_{m'}\} \in \{0,1\}^{m'}$, denote $\ul{\xi}^j:=\{\xi_1,\xi_2,\ldots,\xi_j\}$, for $j=1,2,\ldots,m'$. Then define
\begin{align*}
\Bad^{\ul{\xi}}({\sf H}_{m'}, {\sf H}_m)&:=\Big\{ \ul{v} \in {\cH}_m: \exists \text{ a relabeling }  \widehat{\ul{v}}=\{\hat{v}_1,\hat{v}_2\ldots,\hat{v}_m\} \text{ such that }\\
& \qquad \qquad \widehat{\ul{v}}^{m'}\in \cH_{m'},  \hat{v}_k \in \mathrm{Bad}^{\ul{\xi}}(\hat{v}_1,\hat{v}_2,\ldots,\hat{v}_{m'}), \text{ for } k=m'+1,\ldots,m,\\
& \qquad \qquad \qquad\text{ and }\hat{v}_j \in \mathrm{Good}^{\ul{\xi}^j}(\hat{v}_1,\hat{v}_2,\ldots,\hat{v}_{j-1}), \text{ for } j=3,4,\ldots,m'\Big\}. 
\end{align*}
If there are more than one relabeling of $\ul{v}$ such that the conditions of $\Bad^{\ul{\xi}}({\sf H}_{m'}, {\sf H}_m)$ are satisfied choose one of them arbitrarily. When ${\sf H}_{m'}$ and ${\sf H}_m$ are clear from the context, for brevity we simply write $\Bad^{\ul{\xi}}_{m',m}$.

\end{dfn}

\bigskip
We noted above that to show that $f_r(\underline{\nu},\ul{\xi}) \approx {n}/{2^r}$  one needs $\nu_j \in \Good^{\ul{\xi}^{j}}(\nu_1, \nu_2,...,\nu_{j-1})$,  for all $j =3,4,\ldots,r$. Thus if we have a collection of vertices $v_1,v_2,\ldots,v_r$ such that  $\nu_3 \notin \Good^{\ul{\xi}^{2}}(\nu_1, \nu_2)$ and $\nu_j \in \Good^{\ul{\xi}^{j}}(\nu_1, \nu_2,...,\nu_{j-1})$,  for all $j =4,5,\ldots,r$, then we cannot carry out our argument. To prevent such cases we consider all possible relabeling of  $\ul{v}$ in Definition \ref{dfn:bad_define}. Since we only count the number of isomorphic copies of subgraphs the relabeling of Definition \ref{dfn:bad_define} does not cause any harm.

Coming back to the proof of main result we note that we need to establish that the cardinality of $\cup_{j=1}^r \Bad^{\ul{\xi}}_{j,r}$ is small. To prove this we start by bounding $\sN_v^{B,\xi}$ for every $v \in [n]$, $B \subset [n]$, and $\xi \in \{0,1\}$. The key to the latter is a bound on the variance of $\sN_v^{B,\xi}$, which can be obtained using \eqref{eq:A1_1/2}-\eqref{eq:A2_1/2}. Stitching these arguments together show that we then have the desired bound for $f_r(\ul{v},\ul{\xi})$ for most of the vertices.

Recall from Proposition \ref{prop:main} that this argument stops at $r\approx ((1-\delta) \wedge \f{1}{2}) \log_2\hspace{-2pt}n$. This is due to the fact that for such values of $r$ the cardinality of ``good'' vertices become comparable with that of ``bad'' vertices. On the set of ``bad'' vertices one does not have a good estimate on $f_r(\ul{v},\ul{\xi})$. Therefore, one cannot proceed with this scheme beyond $((1-\delta) \wedge \f{1}{2}) \log_2\hspace{-2pt}n$. The rest of the argument involves finding good bounds on $\bar{f}_{r,\ul{\xi}}$ and an application of Cauchy-Schwarz inequality which relate $\cE_n(r+1)$ with $\cE_n(r)$. The required bound on $\bar{f}_{r,\ul{\xi}}$ is easy to obtain using \eqref{eq:A1_1/2}-\eqref{eq:A2_1/2}.  To complete the proof one also requires a good bound on  $(\sum_{\ul{\nu} \in \mathcal{H}_r} \left\{f_r(\ul{\nu},\xi)- \bar{f}_{r,\xi}\right\})^2$, where $\cH_r$ is the collection of $r$ vertices such that the graph induced by those $r$ vertices is ${\sf H}_r$. This bound can be easily derived upon combining the previously mentioned estimates on $f_r(\ul{\nu},\xi)$ and  $\bar{f}_{r,\xi}$.

The above scheme of the proof of Theorem \ref{thm:main} has been motivated by the proof of \cite[Theorem 1]{cgw}. As mentioned earlier in Section \ref{sec:disc}, plugging in \eqref{eq:A1_1/2}-\eqref{eq:A2_1/2}, and repeating their proof would have only yielded the conclusion of Proposition \ref{prop:main} only upto $s=O(\sqrt{\log n})$. There is a key modification in our approach which we describe below. Likewise in the proof of Proposition \ref{prop:main}, in \cite{cgw} they also need to bound $(\sum_{\ul{\nu} \in \mathcal{H}_r} \left\{f_r(\ul{\nu},\xi)- \bar{f}_{r,\xi}\right\})^2$. However, they bound the former by a much larger quantity, namely $(\sum_{\ul{\nu} \in [n]^r} \left\{f_r(\ul{\nu},\xi)- \bar{f}_{r,\xi}\right\})^2$. Since for large $r$ the cardinality of $\cH_r$ is much smaller compared to $n^r$ direct application of the techniques from \cite[Theorem 1]{cgw} provides a much weaker conclusion. Here we obtain the required bound in a more efficient way by defining ``good'' and ``bad'' vertices, and controlling them separately. Here we should also mention that a similar idea was used in the work of Thomason \cite{thom2} in the context of jumbled graphs.



\section{Proof of Proposition \ref{prop:main}}
\label{sec:proof_prop}
In this section we prove Proposition \ref{prop:main}. As mentioned in Section \ref{sec:proof_outline} the proof is based on an induction-type argument and relies heavily in obtaining a recursive relation between $\cE_n(r)$ and $\cE_n(r+1)$, where we recall that 
\beq
\cE_n({\sf{H}}_r)=\left|\f{n_\Graph({\sf H}_r)}{\f{(n)_r}{|\mathrm{Aut}(\sf{H_r})|}} - \left(\f{1}{2}\right)^{{r \choose 2}}\right|, \notag
\eeq
and $\cE_n(r)=\max_{{\sf H}_r \in \cG(r)}\cE_n({\sf H}_r)$. In this section we prove the following desired recursive relation.
\begin{lem}\label{lem:error_inductive_bd}
Let $\{\Graph_n\}_{n \in \N}$ is a sequence of graphs satisfying assumptions \eqref{eq:A1_1/2}-\eqref{eq:A2_1/2}, such that 
\beq\label{eq:initial_condition}
\f{1}{2} \times \f{(n)_j}{|\mathrm{Aut}({\sf H}_j)|}\left(\f{1}{2}\right)^{{j \choose 2}} \le n_\Graph({\sf H}_j) \le 2 \times  \f{(n)_j}{|\mathrm{Aut}({\sf H}_j)|}\left(\f{1}{2}\right)^{{j \choose 2}}, \,{\sf H}_j \in \cG(j), \, j=2,3,\ldots, r.
\eeq
Then, there exists a positive constant $C^*$, depending only on $C$, and another positive constant, depending only on $\delta$, such that is for $r \le ((1-\delta) \wedge \f{1}{2})\log_2\hspace{-2pt}n - C_0 \log_2 \log_2\hspace{-2pt}n$, we have 
\begin{align}\label{eq:recursion}
\cE_n(r+1) \le \cE_n(r) \left(\frac{1}{2} \right)^r \left(1+ \f{ C^* r}{(\log_2\hspace{-2pt}n)^3}\right) + \f{C^* r }{(\log_2\hspace{-2pt}n)^3} \left(\f{1}{2}\right)^{r+1 \choose 2}.
\end{align}
\end{lem}
First let us show that the proof of Proposition \ref{prop:main} is an immediate consequence of Lemma \ref{lem:error_inductive_bd}. 

\begin{proof}[Proof of Proposition \ref{prop:main}]
The key to the proof is the recursive estimate obtained in Lemma \ref{lem:error_inductive_bd}. Note that for Lemma \ref{lem:error_inductive_bd} to hold, one needs to verify that \eqref{eq:initial_condition} is satisfied. To this end, observe that if $\cE_n(r) \ra 0$, for $r=2,3\ldots,s-1$, then \eqref{eq:initial_condition} is indeed satisfied for $r=s-1$. Therefore, by  Lemma  \ref{lem:error_inductive_bd}, we have $\cE_n(s) \ra 0$. Thus, using Lemma \ref{lem:error_inductive_bd} repeatedly we see that we would be able to control $\cE_n(s)$ if we have a good bound on $\cE_n(2)$, where the bound on the latter follows from our assumptions  
\eqref{eq:A1_1/2}-\eqref{eq:A2_1/2}. Below we make this idea precise.

We begin by noting that, if \eqref{eq:recursion} holds for $r=2,3,\ldots,s-1$, then by induction it is easy to show that
\beq\label{eq:error_bd}
\cE_n(s) \le \left(\prod_{j=2}^{s-1} \alpha_n(j)\right) \cE_n(2) + \sum_{j=2}^{s-1}\gamma_n^s(j),
\eeq
where, for $j=2,3,\ldots,s-1$,
\[
\alpha_n(j):= \left(\frac{1}{2} \right)^j \left(1+ \f{ C^* j}{(\log_2\hspace{-2pt}n)^3}\right), \,\gamma_n^s(j):= \beta_n(j) \prod_{k=j+1}^{s-1} \alpha_n(k), \, \beta_n(j):= \f{C^* j }{(\log_2\hspace{-2pt}n)^3} \left(\f{1}{2}\right)^{j+1 \choose 2}.
\]
Since $s \le \log_2 \hspace{-2pt}n$, from \eqref{eq:error_bd} it can also be deduced that 
\beq\label{eq:cE_n_s_bd}
\cE_n(s) \le 2  \cE_n(2) \left(\f{1}{2}\right)^{s \choose 2} + \f{2C^*}{(\log_2\hspace{-2pt}n)^2}\left(\f{1}{2}\right)^{s \choose 2}.
\eeq
Now, note that there only two graphs in $\cG(2)$, namely the complete graph and the empty graph on two vertices. Therefore,
\begin{align*}
\cE_n(2) =\max_{\xi \in \{0,1\}} \left|\f{\f{1}{2}\sum_{\nu \in [n]} \sN_\nu^\xi}{{n \choose 2}} - \f{1}{2}\right| & = \max_{\xi \in \{0,1\}}\left|\f{\sum_{\nu \in [n]} \sN_\nu^\xi}{n(n-1)} - \f{1}{2}\right| \\
& \le \f{1}{n(n-1)}\max_{\xi \in \{0,1\}}\left|{\sum_{\nu \in [n]} \left(\sN^\xi_\nu - \f{n}{2}\right)}\right|+ \f{1}{2(n-1)}\\
& \le \f{1}{n(n-1)}\max_{\xi \in \{0,1\}}\sum_{\nu\in [n]} \left|\sN_\nu^\xi - \f{n}{2}\right| + \f{1}{2(n-1)} \le 2C n^{\delta-1},
\end{align*}
where the last step follows from \eqref{eq:A1_1/2}. This combining with \eqref{eq:cE_n_s_bd} yields
\beq\label{eq:cE_n_s_new}
\cE_n(s) \le 4 C  n^{\delta-1} \left(\f{1}{2}\right)^{s \choose 2} + \f{2C^*}{(\log_2\hspace{-2pt}n)^2}\left(\f{1}{2}\right)^{s \choose 2},
\eeq
whenever \eqref{eq:recursion} holds for $r=2,3,\ldots,s-1$. Now the rest of the proof is completed by induction. 

Indeed, we note that \eqref{eq:initial_condition} holds for $r=2$, as we have already seen $\cE_n(2) \le 2C n^{\delta-1} \ra 0$, as  $n \ra \infty$ (since $\delta <1$). Therefore, \eqref{eq:recursion} holds for $r=3$, and hence by \eqref{eq:cE_n_s_new} we see that $\cE_n(3) \ra 0$. More generally, if we have that $\cE_n(r) \ra 0$, for $r=2,3,\ldots,\ell$, then it implies that \eqref{eq:initial_condition} holds for $r=2, 3, \ldots, \ell$ as well. Hence, \eqref{eq:cE_n_s_new} now implies that $\cE_n(\ell+1) \ra 0$. Thus, the proof finishes by induction.
\end{proof}

Now we turn our attention to the proof of Lemma \ref{lem:error_inductive_bd}. As already seen in Section \ref{sec:proof_outline} we need to establish that given any $\ul{v} \in [n]^r$ and $\ul{\xi} \in \{0,1\}^r$ if $\nu_j \in \Good^{\ul{\xi}^{j}}(\nu_1, \nu_2,...,\nu_{j-1})$,  for all $j =3,4,\ldots,r$, then $f_r(\underline{\nu},\ul{\xi}) \approx {n}/{2^r}$. We establish this claim formally in the lemma below.
\begin{lem}\label{common-nbr-allgood}
Let $\{\Graph_n\}_{n \in \N}$ be a sequence of graphs satisfying \eqref{eq:A1_1/2}-\eqref{eq:A2_1/2}. {Fix $r < \log_2 \hspace{-2pt}n$,} and let $\ul{v} \in [n]^r$ and $\ul{\xi} \in \{0,1\}^r$. If $\nu_j \in \Good^{\ul{\xi}^{j}}(\nu_1, \nu_2,...,\nu_{j-1})$,  for all $j =3,4,\ldots,r$, and $\bar{C}_0 \ge 4$, then 
\beq\label{eq:good_set_estimates}
\left|\left|\sN_{\nu_1}^{\xi_1} \cap \sN_{\nu_2}^{\xi_2} \cap \cdots \cap \sN_{\nu_r}^{\xi_r}\right|- \f{n}{2^r}\right|  \le 3\wt{C} (\log_2\hspace{-2pt}n)^{-(\bar{C}_0/2-1)}  \times \f{n}{2^r},
\eeq
for all large $n$.
\end{lem}

\bigskip
We also need to show that the cardinality of ``bad'' vertices is relatively small. This is derived in the following lemma.
\begin{lem}\label{lem:Bad_bound}
Let $\{\Graph_n\}_{n \in \N}$ be a sequence of graphs satisfying \eqref{eq:A1_1/2}-\eqref{eq:A2_1/2}. Fix any two positive integers $j< r$, and let ${\sf H}_r$ be a graph on $r$ vertices. Fix ${\sf H}_j$ any one of the sub-graphs of ${\sf H}_r$ induced by $j$ vertices. Assume that 
\[
n_\Graph({\sf H_r}) \ge \f{1}{2} \times \f{(n)_r}{|\mathrm{Aut}({\sf H}_r)|2^{{r \choose 2}}} \quad \text{ and } \quad n_\Graph({\sf H_j}) \le 2 \times \f{(n)_j}{|\mathrm{Aut}({\sf H}_j)|2^{{j \choose 2}}}. 
\]
There exists a large absolute constant $C_0$ such that, for any given $\ul{\xi}=\{\xi_1,\xi_2,\ldots,\xi_j\} \in \{0,1\}^j$, and $r \le ((1-\delta)\wedge \f{1}{2})\log_2\hspace{-2pt}n- C_0 \log_2 \log_2 \hspace{-2pt}n$, we have 
\[
\left|\Bad^{\ul{\xi}}({\sf H}_{j}, {\sf H}_r)\right| \le \f{n_\Graph({\sf H_r})}{\left(\log_2\hspace{-2pt}n \right)^{9(r-j)}},
\]
for all large $n$.
\end{lem}

\bigskip
Recall that we also need bounds on $\bar{f}_{r,\ul{\xi}}$. This is obtained from the following lemma.

\begin{lem}\label{lem:f_r_bar_bound}
For any $r$, let $\ul{\xi}:=\{\xi_1,\xi_2,\ldots,\xi_r\} \in \{0,1\}^r$, and $n(\ul{\xi})$ be the number of zeros in $\ul{\xi}$. 
Then for any $r\le \log_2\hspace{-2pt}n$,
\beq
\bar{f}_{r, \ul{\xi}}= \f{1}{(n)_r}\sum_{\nu \in [n]}(\sN_\nu^0 )^{n(\ul{\xi})}{(\sN_\nu^1)^{r-n(\ul{\xi})}}.\label{eq:f_r_xi_bound}
\eeq
We further have
\beq\label{eq:avgbd}
n \left(\frac{1}{2} \right)^r\left(1-12 C  r n^{\delta-1}\right)\left(1-\f{2r^2}{n}\right) \le \bar{f}_{r,\xi} \leq n \left(\frac{1}{2} \right)^r\left(1+12 C  r n^{\delta-1}\right)\left(1+\f{2r^2}{n}\right).
\eeq
\end{lem}


\bigskip

Finally we need to bound $(\sum_{\ul{\nu} \in \mathcal{H}_r} \left\{f_r(\ul{\nu},\xi)- \bar{f}_{r,\xi}\right\})^2$. Building on Lemma \ref{common-nbr-allgood}, Lemma \ref{lem:Bad_bound}, and Lemma \ref{lem:f_r_bar_bound} we obtain the required bound in the lemma below.

\begin{lem}\label{lem:var_type_bd}
Let $\{\Graph_n\}_{n \in \N}$ be a sequence of graphs satisfying \eqref{eq:A1_1/2}-\eqref{eq:A2_1/2}. Fix any positive integer $ r$, and let ${\sf H}_r$ be a graph on $r$ vertices. Assume that 
\[
n_\Graph({\sf H_r}) \ge \f{1}{2} \times \f{(n)_r}{|\mathrm{Aut}({\sf H}_r)|2^{{r \choose 2}}} \quad \text{ and } \quad n_\Graph({\sf H_j}) \le {2} \times \f{(n)_j}{|\mathrm{Aut}({\sf H}_j)|2^{{j \choose 2}}},
\]
for all $j <r$ and all graphs ${\sf H}_j$ on $j$ vertices. Fix any $\ul{\xi}\in \{0,1\}^r$. Then, there exists a positive constant $C_0$, depending only on $\delta$, and another positive constant $\widehat{C}$, depending only on $C$, such that for any $r \le ((1-\delta) \wedge \frac{1}{2})\log_2\hspace{-2pt}n - C_0 \log_2\log_2\hspace{-2pt}n$, we have
\beq\label{eq:f_r_var}
\left|\sum_{\ul{\nu} \in \cH_r} \left\{f_r(\ul{\nu},\ul{\xi})- \bar{f}_{r,\ul{\xi}}\right\}\right| \le \wh{C} n_\Graph({\sf H} _r) \f{n}{(\log_2\hspace{-2pt}n)^3} \left(\f{1}{2}\right)^{r}, 
\eeq
for all large $n$.
\end{lem}

\bigskip

Before going to the proof let us mention one more lemma that is required to prove the main result inductively. Its proof is elementary. We include the proof for completion.
\begin{lem}\label{lem:subgraphaut}
Fix ${\sf H}_{r+1}$, a graph on $(r+1)$ vertices with vertex set $[r+1]$. Let ${\sf H}_r$ be the sub-graph of ${\sf H}_{r+1}$ with vertex set $[r]$. Then, there exists a $\ul{\xi}^r \in \{0,1\}^r$ such that 
\beq\label{eq:subgraphaut}
\left| \mathrm{Aut}({\sf H}_{r+1})\right| n_\Graph ({\sf H}_{r+1}) =  \left| \mathrm{Aut}({\sf H}_{r})\right| \sum_{\ul{\nu} \in \mathcal{H}_r}{f_r(\ul{\nu},\ul{\xi}^r)}.
\eeq
\end{lem}
\begin{proof}
For any collection of vertices $\{v_1,v_2,\ldots,v_s\}$, let $\Graph(v_1,v_2,\ldots,v_s)$ be the sub-graph induced by those vertices. Let ${\sf H}'$ and ${\sf H}''$ be two graphs with vertex sets $\{v_1,v_2,\ldots,v_s\}$ and $\{v_1',v_2',\ldots,v_s'\}$, respectively. We write ${\sf H}'= {\sf H}''$, if under the permutation $\pi$, for which $\pi(v_i)=v_i'$ for all $i=1,2,\ldots,s$, the graphs are same. Then, for any given collection of vertices $\{v_1,v_2,\ldots,v_r\}$, if $\Graph(v_{\pi(1)}, v_{\pi(2)},\ldots,v_{\pi(r+1)})={\sf H}_{r+1}$ for some permutation of the vertices $\{v_1,v_2,\ldots,v_r\}$, then there are actually $|\mathrm{Aut}({\sf H}_{r+1)}|$ many such permutations. Thus recalling the definition of $n_\Graph(\cdot)$ we immediately deduce that 
\[
\sum_{v_1,v_2\ldots,v_{r+1} \in [n]} \bI(\Graph(v_1,v_2,\ldots,v_{r+1})= {\sf H}_{r+1})= \left| \mathrm{Aut}({\sf H}_{r+1})\right| n_\Graph ({\sf H}_{r+1}).
\]
We also note that   $\Graph(v_1,v_2,\ldots,v_{r+1})= {\sf H}_{r+1}$, if and only if $\Graph(v_1,v_2,\ldots,v_{r})= {\sf H}_{r}$ and $v_{r+1} \in \sN_{v_1}^{\xi_1^r}\cap \sN_{v_2}^{\xi_2^r} \cap \cdots \cap \sN_{v_r}^{\xi_r^r}$, for some appropriately chosen $\ul{\xi}^r=\{\xi_1^r,\xi_2^r,\ldots,\xi_r^r\}$. Thus, recalling the definition of $f_r(\ul{v},\ul{\xi}^r)$ we see that
\[
\sum_{v_1,v_2\ldots,v_{r+1} \in [n]} \bI(\Graph(v_1,v_2,\ldots,v_{r+1})= {\sf H}_{r+1})=  \sum_{v_1,v_2\ldots,v_{r} \in [n]} \bI(\Graph(v_1,v_2,\ldots,v_{r})= {\sf H}_{r}) f_r(\ul{v},\ul{\xi}^r).
\]
Now noting that $f_r(\ul{v}, \ul{\xi}^r)$ is invariant under any permutation of coordinates of $\ul{v}$ and recalling the definitions of $n_\Graph(\cdot)$ and $|\mathrm{Aut}(\cdot)|$ again, we observe that the \abbr{RHS} of the above equation is same as the the \abbr{RHS} of \eqref{eq:subgraphaut}, which completes the proof.
\end{proof}

Now we are ready to prove Lemma \ref{lem:error_inductive_bd}.

\begin{proof}[Proof of Lemma \ref{lem:error_inductive_bd}] 
Fix any graph ${\sf H}_{r+1}$ with vertex set $[r+1]$. Applying the triangle inequality we have
\begin{align}
\cE_n({\sf H}_{r+1})& \le \left|\f{|\mathrm{Aut}({\sf H}_{r+1})| n_\Graph({\sf H}_{r+1})- \bar{f}_{r,\ul{\xi}^r} n_\Graph({\sf H}_{r})|\mathrm{Aut}({\sf H}_r)|}{(n)_{ r+1}}\right| + \left|\f{ \bar{f}_{r,\ul{\xi}^r} n_\Graph({\sf H}_{r})|\mathrm{Aut}({\sf H}_r)|}{{(n)_{r+1}}} - \left(\f{1}{2}\right)^{{r +1 \choose 2}}\right| \notag\\
& \le \left|\f{|\mathrm{Aut}({\sf H}_{r+1})| n_\Graph({\sf H}_{r+1})- \bar{f}_{r,\ul{\xi}^r} n_\Graph({\sf H}_{r})|\mathrm{Aut}({\sf H}_r)|}{(n)_{ r+1}}\right| + \f{\bar{f}_{r,\ul{\xi}^r}{(n)_r}}{{(n)_{r+1}}}\left|\f{n_\Graph({\sf H}_r)|\mathrm{Aut}({\sf H}_r)|}{{(n)_r}} - \left(\f{1}{2}\right)^{{r \choose 2}}\right|\notag\\
& \qquad \qquad \qquad \qquad \qquad \qquad + \left|\f{\bar{f}_{r,\ul{\xi}^r}{(n)_r}}{{(n)_{r+1}}}\left(\f{1}{2}\right)^{{r \choose 2}}- \left(\f{1}{2}\right)^{{r+1 \choose 2}}\right|\label{eq:error_break}\\
& =: \text{Term I} + \text{Term II}+ \text{Term III}.\notag
\end{align}
Recalling the definition of $\cE_n(r)$ and using Lemma \ref{lem:f_r_bar_bound}, we note that
\begin{align}
\text{Term II}= \f{\bar{f}_{r,\ul{\xi}}{(n)_r}}{{(n)_{r+1}}}\left|\f{n_\Graph({\sf H}_r)|\mathrm{Aut}(\mathrm{H}_r)|}{{(n)_r}} - \left(\f{1}{2}\right)^{{r \choose 2}}\right| &\le \cE_n(r)\frac{n}{n-r} \left(\frac{1}{2} \right)^r \left(1+12 C  r n^{\delta-1}\right)\left(1+\f{2r^2}{n}\right) \notag \\
& \le \cE_n(r) \left(\frac{1}{2} \right)^r \left(1+ \f{16 C r}{(\log_2\hspace{-2pt}n)^3}\right), \label{eq:rec_term2}
\end{align}
where the last inequality follows from that fact that $r \le \log_2\hspace{-2pt}n$. Now to control Term III, using Lemma \ref{lem:f_r_bar_bound} we have
\[
\left|\bar{f}_{r,\ul{\xi}} - n \left(\f{1}{2}\right)^r\right|  \le  n\left(\f{1}{2}\right)^r\f{14C r}{(\log_2\hspace{-2pt}n)^4}.
\]
Therefore, using the triangle inequality, 
\begin{align}
\text{Term III}   = \left|\f{\bar{f}_{r,\ul{\xi}}{(n)_r}}{{(n)_{r+1}}}\left(\f{1}{2}\right)^{{r \choose 2}}- \left(\f{1}{2}\right)^{{r+1 \choose 2}}\right| \le \f{n{(n)_r}}{{(n)_{r+1}}}\left(\f{1}{2}\right)^{{r+1 \choose 2}}\f{14 C r}{(\log_2\hspace{-2pt}n)^4}+2\f{r}{n}\left(\f{1}{2}\right)^{{r+1 \choose 2}}.\label{eq:rec_term3}
\end{align}
Now it remains to bound Term I. To this end, using Lemma \ref{lem:subgraphaut} and the triangle inequality, from Lemma \ref{lem:var_type_bd} we obtain
\begin{align}
\text{ Term I} &= \left|\f{|\mathrm{Aut}({\sf H}_{r+1})| n_\Graph({\sf H}_{r+1})- \bar{f}_{r,\ul{\xi}} n_\Graph({\sf H}_{r})|\mathrm{Aut}({\sf H}_r)|}{(n)_{ r+1}}\right| \notag\\
&=    \f{|\mathrm{Aut}({\sf H}_r)|}{(n)_{r+1}}\left|\sum_{\ul{\nu} \in \mathcal{H}_r} \left\{f_r(\ul{\nu},\xi)- \bar{f}_{r,\xi}\right\}\right| \le   \f{|\mathrm{Aut}(\mathrm{H}_r)| \widehat{C}n_\Graph({\sf H} _r) n }{{(n)_{r+1}}(\log_2\hspace{-2pt}n)^3}\left(\f{1}{2}\right)^r \le \f{6 \wh{C}}{(\log_2\hspace{-2pt}n)^3} \left(\f{1}{2}\right)^{r+1 \choose 2},  \label{eq:rec_term1}
\end{align}
where in the last step we use the fact that $ n_\Graph({\sf H}_r) |\mathrm{Aut}({\sf H}_r)| \le 2  {(n)_r}\left(\f{1}{2}\right)^{{r \choose 2}}$.
Finally combining \eqref{eq:rec_term2}-\eqref{eq:rec_term1} we complete the proof.
\end{proof}


\section{Preparatory Technical Lemmas}
\label{sec:proof_lemma}


In this section we prove Lemma \ref{common-nbr-allgood}, Lemma \ref{lem:Bad_bound}, Lemma \ref{lem:f_r_bar_bound}, and Lemma \ref{lem:var_type_bd}. 

We begin with the proof of Lemma \ref{common-nbr-allgood}. Before going to the proof let us recall the definition of ``good'' vertices from Definition \ref{dfn:good_dfn}. Using Definition \ref{dfn:good_dfn} the proof follows by induction.

\begin{proof}[Proof of Lemma \ref{common-nbr-allgood}]
Note that, by our assumption \eqref{eq:A1_1/2}-\eqref{eq:A2_1/2}, we have good bounds on $|\sN_{v_1}^{\xi_1} \cap \sN_{v_2}^{\xi_2}|$ (see \eqref{eq:bd_generalized_nbhd}). We propagate this bound for a general $j$ by induction. Using induction we will prove a slightly stronger bound. Namely we will prove the following:
\beq\label{eq:good_set_estimates_ind}
\left|\left|\sN_{\nu_1}^{\xi_1} \cap \sN_{\nu_2}^{\xi_2} \cap \cdots \cap \sN_{\nu_r}^{\xi_r}\right|- \f{n}{2^r}\right|  \le 3\wt{C} r n^{\f{\vep(\delta-1)}{2}}  \times \f{n}{2^r}.
\eeq
From \eqref{eq:good_set_estimates_ind} the conclusion follows by noting that $r < \log_2\hspace{-2pt}n$.

To this end, using the triangle inequality we note that
\begin{align}
\left|\left|\sN_{\nu_1}^{\xi_1} \cap \sN_{\nu_2}^{\xi_2} \cap \cdots \cap \sN_{\nu_j}^{\xi_j}\right|- \f{n}{2^j}\right|  &\le   \left|\left|\sN_{\nu_1}^{\xi_1} \cap \sN_{\nu_2}^{\xi_2} \cap \cdots \cap \sN_{\nu_j}^{\xi_j}\right|- \f{1}{2}\left|\sN_{\nu_1}^{\xi_1} \cap \sN_{\nu_2}^{\xi_2} \cap \cdots \cap \sN_{\nu_{j-1}}^{\xi_{j-1}}\right|\right| \notag\\
& \qquad \qquad \qquad  \qquad \qquad + \f{1}{2}\left|\left|\sN_{\nu_1}^{\xi_1} \cap \sN_{\nu_2}^{\xi_2} \cap \cdots \cap \sN_{\nu_{j-1}}^{\xi_{j-1}}\right|- \f{n}{2^{j-1}}\right|. \notag
\end{align}
Since $v_j \in \Good^{\ul{\xi}^{j}}(\nu_1, \nu_2,...,\nu_{j-1})$, we therefore have
$$
 \left|\left|\sN_{\nu_1}^{\xi_1} \cap \sN_{\nu_2}^{\xi_2} \cap \cdots \cap \sN_{\nu_j}^{\xi_j}\right|- \f{1}{2}\left|\sN_{\nu_1}^{\xi_1} \cap \sN_{\nu_2}^{\xi_2} \cap \cdots \cap \sN_{\nu_{j-1}}^{\xi_{j-1}}\right|\right| \leq \wt{C} n^{\f{\vep(\delta-1)}{2}} \left|\sN_{\nu_1}^{\xi_1} \cap \sN_{\nu_2}^{\xi_2} \cap \cdots \cap \sN_{\nu_{j-1}}^{\xi_{j-1}}\right|, $$
for all $j=3,4,\ldots,r$. Recalling that $\vep=\frac{\bar{C}_0 \log_2\hspace{-2pt}n \log_2\hspace{-2pt}n}{(1-\delta)\log_2\hspace{-2pt}n}$ and using induction hypothesis we obtain
\begin{align}
\left|\left|\sN_{\nu_1}^{\xi_1} \cap \sN_{\nu_2}^{\xi_2} \cap \cdots \cap \sN_{\nu_j}^{\xi_j}\right|- \f{n}{2^j}\right|  
&  \le \wt{C} n^{\f{\vep(\delta-1)}{2}} \left|\sN_{\nu_1}^{\xi_1} \cap \sN_{\nu_2}^{\xi_2} \cap \cdots \cap \sN_{\nu_{j-1}}^{\xi_{j-1}}\right|+ {3}\wt{C}(j-1) n^{\f{\vep(\delta-1)}{2}}  \times \f{n}{2^j}.\notag
\end{align}
Since $ r n^{\f{\vep(\delta -1)}{2}} \ra 0$, as $n \ra \infty$, the induction hypothesis further implies 
\beq\label{eq:common_neighbor}
\left|\sN_{\nu_1}^{\xi_1} \cap \sN_{\nu_2}^{\xi_2} \cap \cdots \cap \sN_{\nu_{j-1}}^{\xi_{j-1}}\right| \le  \f{n}{2^{j-1}} + 3 \wt{C} (j-1)n^{\f{\vep(\delta-1)}{2}} \times \f{n}{2^{j-1}} \le \frac{3}{2}\times \f{n}{2^{j-1}} = 3 \times \f{n}{2^{j}},
\eeq
for all large $n$, and hence we have \eqref{eq:good_set_estimates_ind}. This completes the proof.
\end{proof}

\bigskip

We now proceed to the proof of Lemma \ref{lem:Bad_bound}. That is, we want to show that the number of ``bad'' vertices is negligible. To prove Lemma \ref{lem:Bad_bound} we first show that the number of vertices $v$ such that $v \notin \mathrm{Good}^\xi(B)$ is small. To this end, we first prove the following variance bound. Here for ease of exposition write $\E(\cdot)$, $\mathrm{Var}(\cdot)$ etc where the expectations will be taken with respect to the empirical measure over $[n]$. 
\begin{lem}\label{lem:var_N_nu}
Let $\{\Graph_n\}_{n \in \N}$ be a sequence of graphs satisfying \eqref{eq:A1_1/2}-\eqref{eq:A2_1/2}. Then for any set $B \subset [n]$, and $\xi\in \{0,1\}$,
\[
\mathrm{Var} \left(\sN_\nu^{B,\xi}\right) \le \frac{|B|}{4} + \bar{C} {|B|^2 n^{\delta-1}},
\]
for some constant $\bar{C}$, depending only on $C$.
\end{lem}

\begin{proof}
Recall that $a^\xi_{v_1v_2}=\bI(a_{v_1v_2}=\xi)$, where $(a_{\nu_1 \nu_2})_{\nu_1,\nu_2=1}^n$ is the adjacency matrix of the graph $\Graph_n$. Therefore, we have

\begin{align*}
& \sum_{\nu_1,\nu_2 \in [n]}{\left({\sum_{\nu \in B}\left(a_{\nu \nu_1}^\xi - a_{\nu \nu_2}^\xi\right)}\right)^2} \\
& \qquad \qquad \qquad = \sum_{\nu_1,\nu_2 \in [n]}\left\{\left(\sum_{\nu \in B}a_{\nu \nu_1}^\xi\right)^2 + \left(\sum_{\nu \in B}a_{\nu \nu_2}^\xi\right)^2- 2\left(\sum_{\nu \in B}a_{\nu \nu_1}^\xi\right) \left(\sum_{\nu \in B}a_{\nu \nu_2}^\xi\right)\right\} \\
& \qquad \qquad \qquad = 2n \sum_{\nu_1 \in [n]}\left(\sum_{\nu \in B}a_{\nu \nu_1}^\xi\right)^2 - 
2\sum_{\nu,\nu' \in B}\left(\sum_{\nu_1 \in [n]}a_{\nu \nu_1}^\xi\right)\left(\sum_{\nu_2 \in [n]}a_{\nu'  \nu_2}^\xi\right) \\
& \qquad \qquad \qquad =  2n \sum_{\nu_1 \in [n]} \left\{\sum_{\nu \in B}{a_{\nu \nu_1}^\xi} + 2\sum_{\nu < \nu' \in B}{a_{\nu \nu_1}^\xi a_{\nu' \nu_1}^\xi}\right\} - 2\left(\sum_{\nu \in B}{\sN_{\nu}^\xi}\right)^2 \\
& \qquad \qquad \qquad = 2n\left\{\sum_{\nu \in B}{\sN_{\nu}^\xi} + 2 \sum_{\nu < \nu' \in B}{|\sN_{\nu}^\xi \cap \sN_{\nu'}^\xi|}\right\}
- 2\left(\sum_{\nu \in B}{\sN_{\nu}^\xi}\right)^2. 
\end{align*}
Using \eqref{eq:A1_1/2} and \eqref{eq:A2_1/2}, by the triangle inequality, we note that
\beq\label{eq:bd_generalized_nbhd}
\sup_{1 \leq \nu \leq n}{\left|\left|\sN_\nu^\xi\right| - \frac{n}{2}\right|}, \sup_{1\leq \nu \ne \nu' \leq n}{\left|\left|\sN_\nu^\xi \cap \sN_{\nu'}^{\xi'}\right| - \frac{n}{4}\right|}  < 3Cn^{\delta}, \qquad \text{ for any } \xi,\xi' \in \{0,1\}.
\eeq
Thus, continuing from above
\begin{align*}
&\sum_{\nu_1,\nu_2 \in [n]}{\left({\sum_{\nu \in B}\left(a_{\nu \nu_1}^\xi - a_{\nu \nu_2}^\xi\right)}\right)^2}\\
& \qquad \qquad \leq  2n \left\{\sum_{\nu \in B}\left(\frac{n}{2} + 3Cn^\delta\right) + 2\sum_{\nu < \nu' \in B}\left(\frac{n}{4}+3Cn^{\delta}\right)\right\} - 2\left(\sum_{\nu \in B}\left(\frac{n}{2} - 3Cn^{\delta}\right)\right)^2 \\
& \qquad \qquad = 2n \left\{|B|\left(\frac{n}{2} + 3Cn^\delta\right)+2 {|B| \choose 2}\left(\frac{n}{4} +3Cn^{\delta}\right)\right\}
- 2|B|^2\left(\frac{n}{2} - 3Cn^{\delta}\right)^2 \\
& \qquad \qquad  \le \frac{n^2 |B|}{2} + 2\bar{C}|B|^2 n^{1+\delta },
\end{align*}
for some constant $\bar{C}$, depending only on $C$. Finally observing that 
$$ 2n^2\mathrm{Var}(\sN_\nu^{B,\xi}) =\sum_{\nu_1,\nu_2 \in [n]}{\left({\sum_{\nu \in B}(a_{\nu \nu_1}^\xi - a_{\nu \nu_2}^\xi)}\right)^2},$$ 
completes the proof.
\end{proof}


\begin{rmk}
Note that for an Erd\H{o}s-R\'{e}yni graph with edge connectivity probability $1/2$, one can show that $\mathrm{Var} (\sN_\nu^{B,\xi}) = \frac{|B|}{4}$, for any subset vertices $B$, where the variance is with respect to the randomness of the edges and the uniform choice of the vertex $v$. 
For pesudo-random graphs satisfying \eqref{eq:A1_1/2}-\eqref{eq:A2_1/2}, repeating the same steps as in the proof of Lemma \ref{lem:var_N_nu}, we can also obtain that
\[
\mathrm{Var} \left(\sN_\nu^{B,\xi}\right) \ge \frac{|B|}{4} - \bar{C} {|B|^2 n^{\delta-1}}.
\]
Therefore, we see that pseudo-random graphs satisfying \eqref{eq:A1_1/2}-\eqref{eq:A2_1/2} are not much different from  $\Graph(n,1/2)$ in this aspect. 
\end{rmk}

\vskip10pt

\noindent
From Lemma \ref{lem:var_N_nu}, using Markov's inequality we obtain the following result.
\begin{lem}\label{lem:good_B}
Let $\{\Graph_n\}_{n \in \N}$ be a sequence of graphs satisfying \eqref{eq:A1_1/2}-\eqref{eq:A2_1/2}. Fix $\vep = \f{\bar{C}_0\log_2\hspace{-2pt}\log_2\hspace{-2pt}n}{(1-\delta)\log_2\hspace{-2pt}n}$. Then, there exists a positive constant $\wt{C}$, depending only on $C$, such that for any set $B \subset [n]$, and $\xi \in \{0,1\}$
\[
\left|\left\{ \nu \in [n]: \left|\sN_\nu^{B,\xi} - \frac{|B|}{2}\right| > \wt{C} |B| n^{\vep(\delta-1)/2}\right\}\right| \le n (\log_2\hspace{-2pt}n)^{\bar{C}_0}\Upsilon_n(|B|,\delta),
\]
where $\Upsilon_n(x,\delta):=\frac{x^{-1}+n^{\delta-1}}{2}$.
\end{lem}

\begin{proof}
The proof is a straightforward application of Lemma \ref{lem:var_N_nu}. From Lemma \ref{lem:var_N_nu}, by Chebychev's inequality we deduce,
\begin{align}
\P(|\sN_\nu^{B,\xi} - \E(\sN_\nu^{B,\xi})| > \varpi) \leq \frac{\mathrm{Var}(\sN_\nu^{B,\xi})}{\varpi^2} \le \f{\bar{C}(|B|+|B|^2 n^{\delta-1})}{\varpi^2},\notag
\end{align}
for every $\varpi >0$. Setting $\varpi=\sqrt{2\bar{C}}|B| n^{\vep(\delta-1)/2}$, from above we therefore obtain
\beq \label{eq:markov}
\P\left(|\sN_\nu^{B,\xi} - \E(\sN_\nu^{B,\xi})| > \sqrt{2\bar{C}}|B| n^{\vep(\delta-1)/2}\right) \le  \f{\f{n^{\vep(1-\delta)}}{|B|}+n^{(\delta-1)(1-\vep)}}{2}.
\eeq
Next we note that
\[
\sum_{\nu \in [n]}\sN_\nu^{B,\xi}= \sum_{\nu' \in B}\sum_{\nu \in [n]} a_{\nu \nu'}^\xi= \sum_{\nu' \in B} \sN_{\nu'}^\xi,
\]
and thus using \eqref{eq:bd_generalized_nbhd}, we further have
\beq\label{eq:expec}
\frac{|B|}{2} - 3C|B|n^{\delta-1} \leq \E\left(\sN_\nu^{B,\xi}\right)  \leq \frac{|B|}{2} + 3C|B|n^{\delta-1}.
\eeq
Thus combining \eqref{eq:markov}-\eqref{eq:expec} the required result follows upon using triangle inequality.
\end{proof}


We now use Lemma \ref{lem:good_B} to prove Lemma \ref{lem:Bad_bound}. Before going to the proof of Lemma \ref{lem:Bad_bound} we need one more elementary result.

\begin{lem}\label{lem:automorphism}
Let ${\sf H}$ be any graph with vertex set $[m]$, and ${\sf H}'$ be one of its vertex-deleted induced sub-graph. That is, ${\sf H}'$ is the sub-graph induced by the vertices $[m] \setminus \{v\}$ for some $v \in [m]$. 
Then
\[
\left|\mathrm{Aut}({\sf H})\right| \le m \left|\mathrm{Aut}({\sf H}')\right| .
\]
\end{lem}

\begin{proof}
Let us denote $\mathrm{Aut}_v({\sf H})$ to be the vertex-stablilizer sub-group. That is,
\[
\mathrm{Aut}_v({\sf H}):=\left\{\pi \in \mathrm{Aut}({\sf H}): \pi(v)=v\right\}.
\]
Clearly  $\mathrm{Aut}_v({\sf H})$ can be embedded into $\mathrm{Aut}({\sf H}')$, and hence $\left|\mathrm{Aut}_v({\sf H})\right| \le\left|\mathrm{Aut}({\sf H}')\right|$. Thus we only need to show that
\[
\left|\mathrm{Aut}({\sf H})\right| \le m \left|\mathrm{Aut}_v({\sf H})\right| .
\]
Using Lagrange's theorem (see \cite[Section 3.2]{DF} for more details), we note that this boils down to showing that the number of distinct left cosets of $\mathrm{Aut}_v({\sf H})$ in $\mathrm{Aut}({\sf H})$ is less than or equal to $m$. To this end, it is easy to check that for any $\pi, \pi' \in \mathrm{Aut}({\sf H})$, if $\pi(v)=\pi'(v)$, then $\pi \mathrm{Aut}_v({\sf H})=\pi' \mathrm{Aut}_v({\sf H})$. Since $\pi$ is a permutation on $[m]$, we immediately have the desired conclusion.
\end{proof}

Now we are ready to prove Lemma \ref{lem:Bad_bound}.

\begin{proof}[Proof of Lemma \ref{lem:Bad_bound}]
Recalling the definition of $\cH_r$ (see Definition \ref{dfn:bad_define}), we see that given any $\ul{v} \in \cH_r$, there exists a relabeling $\hat{v}$, such that $\hat{\ul{v}}^j \in \cH_j$ (see Definition \ref{dfn:bad_define} for a definition of $\hat{\ul{v}}^j$). We also note that $\hat{v}_k \in \mathrm{Good}^{\ul{\xi}^k}(\hat{v}_1,\hat{v}_2,\ldots,\hat{v}_{k-1})$ for $k=3,4,\ldots,j$. Therefore using Lemma \ref{common-nbr-allgood} we note that 
\[
\left|\sN_{\hat{\nu}_1}^{\xi_1} \cap \sN_{\hat{\nu}_2}^{\xi_2} \cap \cdots \cap \sN_{\hat{\nu}_j}^{\xi_j}\right|\ge  \f{n}{2^{j+1}} \ge \f{n}{2^{r}}.
\]
Thus applying Lemma \ref{lem:good_B}, using the union bound, from our assumption on $n_{\Graph}({\sf H}_j)$ we immediately deduce that
\[
|\Bad^{\ul{\xi}}_{j,r}| \le 2 \f{(n)_j}{|\mathrm{Aut}({\sf H}_j)|2^{{j \choose 2}}} \times \left(n (\log_2\hspace{-2pt}n)^{\bar{C}_0}\Upsilon_n\left(\f{n}{2^r},\delta\right)\right)^{r-j}.
\]
Recall that by our assumption 
$$n_{\Graph}({\sf H}_r) \ge \f{1}{2}\f{(n)_r}{|\mathrm{Aut}({\sf H}_r)|2^{{r \choose 2}}}.$$
Further applying Lemma \ref{lem:automorphism} repetitively we deduce that
\[
|\mathrm{Aut}({\sf H}_r)| \le \f{r!}{j!} |\mathrm{Aut}({\sf H}_j)|.
\]
Hence using Stirling's approximation we obtain
\begin{align}
&\log_2 \left(\f{|\Bad_{j,r}^{\ul{\xi}}|}{n_{\Graph}({\sf H}_r)}\right) \notag\\
 \le & \log_2\hspace{-2pt}4 + \log_2\left(\f{{n \choose j}}{{n \choose r}}\right)+ (r-j)\left[\log_2\hspace{-2pt}n+ \bar{C}_0 \log_2\hspace{-2pt} \log_2\hspace{-2pt}n +\log_2 \hspace{-2pt}\Upsilon_n(n/2^r,\delta)\right]+ \left[{r \choose 2} - {j \choose 2}\right] \notag\\
 \le & \log_2\hspace{-2pt}4 -\sum_{k=j}^{r-1} \log_2\hspace{0pt}(n-k)+\log_2 (\sqrt{2 \pi} e)-(r-j)\log_2\hspace{-2pt}e+ \f{1}{2}(\log_2\hspace{-2pt}r-\log_2\hspace{-2pt}j)  \notag\\
&  \qquad+r \log_2\hspace{-2pt}r- j \log_2\hspace{-2pt}j+ (r-j)\left[\log_2\hspace{-2pt}n+ \bar{C}_0 \log_2\hspace{-2pt} \log_2\hspace{-2pt}n +\log_2 \hspace{-2pt}\Upsilon_n(n/2^r,\delta)\right]+ \left[\sum_{k=j}^{r-1} k\right].\notag
\end{align}
Using the fact that $\log_e(1-x) \ge -2 x$, for $x \in (0,1/2)$, we further note that
\begin{align}\label{eq:prelim_bd_p1}
(r-j) \log_2\hspace{-2pt}n - \sum_{k=j}^{r-1} \log_2\hspace{0pt}(n-k)= -\sum_{j=k}^{r-1} \log_2\left(1-\f{k}{n}\right) \le 2\log_2 \hspace{-2pt}e\sum_{k=1}^{r-1} \f{k}{n} \le \f{r^2 \log_2 \hspace{-2pt}e }{n}.
\end{align}
Next note that the function $F_1 (x):= \f{1}{2}\log_e\hspace{-2pt}x -x$ is decreasing in $x$ for all $x \ge 1/2$. Thus
\beq\label{eq:prelim_bd_p2}
\f{1}{2}(\log_2\hspace{-2pt}r-\log_2\hspace{-2pt}j)  -(r-j)\log_2\hspace{-2pt}e =\log_2\hspace{-2pt}e \left[ \f{1}{2} \left(\log_e\hspace{-2pt}r- \log_e\hspace{-2pt}j\right) - (r-j)\right] \le 0. 
\eeq
Thus
\begin{align}\label{eq:prelim_bd}
\log_2 \left(\f{|\Bad_{j,r}^{\ul{\xi}}|}{n_{\Graph}({\sf H}_r)}\right)  & \le  \log_2\hspace{-2pt}4+ \log_2 (\sqrt{2 \pi} e) + \f{r^2 \log_2 \hspace{-2pt}e }{n}+ r \log_2\hspace{-2pt}r- j \log_2\hspace{-2pt}j \notag\\
& \qquad \qquad  +(r-j)\left[ \bar{C}_0 \log_2\hspace{-2pt} \log_2\hspace{-2pt}n +\log_2 \hspace{-2pt}\Upsilon_n(n/2^r,\delta)\right]+ \left[\sum_{k=j}^{r-1} k\right].
\end{align}
Recalling the definition of $\Upsilon_n(\cdot,\cdot)$ we note that
\[
\log_2\hspace{-2pt} \Upsilon_n(n/2^r,\delta) \le \max\{(r-\log_2\hspace{-2pt}n) , (\delta-1) \log_2 \hspace{-2pt}n\}.
\]
Thus, if $r \le ((1-\delta)\wedge \f{1}{2})\log_2\hspace{-2pt}n - (\bar{C}_0+12) \log_2\hspace{-1pt}\hspace{-2pt}\log_2 \hspace{-2pt}n$, we have $\log_2\hspace{-2pt}\Upsilon_n(n/2^r,\delta) \le -r-(\bar{C}_0+12)\log_2\hspace{-1pt}\log_2\hspace{-2pt}n$. Therefore noting that 
\beq\label{eq:prelim_bd_p3}
 \sum_{k=j}^{r-1}k \le \f{(r+j)(r-j)}{2}\le r(r-j),
\eeq
from \eqref{eq:prelim_bd} we deduce
\begin{align}
\log_2 \left(\f{|\Bad_{j,r}^{\ul{\xi}}|}{n_{\Graph}({\sf H}_r)}\right) & \le r \log_2\hspace{-2pt}r- j \log_2\hspace{-2pt}j - 11(r-j)\log_2 \log_2\hspace{-2pt}n, \notag
\end{align}
for all large $n$. Hence, to complete the proof it only remains to show that  $r \log_2\hspace{-2pt}r- j \log_2\hspace{-2pt}j  \le 2 \log_2\log_2 \hspace{-2pt}n(r-j)$.

To this end, fixing $r, n$, denote $F_2(x):= x \log_2\hspace{-2pt}x.$	
Using the Mean-Value Theorem, and recalling the fact that $r \le \log_2\hspace{-2pt}n$, we note that 
\[
r \log_2\hspace{-2pt} r - j \log_2\hspace{-2pt} j \le \sup_{x \in [1,r]} \left(\log_2\hspace{-2pt}e + \log_2\hspace{-2pt}x\right) (r-j) \le \left(\log_2\hspace{-2pt}e + \log_2\hspace{-2pt}\log_2\hspace{-2pt}n\right)(r-j).
\]
This completes the proof.
\end{proof}


Building on Lemma \ref{lem:Bad_bound} we now derive a bound on $\sum_{\ul{\nu} \in \mathcal{H}_r} f_r(\ul{\nu},\ul{\xi})$ and $\sum_{\ul{\nu} \in \mathcal{H}_r}{f_r(\ul{\nu},\ul{\xi})(f_r(\ul{\nu},\ul{\xi}) -1)}$ which will later be used in the proof of Lemma \ref{lem:var_type_bd}.

\begin{lem}\label{goodbad}
Let $\{\Graph_n\}_{n \in \N}$ be a sequence of graphs satisfying \eqref{eq:A1_1/2}-\eqref{eq:A2_1/2}. Fix any positive integer $ r$, and let ${\sf H}_r$ be a graph on $r$ vertices. Assume that 
\[
n_\Graph({\sf H_r}) \ge \f{1}{2} \times \f{(n)_r}{|\mathrm{Aut}({\sf H}_r)|2^{{r \choose 2}}} \quad \text{ and } \quad n_\Graph({\sf H_j}) \le {2} \times \f{(n)_j}{|\mathrm{Aut}({\sf H}_j)|2^{{j \choose 2}}},
\]
for all graphs ${\sf H}_j$ on $j$ vertices and all $j < r$. Fix any $\ul{\xi} =\{\xi_1,\xi_2,\ldots,\xi_r\}\in \{0,1\}^r$. Then, there exists a large absolute constant $C_0$, such that for any $r \le ((1-\delta) \wedge \f{1}{2}) \log_2\hspace{-2pt}n - C_0 \log_2 \log_2\hspace{-2pt}n$, we have


{\beq\label{eq:term2}
\left|\sum_{\ul{\nu} \in \mathcal{H}_r} f_r(\ul{\nu},\ul{\xi})- n_\Graph({ \sf{H}}_r)\f{n}{2^r}\right| \le 13\wt{C}r  n_\Graph({ \sf H}_r) \f{n}{(\log_2\hspace{-2pt} n)^7}\left(\f{1}{2}\right)^r,
\eeq}
and
\beq\label{eq:seconorder}
\sum_{\ul{\nu} \in \mathcal{H}_r}{f_r(\ul{\nu},\ul{\xi})(f_r(\ul{\nu},\ul{\xi}) -1)} \le n_\Graph({\sf H}_r) \left(\f{n}{2^r}\right)^2\left(1+ \f{40\wt{C} r}{(\log_2\hspace{-2pt}n)^7}\right).
\eeq

\end{lem}

\bigskip

{We prove this lemma using Lemma \ref{lem:Bad_bound} but since we cannot apply it directly we resort to the following argument. Roughly the idea is to find a re-ordering, for any $r$ tuples such that for some $j(\leq r)$ we can safely plug in the correct estimate for the first $j$ elements from Lemma \ref{common-nbr-allgood} and for the last $r-j$ terms can be controlloed by error estimates obtained from Lemma \ref{lem:Bad_bound}.  Below we provide the formal argument.}

\begin{proof}

We begin by claiming that given any $\ul{v} \in \cH_r$, there exists a reordering $\{\hat{v}_1,\hat{v}_2,\ldots, \hat{v}_r\}$, $\{\hat{\xi}_1,\hat{\xi}_2,\ldots,\hat{\xi}_r\}$, and $3 \le j \le r$ such that $\hat{v}_i \in \mathrm{Good}^{\hat{\ul{\xi}}^i}(\hat{v}_1,\hat{v}_2,\ldots,\hat{v}_{i-1})$ for all $i=3,4,\ldots,j$, and $\hat{v}_k \in \mathrm{Bad}^{\hat{\ul{\xi}}^j}(\hat{v}_1,\hat{v}_2,\ldots,\hat{v}_j)$ for all $k=j+1,\ldots,r$. 

Indeed, choose $\hat{v}_1$ and $\hat{v}_2$ arbitrarily, and choose $\hat{\xi}_1$ and $\hat{\xi}_2$ accordingly. That is, if $\hat{v}_1=v_{i_1}$, and $\hat{v}_2=v_{i_2}$, for some indices $i_1$ and $i_2$, then set $\hat{\xi}_1=\xi_{i_1}$, and $\hat{\xi}_2=\xi_{i_2}$. Next, partition the set $\cA_2:=\{v_1,v_2,\ldots,v_r\} \setminus \{\hat{v}_1,\hat{v}_2\}= \cA_2^{(1)} \cup \cA_2^{(0)}$, where $\cA_2^{(1)}:=\{v \in \cA_2: \xi_v =1\}$, and  $\cA_2^{(0)}:=\{v \in \cA_2: \xi_v =0\}$. For $\xi \in\{0,1\}$, if there is a vertex $v \in \cA_2^{(\xi)}$ such that $v \in \mathrm{Good}^{\hat{\ul{\xi}}^3}(\hat{v}_1,\hat{v}_2)$, where $\hat{\ul{\xi}}^3=\{\hat{\xi}_1,\hat{\xi}_2,\xi\}$, then set $\hat{v}_3=v$, and $\hat{\xi}_3=\xi$. If there is more than choice, choose one of them arbitrarily. Now continue by induction. That is, having chosen $\hat{v}_1, \hat{v}_2,\ldots,\hat{v}_{i-1}$ partition the set $\cA_{i-1}:=\{v_1,v_2,\ldots,v_r\} \setminus \{\hat{v}_1,\hat{v}_2, \ldots, \hat{v}_{i-1}\}= \cA_{i-1}^{(1)} \cup \cA_{i-1}^{(0)}$, where $\cA_{i-1}^{(1)}:=\{v \in \cA_{i-1}: \xi_v =1\}$, and  $\cA_{i-1}^{(0)}:=\{v \in \cA_{i-1}: \xi_v =0\}$. Again for some $\xi \in\{0,1\}$, if there is a vertex $v \in \cA_{i-1}^{(\xi)}$ such that $v \in \mathrm{Good}^{\hat{\ul{\xi}}^i}(\hat{v}_1,\hat{v}_2,\ldots,\hat{v}_{i-1} )$, where $\hat{\ul{\xi}}^i=\{\hat{\xi}_1,\hat{\xi}_2,\ldots,\hat{\xi}_{i-1},\xi\}$, then set $\hat{v}_i=v$, and $\hat{\xi}_i=\xi$. 

Note that if the above construction stops at $j$, then it is obvious that $\hat{v}_i \in \mathrm{Good}^{\hat{\ul{\xi}}^i}(\hat{v}_1,\hat{v}_2,\ldots,\hat{v}_{i-1})$ for all $i=3,4,\ldots,j$, and $\hat{v}_k \in \mathrm{Bad}^{\hat{\ul{\xi}}^j}(\hat{v}_1,\hat{v}_2,\ldots,\hat{v}_j)$ for all $k=j+1,\ldots,r$, and hence we have our claim.

For brevity, let us denote $\mathrm{Bad}_j(\mathcal{H}_r)$ to be the collection of all ${\ul{v}} \in \cH_r$, such that for some re-ordering $\widehat{\ul{\nu}}$ of $\ul{v}$ we have $\hat{v}_i \in \mathrm{Good}^{\hat{\ul{\xi}}^i}(\hat{v}_1,\hat{v}_2,\ldots,\hat{v}_{i-1})$ for all $i=3,4,\ldots,j$, and $\hat{v}_k \in \mathrm{Bad}^{\hat{\ul{\xi}}^j}(\hat{v}_1,\hat{v}_2,\ldots,\hat{v}_j)$ for all $k=j+1,\ldots,r$. Given a $\ul{v} \in \cH_r$, it may happen that $\ul{v}$ belong to $\mathrm{Bad}_j(\mathcal{H}_r)$ for two different indices $j$. To avoid confusion we choose the smallest index $j$. When $j=r$ we denote the corresponding set by $\mathrm{Good}_r(\mathcal{H}_r)$ instead of $\mathrm{Bad}_r(\mathcal{H}_r)$. Equipped with these notations, we now note that
\begin{align}
\sum_{\ul{\nu} \in \mathcal{H}_r}{f_r(\ul{\nu},\ul{\xi})} &= \sum_{{\ul{\nu}} \in \mathrm{Good}_r(\mathcal{H}_r)}{f_r({\ul{\nu}},\ul{\xi})}+ \sum_{j=3}^{r-1}\sum_{{\ul{\nu}} \in \mathrm{Bad}_j(\mathcal{H}_r)}{f_r({\ul{\nu}},\ul{\xi})}\label{eq:good+bad}.
\end{align}
We show the first term in the \abbr{RHS} of \eqref{eq:good+bad} is the dominant term, and other term is negligible. First we find a good estimate on the dominant term. 

To this end, from  Lemma \ref{common-nbr-allgood}, using the union bound, and choosing $\bar{C}_0\ge 16$, we obtain that
\beq\label{eq:term1}
\left|\sum_{{\ul{\nu}} \in \mathrm{Good}_r(\mathcal{H}_r)}{|\sN_{\nu_1}^{\xi_1} \cap \sN_{\nu_2}^{\xi_2} \cap \cdots \cap \sN_{\nu_r}^{\xi_r}|}- |\mathrm{Good}_r(\mathcal{H}_r)| \times \f{n}{2^r} \right| \le 
3\wt{C}  |\mathrm{Good}_r(\mathcal{H}_r)| (\log_2\hspace{-2pt}n)^{-7}  \times \f{n}{2^r}.
\eeq
On the other hand, if $\ul{\nu} \in \mathrm{Bad}_j(\mathcal{H}_r)$, then $\ul{v}\in \mathrm{Bad}^{\ul{\hat{\xi}}^j}({\sf H}_j, {\sf H}_r)$ for some $\ul{\hat{\xi}}^j \in \{0,1\}^j$, and some sub-graph ${\sf H}_j$ of the graph ${\sf H}_r$, induced by $j$ vertices. Given any graph ${\sf H}_r$ on $r$ vertices, there are at most $r^{(r-j)}$ many induced sub-graphs on $j$ vertices, and given any $\ul{\xi}$ there are at most $r^{(r-j)}$ many choices of $\ul{\hat{\xi}}^j$. Since $r \le \log_2\hspace{-2pt}n$, from Lemma \ref{lem:Bad_bound} we deduce 
\beq\label{eq:bad_j_H_r}
|\mathrm{Bad}_j(\mathcal{H}_r)| \le \f{n_\Graph({\sf H_r})} {(\log_2\hspace{-2pt}n)^{7(r-j)}}.
\eeq
Now taking a union over $j=3,4,\ldots,r-1$, we further obtain
\[
 \left|\mathrm{Good}_r(\mathcal{H}_r)- n_\Graph({\sf H_r}) \right| \le 2 \f{n_\Graph({\sf H}_r) }{(\log_2\hspace{-2pt}n)^7}.
\]
This together with \eqref{eq:term1}, upon recalling the definition of $f_r(\ul{\nu},\ul{\xi})$, now implies that
\beq\label{eq:termIdone1}
\left|\sum_{{\ul{\nu}} \in \mathrm{Good}_r(\mathcal{H}_r)}f_r(\ul{\nu},\ul{\xi})- n_\Graph({\sf H_r}) \times \f{n}{2^r} \right| \le 
5\wt{C} r n_\Graph({\sf H_r})  \f{n}{(\log_2\hspace{-2pt}n)^7}\times \left(\f{1}{2}\right)^r.
\eeq
This provides the necessary error bound for the first term in the \abbr{RHS} of \eqref{eq:good+bad}. We now bound the second term appearing in the \abbr{RHS} of \eqref{eq:good+bad}. 

Proceeding to control the second term, we first observe that ${f_r(\ul{\nu},\ul{\xi})}$ is invariant under the permutation of the coordinates of $\ul{v}$ (with same permutation applied on $\ul{\xi}$). Thus
\[
\sum_{j=3}^{r-1}\sum_{{\ul{\nu}} \in \mathrm{Bad}_j(\mathcal{H}_r)}{f_r({\ul{\nu}},\ul{\xi})}=  \sum_{j=3}^{r-1}\sum_{{\ul{\nu}} \in \mathrm{Bad}_j(\mathcal{H}_r)}{f_r(\widehat{\ul{\nu}},\wh{\ul{\xi}})}.
\]
Now note that if $\ul{\nu} \in \Bad_j(\mathcal{H}_r)$ then for the corresponding $\widehat{\ul{\nu}}$ we have $\hat{\nu}_i \in \mathrm{Good}^{\hat{\ul{\xi}}^i}(\hat{\nu}_1,\hat{\nu}_2,\ldots,\hat{\nu}_{i-1})$, for all $i=3,4,\ldots,j$. Therefore using Lemma \ref{common-nbr-allgood}, we obtain that 
\[
{|\sN_{\hat{\nu}_1}^{\hat{\xi}_1} \cap \sN_{\hat{\nu}_2}^{\hat{\xi}_2} \cap \cdots \cap \sN_{\hat{\nu}_r}^{\hat{\xi}_r}|} \le {|\sN_{\hat{\nu}_1}^{\hat{\xi}_1} \cap \sN_{\hat{\nu}_2}^{\hat{\xi}_2} \cap \cdots \cap \sN_{\hat{\nu}_j}^{\hat{\xi}_j}|}\le 2 \times \f{n}{2^j}.
\]
This together with \eqref{eq:bad_j_H_r}, now implies
\begin{align}
\sum_{j=3}^{r-1}\sum_{{\ul{\nu}} \in \mathrm{Bad}_j(\mathcal{H}_r)}{f_r({\ul{\nu}},\ul{\xi})} & =  \sum_{j=3}^{r-1}\sum_{{\ul{\nu}} \in \mathrm{Bad}_j(\mathcal{H}_r)}{|\sN_{\hat{\nu}_1}^{\hat{\xi}_1} \cap \sN_{\hat{\nu}_2}^{\hat{\xi}_2} \cap \cdots \cap \sN_{\hat{\nu}_r}^{\hat{\xi}_r}|} \notag\\
 &\le \sum_{j =3}^{r} 2|\Bad_j(\mathcal{H}_r)| \times \f{n}{2^j} \notag\\
 & \le 2\sum_{j <r} \f{n_\Graph({\sf H}_r)}{(\log_2\hspace{-2pt}n)^{7(r-j)}}\times \f{n}{2^j}\notag\\
 &\le  2\f{n}{2^r}\sum_{j <r} n_\Graph({\sf H}_r) \left(2(\log_2\hspace{-2pt}n)^{-7}\right)^{(r-j)} \le 8rn_\Graph({\sf H}_r) \f{n}{(\log_2\hspace{-2pt}n)^7}\left(\f{1}{2}\right)^r .\label{eq:termIIdone}
\end{align}
Thus combining \eqref{eq:termIdone1}-\eqref{eq:termIIdone}, from \eqref{eq:good+bad} we deduce that
\beq
\left|\sum_{\ul{\nu} \in \mathcal{H}_r} f_r(\ul{\nu},\xi)- n_\Graph({\sf H}_r)\f{n}{2^r}\right| \le 13\wt{C} r n_\Graph({\sf H_r})  \f{n}{(\log_2\hspace{-2pt}n)^7}\times \left(\f{1}{2}\right)^r.\notag
\eeq
This completes the proof of \eqref{eq:term2}. To prove \eqref{eq:seconorder} we proceed similarly as above. As before, we split the sum in two parts
\begin{align*}
\sum_{\nu \in \mathcal{H}_r}{f_r(\nu,\xi)(f_r(\nu,\xi) -1)} 
& =  \sum_{\ul{\nu} \in \mathrm{Good}_r(\mathcal{H}_r)}{f_r(\nu,\xi)(f_r(\nu,\xi) -1)}  +   \sum_{j=3}^{r-1}\sum_{{\ul{\nu}} \in \mathrm{Bad}_j(\mathcal{H}_r)} {f_r(\nu,\xi)(f_r(\nu,\xi) -1)} .
\end{align*}
Proceeding as in \eqref{eq:termIIdone} we see that
\begin{align}\label{eq:term1b}
\sum_{j=3}^{r-1}\sum_{{\ul{\nu}} \in \mathrm{Bad}_j(\mathcal{H}_r)}  {f_r(\nu,\xi)(f_r(\nu,\xi) -1)}  
  \le \sum_{j=3}^{r-1}\sum_{{\ul{\nu}} \in \mathrm{Bad}_j(\mathcal{H}_r)} {f_r(\widehat{\ul{\nu}},\wh{\ul{\xi}})}^2 
 \le &    4\sum_{j <r} \f{n_\Graph({\sf H}_r)}{(\log_2\hspace{-2pt}n)^{7(r-j)}}\times \left(\f{n}{2^j}\right)^2\notag\\
 \le & 32 rn_\Graph({\sf H}_r) \f{n^2}{(\log_2\hspace{-2pt}n)^7}\left(\f{1}{2}\right)^{2r}.
\end{align}
On the other hand, using \eqref{eq:good_set_estimates} we deduce
\begin{align}\label{eq:term1g}
\sum_{\ul{\nu} \in \mathrm{Good}_r(\mathcal{H}_r)} {f_r(\widehat{\ul{\nu}},\wh{\ul{\xi}})} \left({f_r(\widehat{\ul{\nu}},\wh{\ul{\xi}})}-1\right) \le n_\Graph({\sf H}_r) \left(\f{n}{2^r}\right)^2\left(1+ \f{3\wt{C} r}{(\log_2\hspace{-2pt}n)^7} \right)^2.
\end{align}
Combining \eqref{eq:term1b}-\eqref{eq:term1g} the proof of \eqref{eq:seconorder} completes.
\end{proof}

We next prove Lemma \ref{lem:f_r_bar_bound} where we obtain bounds on $\bar{f}_{r,\ul{\xi}}$.

\begin{proof}[Proof of Lemma \ref{lem:f_r_bar_bound}] 
Recall that 
\[
\bar{f}_{r,\ul{\xi}}= \f{1}{(n)_r} \sum_{\ul{\nu}} f_r(\ul{\nu},\ul{\xi}),
\]
where
\[
{f_r(\underline{\nu},\ul{\xi})} = {\left|\sN_{\nu_1}^{\xi_1} \cap \sN_{\nu_2}^{\xi_2} \cap \cdots \cap \sN_{\nu_r}^{\xi_r}\right|}
\]
(see Definition \ref{dfn:f_r_bar_xi}). Therefore
\[
\bar{f}_{r, \ul{\xi}}= \f{1}{(n)_r}\sum_{\ul{\nu}\in [n]^r}\sum_{v \in [n]} \prod_{i=1}^r a_{vv_i}^{\xi_i}.
\]
Now interchanging the summations we arrive at \eqref{eq:f_r_xi_bound}. To prove \eqref{eq:avgbd} we begin by observing the following inequality:
\beq
 e^{- \f{\ell(\ell-1)}{m}} = e^{-2\sum_{j=0}^{\ell-1}\f{j}{m}} \le \prod_{j=0}^{\ell-1} \left(1-\f{j}{m}\right) \le e^{-\sum_{j=0}^{\ell-1}\f{j}{m}}= e^{- \f{\ell(\ell-1)}{2m}}, \quad \text{ for all } \ell <\f{m}{2},\label{eq:fact}
\eeq
 For ease of writing, let us assume that $n(\ul{\xi})=k$. Then, using \eqref{eq:fact}, from \eqref{eq:bd_generalized_nbhd}, upon applying the triangle inequality, we deduce
\begin{align}
\bar{f}_{r,\ul{\xi}} & \le \f{n}{(n)_r} \left(\frac{n}{2} + 3C n^{\delta} \right)^k \left(\frac{n}{2} + 3C n^{\delta} \right)^{r-k} \notag \\
&= n \left(\frac{1}{2}+ 3C n^{\delta-1} \right)^r \left[\prod_{j=0}^{r-1}\left(1-\f{j}{n}\right)\right]^{-1}  \le  n \left(\frac{1}{2} +3 Cn ^{\delta-1}\right)^r e^{\f{r^2}{n}}.\notag
\end{align}
Since $(1+x)^r \le 1+2r x$ for $rx \le 1$, and $e^x \le 1+2x$ for $x \le \log 2$, we further obtain that
\beq\label{avgubd}
\bar{f}_{r,\ul{\xi}}  \le n \left(\frac{1}{2} \right)^r \left(1+ 12 C  r n^{\delta-1}\right)\left(1+\f{2r^2}{n}\right),
\eeq
which proves the upper bound in \eqref{eq:avgbd}.


To obtain the lower bound of $\bar{f}_{r,\xi}$ we proceed similarly to deduce that
\begin{align}
\bar{f}_{r,\ul{\xi}} & \geq \f{n}{(n)_r} \left(\frac{n}{2} - 3C n^{\delta} \right)^k \left(\frac{n}{2} - 3C n^{\delta} \right)^{r-k} \notag \\
&= n \left(\frac{1}{2}- 3C n^{\delta-1} \right)^r \left[\prod_{j=0}^{r-1}\left(1-\f{j}{n}\right)\right]^{-1}  \ge  n \left(\frac{1}{2} -3 Cn ^{\delta-1}\right)^r.\notag
\end{align}
Using the facts that for $x \in [0,1/r)$, we have $(1-x)^r \ge 1-rx$, we further have,
\begin{equation}\label{avglbd}
\bar{f}_{r,\ul{\xi}} \geq n \left(\frac{1}{2} \right)^r\left(1-6 C  r n^{\delta-1}\right).
\end{equation}
This completes the proof.
\end{proof}

\bigskip

Now combining the previous results we complete the proof of Lemma \ref{lem:var_type_bd}.

\begin{proof}[Proof of Lemma \ref{lem:var_type_bd}]
To prove the lemma we use Cauchy-Schwarz inequality as follows 
\begin{align}
\left|\sum_{\ul{\nu} \in \mathcal{H}_r}\{f_{r}(\ul{\nu},\ul{\xi}) - \bar{f}_{r,\ul{\xi}}\}\right|^2 &\leq n_\Graph({\sf H}_r)\sum_{\ul{\nu} \in \mathcal{H}_r}\{f_{r}(\ul{\nu},\ul{\xi}) - \bar{f}_{r,\ul{\xi}}\}^2 \notag\\
&=n_\Graph({\sf H}_r)\left[ \sum_{\ul{\nu} \in \mathcal{H}_r}{f_r^2(\ul{\nu},\ul{\xi})} - 2 \bar{f}_{r,\ul{\xi}}\sum_{\nu \in \mathcal{H}_r}{f_r(\ul{\nu},\ul{\xi})} +n_\Graph({\sf H}_r)\bar{f}_{r,\ul{\xi}}^2\right] \notag\\
&= n_\Graph({\sf H}_r)\bigg[\sum_{\ul{\nu} \in \mathcal{H}_r}{f_r(\ul{\nu},\ul{\xi})(f_r(\ul{\nu},\ul{\xi}) -1)} + \sum_{\ul{\nu} \in \mathcal{H}_r}f_{r}(\ul{\nu},\ul{\xi})\notag\\
& \qquad \qquad \qquad\qquad \qquad \qquad  +n_\Graph({\sf H}_r)\bar{f}_{r,\ul{\xi}}^2 - 2\bar{f}_{r,\ul{\xi}}\sum_{\ul{\nu} \in \mathcal{H}_r}{f_r(\ul{\nu},\ul{\xi})}\bigg]. \label{eq:break}
\end{align}
Note that we have already obtained bounds on the first two terms inside the square bracket from Lemma \ref{goodbad}. The bound on the third term was derived in Lemma \ref{lem:f_r_bar_bound}. We obtain bounds on the last term follows upon combining Lemma \ref{goodbad} and Lemma \ref{lem:f_r_bar_bound}. We then plug in these estimates one by one, and combine them together to finish the proof.

To this end, we start controlling the last term of \eqref{eq:break}. Using Lemma \ref{goodbad} and Lemma \ref{lem:f_r_bar_bound} we see that

\begin{align*}
\bar{f}_{r,\ul{\xi}} \sum_{\ul{\nu} \in \mathcal{H}_r} f_r(\ul{\nu},\xi) \ge n \left(\frac{1}{2} \right)^r (1-12 C r n^{\delta-1})\left(1- \f{2 r^2}{n}\right)  n_\Graph({\sf H}_r) \f{n}{2^r} \left(1- \f{13 \wt{C}r}{(\log_2\hspace{-2pt}n)^7}\right).
\end{align*}
Therefore
\begin{align}
\bar{f}_{r,\ul{\xi}} \sum_{\ul{\nu} \in \mathcal{H}_r} f_r(\ul{\nu},\xi) & \ge n^2n_\Graph({\sf H}_r)\left(\f{1}{2}\right)^{2r}\left(1- \f{27 \wt{C}r}{(\log_2\hspace{-2pt}n)^7}\right). \label{eq:term4}
\end{align}
Now combining \eqref{eq:term2}-\eqref{eq:seconorder}, \eqref{eq:avgbd}, and \eqref{eq:term4}, from \eqref{eq:break} we obtain
\begin{align}
\f{\left|\sum_{\nu \in \mathcal{H}_r}\{f_{r}(\ul{\nu},\ul{\xi}) - \bar{f}_{r,\ul{\xi}}\}\right|^2}{n_\Graph({\sf H}_r)^2} & \le    \left(\f{n}{2^r}\right)^2\left(1+ \f{40 \wt{C}r}{(\log_2\hspace{-2pt}n)^7} \right)+ \f{n}{2^r}\left(1+ \f{13\wt{C} r}{(\log_2\hspace{-2pt}n)^7} \right) \notag\\
& \quad +  \left(\frac{n}{2^r} \right)^2\left(1+12 C  r n^{\delta-1}\right)^2\left(1+\f{2r^2}{n}\right)^2-  2\left(\frac{n}{2^r} \right)^2\left(1- \f{27 \wt{C}r}{(\log_2\hspace{-2pt}n)^7}\right).\label{eq:combine}
\end{align}
Simplifying  the \abbr{RHS} of \eqref{eq:combine}, and noting that $r \le \log_2\hspace{-2pt}n$, the proof completes.
\end{proof}


\section{Proofs of Theorem \ref{thm:main} and Theorem \ref{thm:main_improve}}
\label{sec:proof_thm}

In this section we prove our main result Theorem \ref{thm:main} followed by the proof of Theorem \ref{thm:main_improve}. 

\medskip

The proof of Theorem \ref{thm:main} uses the same ideas as in the proof of Proposition \ref{prop:main}. Since in Theorem \ref{thm:main} we allow any $p\in (0,1)$, unlike Proposition \ref{prop:main} the argument cannot be symmetric with respect to the presence and absence of an edge. This calls for changes in some definitions and some of the steps in the proof of Proposition \ref{prop:main}. Below we explain the changes and modifications necessary to extend the proof of Proposition \ref{prop:main} to establish Theorem \ref{thm:main}.

Fix a positive integer $r$ and let ${\sf H}_r$ be a graph with vertex set $[r]$. Further, for $j=2,3,\ldots,r-1$, let ${\sf H}_j$ be the sub-graph of ${\sf H}_r$ induced by the vertices $[j]$. Recall that in Proposition \ref{prop:main} we showed that the number of induced isomorphic copies of ${\sf H}_r$ is approximately same as that of an Erd\H{o}s-R\'{e}yni graph by showing that the same is true for ${\sf H}_j$ for every $j=2,3,\ldots,r-1$, and given any such isomorphic copy, ${\sf H}_{j,n}$, of ${\sf H}_j$ in $\Graph_n$, the number of common generalized neighbors (recall Definition \ref{nbrdfn}) of the vertices of ${\sf H}_{j,n}$ is about $n/2^j$. Putting these two observations together we propagated the error estimates.

Since in the set-up of Theorem \ref{thm:main} the presence  and absence of an edge do not have the same probability, the number of common generalized neighbors of the vertices of ${\sf H}_{j,n}$ should depend on the number of edges present in that common generalized neighborhood. Therefore we cannot use Definition \ref{dfn:good_dfn} and Definition \ref{dfn:bad_define} to define {\em Good} and {\em Bad} vertices. To reflect the fact that $p \ne 1/2$ we adapt those definitions as follows. For ease of writing let us denote $q:=1-p$. 
\begin{dfn}
For any given set $B \subset [n]$, and $\xi \in \{0,1\}$, define
\[
\mathrm{Good}^{\xi,p}(B):= \left\{ \nu \in [n]: \left|\sN_\nu^{B,\xi} - {|B|}p^\xi q^{1-\xi}\right| \le \wt{C} |B| p^\xi q^{1-\xi} n^{\vep(\delta-1)/2}\right\},
\]
where $\vep= \frac{\bar{C}_0 \log \log n}{(1-\delta)\log n}$ for some large constant $\bar{C}_0$. 

Equipped with the definition of $\mathrm{Good}^{\xi,p}(B)$, similar to Definition \ref{dfn:good_dfn} we next define $\mathrm{Good}^{\ul{\xi},p}$ and $\mathrm{Bad}^{\ul{\wt{\xi}},p}(v_1,v_2,\ldots,v_m)$, where $\ul{\xi}:=\{\xi_1,\xi_2,\ldots,\xi_m,\xi_{m+1}\} \in \{0,1\}^{m+1}$ and $\ul{\t{\xi}}:=\{\xi_1,\xi_2,\ldots,\xi_m,\}$. Proceeding as in Definition \ref{dfn:bad_define} we also  define $\Bad^{\ul{\xi},p}({\sf H}_{m'}, {\sf H}_m)$.
\end{dfn}
Next we need to extend Lemma \ref{common-nbr-allgood}, Lemma \ref{lem:Bad_bound}, Lemma \ref{lem:f_r_bar_bound}, and Lemma \ref{lem:var_type_bd} to allow any $p \in (0,1)$. To this end, we note that Lemma \ref{common-nbr-allgood} extends to the following result.

\begin{lem}\label{common-nbr-allgood_p}
Let $\{\Graph_n\}_{n \in \N}$ be a sequence of graphs satisfying assumptions {\bf (A1)} and {\bf (A2)}. {Fix $r < \log n$,} and let $\ul{v} \in [n]^r$ and $\ul{\xi} \in \{0,1\}^r$. Let $n(\ul{\xi}):=|\{i \in [r]: \xi_i =1\}|$. If $\nu_j \in \Good^{\ul{\xi}^{j},p}(\nu_1, \nu_2,...,\nu_{j-1})$,  for all $j =3,4,\ldots,r$, and $\bar{C}_0 \ge 4$, then 
\beq\label{eq:good_set_estimates_p}
\left|\left|\sN_{\nu_1}^{\xi_1} \cap \sN_{\nu_2}^{\xi_2} \cap \cdots \cap \sN_{\nu_r}^{\xi_r}\right|- {n}p^{n(\ul{\xi})}q^{r-n(\ul{\xi})}\right|  \le 3\wt{C} (\log n)^{-(\bar{C}_0/2-1)}   {n}p^{n(\ul{\xi})}q^{r-n(\ul{\xi})},
\eeq
for all large $n$.
\end{lem}

\begin{proof}
Similar to the proof of Lemma \ref{common-nbr-allgood} we use induction. For any $j=2,3,\ldots,r$, let us denote $\ul{\xi}^j:=\{\xi_1,\xi_2,\ldots,\xi_j\}$. We consider two cases  $n(\ul{\xi}^j)=n(\ul{\xi}^{j-1})+1$ and $n(\ul{\xi}^j)=n(\ul{\xi}^{j-1})$ separately. Below we only provide the argument for $n(\ul{\xi}^j)=n(\ul{\xi}^{j-1})+1$ . Proof of the other case is same and hence omitted. 

Focusing on the case $n(\ul{\xi}^j)=n(\ul{\xi}^{j-1})+1$, using the triangle inequality we have
\begin{align}
&\left|\left|\sN_{\nu_1}^{\xi_1} \cap \sN_{\nu_2}^{\xi_2} \cap \cdots \cap \sN_{\nu_j}^{\xi_j}\right|- {n}p^{n(\ul{\xi}^j)}q^{j-n(\ul{\xi}^j)}\right| \notag \\
\le &   \left|\left|\sN_{\nu_1}^{\xi_1} \cap \sN_{\nu_2}^{\xi_2} \cap \cdots \cap \sN_{\nu_j}^{\xi_j}\right|- p\left|\sN_{\nu_1}^{\xi_1} \cap \sN_{\nu_2}^{\xi_2} \cap \cdots \cap \sN_{\nu_{j-1}}^{\xi_{j-1}}\right|\right| \notag\\
& \qquad \qquad \qquad \qquad \qquad\qquad + p\left|\left|\sN_{\nu_1}^{\xi_1} \cap \sN_{\nu_2}^{\xi_2} \cap \cdots \cap \sN_{\nu_{j-1}}^{\xi_{j-1}}\right|- {n}p^{n(\ul{\xi}^{j-1})}q^{j-1-n(\ul{\xi}^{j-1})}\right|.\notag 
\end{align}
Since $\nu_j \in \Good^{\ul{\xi}^{j},p}(\nu_1, \nu_2,...,\nu_{j-1})$, we also have
\begin{align}
 &\left|\left|\sN_{\nu_1}^{\xi_1} \cap \sN_{\nu_2}^{\xi_2} \cap \cdots \cap \sN_{\nu_j}^{\xi_j}\right|- p\left|\sN_{\nu_1}^{\xi_1} \cap \sN_{\nu_2}^{\xi_2} \cap \cdots \cap \sN_{\nu_{j-1}}^{\xi_{j-1}}\right| 
 \right| \notag\\
 & \qquad \qquad \qquad\qquad \qquad \qquad \qquad \qquad \le \wt{C}p n^{\vep(\delta-1)/2} \left|\sN_{\nu_1}^{\xi_1} \cap \sN_{\nu_2}^{\xi_2} \cap \cdots \cap \sN_{\nu_{j-1}}^{\xi_{j-1}}\right|.\notag
\end{align}
Now using the induction hypothesis and proceeding as in Lemma \ref{common-nbr-allgood} we complete the proof.
\end{proof}

\bigskip
The next step to prove Theorem \ref{thm:main} is to extend Lemma \ref{lem:Bad_bound} for any $p \in (0,1)$. Recall that a key ingredient in the proof of Lemma \ref{lem:Bad_bound} is the variance bound obtained in Lemma \ref{lem:var_N_nu}. Here using assumptions {\bf (A1)} and {\bf (A2)}, proceeding same as in the proof of Lemma \ref{lem:var_N_nu}, one can easily obtain the following variance bound. We omit its proof. For clarity of presentation, hereafter, without loss of generality,  we assume that $p \ge 1/2$. For $p <1/2$, interchanging the roles of $p$ and $q$, one can obtain the same conclusions.  
\begin{lem}\label{lem:var_N_nu_p}
Let $\{\Graph_n\}_{n \in \N}$ be a sequence of graphs satisfying assumptions {\bf (A1)} and {\bf (A2)}. Then for any set $B \subset [n]$ and $\xi\in \{0,1\}$,
\[
\mathrm{Var} \left(\sN_\nu^{B,\xi}\right) \le |B| pq + \bar{C} {|B|^2 n^{\delta-1}},
\]
for some constant $\bar{C}$, depending only on $C$.
\end{lem}

\medskip
Building on Lemma \ref{lem:var_N_nu_p} we extend Lemma \ref{lem:Bad_bound} to obtain the foliowing result:

\begin{lem}\label{lem:Bad_bound_p}
Let $\{\Graph_n\}_{n \in \N}$ be a sequence of graphs satisfying {\bf (A1)} and {\bf (A2)}. Fix any two positive integers $j< r$, and let ${\sf H}_r$ be a graph on $r$ vertices. Fix ${\sf H}_j$ any one of the sub-graphs of ${\sf H}_r$ induced by $j$ vertices. Assume that 
\[
n_\Graph({\sf H_r}) \ge \f{1}{2} \times \f{(n)_r}{|\mathrm{Aut}({\sf H}_r)|}\left(\f{p}{q}\right)^{|E({\sf H}_r)|}q^{{r \choose 2}} \quad \text{ and } \quad n_\Graph({\sf H_j}) \le 2 \times \f{(n)_j}{|\mathrm{Aut}({\sf H}_j)|} \left(\f{p}{q}\right)^{|E({\sf H}_j)|}q^{{j \choose 2}}. 
\]
There exists a large positive constant $C_0$, depending only on $p$ such that, for any given $\ul{\xi}=\{\xi_1,\xi_2,\ldots,\xi_j\} \in \{0,1\}^j$, and $r \le ((1-\delta)\wedge \f{1}{2})(\log n/\log (1/q))- C_0 \log_2 \log_2 \hspace{-2pt}n$, we have 
\[
\left|\Bad^{\ul{\xi}}({\sf H}_{j}, {\sf H}_r)\right| \le \f{n_\Graph({\sf H_r})}{\left(\log n \right)^{9(r-j)}},
\]
for all large $n$.
\end{lem}

\begin{proof}
First using the variance bound from Lemma \ref{lem:var_N_nu_p} proceeding as in Lemma \ref{lem:good_B} we obtain
\[
\left|\left\{ \nu \in [n]: \left|\sN_\nu^{B,\xi} - {|B|}p^\xi q^{1-\xi}\right| > \wt{C} |B| p^\xi q^{1-\xi}n^{\vep(\delta-1)/2}\right\}\right| \le n (\log n)^{\bar{C}_0}\Upsilon_n^p(|B|,\delta),
\]
where $\Upsilon_n^p(x,\delta):=\frac{2 x^{-1}+q^{-2}n^{\delta-1}}{2}$. Further, using Lemma \ref{common-nbr-allgood_p} we note that 
\[
\left|\sN_{\hat{\nu}_1}^{\xi_1} \cap \sN_{\hat{\nu}_2}^{\xi_2} \cap \cdots \cap \sN_{\hat{\nu}_j}^{\xi_j}\right|\ge  \f{n}{2}p^{n(\ul{\xi})}q^{j-n(\ul{\xi})} \ge \f{n}{2} q^j \ge \f{n}{2}q^r,
\]
where in the second inequality above we use the fact that $p \ge 1/2$.
Therefore, applying Stirling's approximation and proceeding as in \eqref{eq:prelim_bd_p1}-\eqref{eq:prelim_bd_p2} we obtain
\begin{align}\label{eq:prelim_bd}
\log \left(\f{|\Bad_{j,r}^{\ul{\xi}}|}{n_{\Graph}({\sf H}_r)}\right)  & \le  \log 4+ \log (\sqrt{2 \pi} e) + \f{r^2 }{n}+ r \log r- j \log j  + \log (p/q) (|E({\sf H}_j)| - |E({\sf H}_r)|)\notag\\
& \qquad \qquad  +(r-j)\left[ \bar{C}_0 \log \log n +\log \Upsilon_n((n/2)q^r,\delta)\right]+ \log(1/q)\left[\sum_{k=j}^{r-1} k\right].
\end{align}
Since
\[
\log \Upsilon_n^p((n/2)q^r,\delta) \le \max\{r\log(1/q)-\log n +2 \log 2 , (\delta-1) \log n + 2 \log (1/q)\},
\]
we note that if $\log(1/q) r \le ((1-\delta) \wedge \f{1}{2})\log n - (\bar{C}_0+13)\log \log n$, then $\log \Upsilon_n^p((n/2)q^r,\delta) \le - \log(1/q) r -(\bar{C}_0+12)\log \log n$. Further noting that $p \ge q$ and $|E({\sf H}_r)| \ge |E({\sf H}_j)|$  we proceed as in the proof of Lemma \ref{lem:Bad_bound} to arrive at the desired conclusion. 
\end{proof}

\bigskip
Next note that a key ingredient in the proof of Lemma \ref{lem:var_type_bd} is Lemma \ref{goodbad}. Therefore, we also need to find the analogue of Lemma \ref{goodbad} for general $p$.

\begin{lem}\label{goodbad_p}
Let $\{\Graph_n\}_{n \in \N}$ be a sequence of graphs satisfying {\bf (A1)} and {\bf (A2)}. Fix any positive integer $ r$, and let ${\sf H}_r$ be a graph on $r$ vertices. Assume that 
\[
n_\Graph({\sf H_r}) \ge \f{1}{2} \times \f{(n)_r}{|\mathrm{Aut}({\sf H}_r)|}\left(\f{p}{q}\right)^{|E({\sf H}_r)|}q^{{r \choose 2}} \quad \text{ and } \quad n_\Graph({\sf H_j}) \le 2 \times \f{(n)_j}{|\mathrm{Aut}({\sf H}_j)|} \left(\f{p}{q}\right)^{|E({\sf H}_j)|}q^{{j \choose 2}}. 
\]
for all graphs ${\sf H}_j$ on $j$ vertices and all $j < r$. Fix any $\ul{\xi} =\{\xi_1,\xi_2,\ldots,\xi_r\}\in \{0,1\}^r$. Then, there exists a large positive constant $C_0$, depending only on $p$, such that for any $r \le ((1-\delta) \wedge \f{1}{2}) (\log n / \log (1/q)) - C_0 \log \log n$, we have

{\beq\label{eq:term2_p}
\left|\sum_{\ul{\nu} \in \mathcal{H}_r} f_r(\ul{\nu},\ul{\xi})- n_\Graph({ \sf{H}}_r)n\left(\f{p}{q}\right)^{n(\ul{\xi})}q^r\right| \le 13\wt{C}r  n_\Graph({ \sf H}_r) \f{n}{(\log n)^7}\left(\f{p}{q}\right)^{n(\ul{\xi})}q^r,
\eeq}
and
\beq\label{eq:seconorder_p}
\sum_{\ul{\nu} \in \mathcal{H}_r}{f_r(\ul{\nu},\ul{\xi})(f_r(\ul{\nu},\ul{\xi}) -1)} \le n_\Graph({\sf H}_r) \left(n\left(\f{p}{q}\right)^{n(\ul{\xi})}q^r\right)^2\left(1+ \f{40\wt{C} r}{(\log_2\hspace{-2pt}n)^7}\right).
\eeq
\end{lem}

\begin{proof}
The proof is again mostly similar to that of Lemma \ref{goodbad}. Changes are required in only couple of places. Proceeding same as in the proof of Lemma \ref{goodbad} we see that \eqref{eq:termIdone1} extends to 
\beq
\left|\sum_{{\ul{\nu}} \in \mathrm{Good}_r(\mathcal{H}_r)}f_r(\ul{\nu},\ul{\xi})- n_\Graph({\sf H_r}) n\left(\f{p}{q}\right)^{n(\ul{\xi})}q^r \right| \le 
5\wt{C} r n_\Graph({\sf H_r})  (\log_2\hspace{-2pt}n)^{-7}n\left(\f{p}{q}\right)^{n(\ul{\xi})}q^r. \notag
\eeq
To extend \eqref{eq:termIIdone} we first note that for $\ul{\nu} \in \Bad_j(\mathcal{H}_r)$ there exists a relabeling $\widehat{\ul{\nu}}$ such that $\hat{\nu}_i \in \mathrm{Good}^{\hat{\ul{\xi}}^i}(\hat{\nu}_1,\hat{\nu}_2,\ldots,\hat{\nu}_{i-1})$, for all $i=3,4,\ldots,j$. Therefore using Lemma \ref{common-nbr-allgood_p}, we obtain that 
\[
{|\sN_{\hat{\nu}_1}^{\hat{\xi}_1} \cap \sN_{\hat{\nu}_2}^{\hat{\xi}_2} \cap \cdots \cap \sN_{\hat{\nu}_r}^{\hat{\xi}_r}|} \le {|\sN_{\hat{\nu}_1}^{\hat{\xi}_1} \cap \sN_{\hat{\nu}_2}^{\hat{\xi}_2} \cap \cdots \cap \sN_{\hat{\nu}_j}^{\hat{\xi}_j}|}\le 2 n \left(\f{p}{q}\right)^{n(\ul{\xi}^j)} q^j \le  2n\left(\f{p}{q}\right)^{n(\ul{\xi})} q^j,
\]
where in the last step we again used the fact that $p \ge 1/2$. Proceeding similar to \eqref{eq:termIIdone} we then deduce
\begin{align}
\sum_{j=3}^{r-1}\sum_{{\ul{\nu}} \in \mathrm{Bad}_j(\mathcal{H}_r)}{f_r({\ul{\nu}},\ul{\xi})} 
 &\le \sum_{j =3}^{r} 2|\Bad_j(\mathcal{H}_r)|   n\left(\f{p}{q}\right)^{n(\ul{\xi})} q^j \notag\\
 & \le 2\sum_{j <r} \f{n_\Graph({\sf H}_r)}{(\log_2\hspace{-2pt}n)^{7(r-j)}}n\left(\f{p}{q}\right)^{n(\ul{\xi})} q^j \le 8rn_\Graph({\sf H}_r) {(\log n)^{-7}}n\left(\f{p}{q}\right)^{n(\ul{\xi})} q^j .\notag
\end{align}
The rest of the proof requires similar adaptation. We omit the details.
\end{proof}

Imitating the proof of Lemma \ref{lem:f_r_bar_bound} we then obtain the following lemma.

\begin{lem}\label{lem:f_r_bar_bound_p}
For any $r$, let $\ul{\xi}:=\{\xi_1,\xi_2,\ldots,\xi_r\} \in \{0,1\}^r$. 
Then for any $r \le \log n$,
\beq\label{eq:avgbd}
n \left(\frac{p}{q} \right)^{n(\ul{\xi})} q^r \left(1-12 C  r n^{\delta-1}\right)\left(1-\f{2r^2}{n}\right) \leq \bar{f}_{r,\xi} \leq n \left(\frac{p}{q} \right)^{n(\ul{\xi})} q^r\left(1+12 C  r n^{\delta-1}\right)\left(1+\f{2r^2}{n}\right).
\eeq
\end{lem}

Since we have all the required ingredients we can prove Theorem \ref{thm:main} imitating the proof of Proposition \ref{prop:main}. The details are omitted.

\begin{rmk}
\label{rmk:cliq-ind}
If one is only interested in counting the number of cliques of independent sets in $\Graph_n$ satisfying {\bf (A1)}-{\bf (A2)} the from the proof of Theorem \ref{thm:main} it follows that one slightly improve the conclusion of Theorem \ref{thm:main}. For example, when considering the number of cliques one can replace $\gamma_p$ in \eqref{eq:thm_main_display} by $p^{-1}$, whereas for independent sets one can replace the same by $q^{-1}$. 
\end{rmk}

\begin{rmk}
	\label{rmk:delt-p}
Note that in the proof of Theorem \ref{thm:main} we obtained a rate of convergence for the \abbr{LHS} of \eqref{eq:thm_main_display}. Therefore we can easily consider $p,\delta \in (0,1)$ such that $\min\{p,q\},1-\delta \ra 0$ as $n \ra \infty$. Indeed, in the random geometric graph setting when the dimension of the space is poly-logarithmic of the number of points, the assumptions {\bf (A1)}-{\bf (A2)} are satisfied with $\delta \ra 1$ as $n \ra \infty$. To apply Theorem \ref{thm:main} there, we adapt our current proof to accommodate $\delta \ra 1$. For more details on this, see proof of Theorem \ref{thm:rgeom-impli}. Similarly we can also consider $\min\{p,q\} \ra 0$. This would establish a version of Theorem \ref{thm:main} for sparse graphs. Further details are omitted.
\end{rmk}

We now proceed to the proof of Theorem \ref{thm:main_improve}. For clarity of presentation we once again only provide the proof for $p=1/2$. Recall that a key to the proof of Theorem \ref{thm:main} is the variance bound obtained in Lemma \ref{lem:var_N_nu_p}. Using Chebychev's inequality, Lemma \ref{lem:var_N_nu_p} was then applied to $(\mathrm{Good}^\xi(B))^c$. With the help of two additional assumptions {\bf (A3)}-{\bf (A4)} we improve that bound by using Markov's inequality with fourth power. To do that, first we obtain the following bound on the fourth central moment of $\sN_v^{B,\xi}$. 

\begin{lem}\label{lem:fourth_moment_N_nu_p}
Let $\{\Graph_n\}_{n \in \N}$ be a sequence of graphs satisfying assumptions {\bf (A1)}-{\bf (A4)} with $p=1/2$. Then for any set $B \subset [n]$, and $\xi\in \{0,1\}$,
\beq
\E\left[\left(\sN_\nu^{B,\xi}- \E [\sN_\nu^{B,\xi}]\right)^4\right] \le \bar{C}\left({|B|^2} + {|B|^4 n^{\delta-1}}\right),\label{eq:fourth_moment}
\eeq
for some constant $\bar{C}$, depending only on $C$.
\end{lem}

\bigskip
The main advantage of Lemma \ref{lem:fourth_moment_N_nu_p} compared to Lemma \ref{lem:var_N_nu} is the absence of $|B|^3$ on the  \abbr{RHS} of \eqref{eq:fourth_moment}. This helps in obtaining a better bound on ``bad'' vertices. Its proof is a simple consequence of assumptions {\bf (A1)}-{\bf (A4)}.

\begin{proof}
One can easily check that
\begin{align}
 & \E\left(\sN_\nu^{B,\xi} - \E(\sN_\nu^{B,\xi})\right)^4\notag\\
  =  & \E(\sN_\nu^{B,\xi})^4 -4\E(\sN_\nu^{B,\xi})^3 \E(\sN_\nu^{B,\xi}) + 6 \E(\sN_\nu^{B,\xi})^2 (\E(\sN_\nu^{B,\xi}))^2 - 3 (\E(\sN_\nu^{B,\xi}))^4.\label{eq:fourth_moment_brk}
\end{align}
Now using {\bf (A1)}-{\bf (A4)} we find bounds on each of the terms in the \abbr{RHS} of \eqref{eq:fourth_moment_brk}. Considering the first term in the \abbr{RHS} of \eqref{eq:fourth_moment_brk} we see that
\begin{align*}
\E(\sN_\nu^{B,\xi})^4 &= \frac{1}{n}\sum_{\nu \in [n]}\left(\sum_{u \in B}{a_ {\nu u}^\xi}\right)^4  = \frac{1}{n}\sum_{\nu \in [n]}\sum_{u_1, u_2,u_3,u_4 \in B}{a_ {\nu u_1}^\xi a_ {\nu u_2}^\xi a_ {\nu u_3}^\xi a_ {\nu u_4}^\xi}.
\end{align*}
We further note that
\begin{align*}
\sum_{u_1, u_2, u_3, u_4 \in B}{a_ {\nu u_1}^\xi a_ {\nu u_2}^\xi a_ {\nu u_3}^\xi a_ {\nu u_4}^\xi} & =  \sum_{u_1\neq u_2\neq u_3\neq u_4 \in B}{a_ {\nu u_1}^\xi a_ {\nu u_2}^\xi a_ {\nu u_3}^\xi a_ {\nu u_4}^\xi}  +  6\sum_{u_1\neq u_2\neq u_3 \in B}{a_ {\nu u_1}^\xi a_ {\nu u_2}^\xi a_ {\nu u_3}^\xi }  \\
& +  7\sum_{u_1\neq u_2 \in B}{a_ {\nu u_1}^\xi a_ {\nu u_2}^\xi } + \sum_{u_1 \in B }{a_ {\nu u_1} ^\xi}.
\end{align*}
Using {\bf (A1)}-{\bf (A4)} we therefore deduce that
\begin{align}\label{eq:fourth moment}
&\E(\sN_\nu^{B,\xi})^4 \notag\\
  \leq  & (|B|)_4 \left(\f{1}{16}+ \kappa Cn^{\delta-1}\right) + 6(|B|)_3\left(\frac{1}{8} +\kappa C n^{\delta-1} \right) + 7(|B|)_2\left(\frac{1}{4} + \kappa C n^{\delta-1} \right) +  |B|\left(\frac{1}{2} +\kappa C n^{\delta-1} \right),
\end{align}
for some absolute constant $\kappa$. By similar arguments we also obtain that 
\begin{align}
\label{eq:third moment}
\E(\sN_\nu^{B,\xi})^3 \geq  (|B|)_3\left(\frac{1}{8} -\kappa C n^{\delta-1} \right) + 3(|B|)_2\left(\frac{1}{4} -\kappa C n^{\delta-1} \right) +  |B|\left(\frac{1}{2} -\kappa C n^{\delta-1} \right),
\end{align}

\begin{align}\label{eq:second moment}
\E(\sN_\nu^{B,\xi})^2 \leq (|B|)_2\left(\frac{1}{4} + \kappa C n^{\delta-1}\right) +  |B|\left(\frac{1}{2} +\kappa C n^{\delta-1} \right),
\end{align}
and
\begin{align}\label{eq:first moment}
  |B|\left(\frac{1}{2} - \kappa C n^{\delta-1} \right)\leq \E(\sN_\nu^{B, \xi}) \leq   |B|\left(\frac{1}{2} + \kappa C n^{\delta-1} \right).
  \end{align}
Now upon combining \eqref{eq:fourth moment}-\eqref{eq:first moment} the result follows from \eqref{eq:fourth_moment_brk}.
\end{proof}

Next from Lemma \ref{lem:fourth_moment_N_nu_p} using Markov's inequality we immediately obtain the following result. 
\begin{lem}\label{lem:good_B_improve}
Let $\{\Graph_n\}_{n \in \N}$ be a sequence of graphs satisfying assumptions {\bf (A1)}-{\bf (A4)} with $p=1/2$. Fix $\vep = \f{\bar{C}_0\log_2\hspace{-2pt}\log_2\hspace{-2pt}n}{(1-\delta)\log_2\hspace{-2pt}n}$. Then, there exists a positive constant $\wt{C}$, depending only on $C$, such that for any set $B \subset [n]$, and $\xi \in \{0,1\}$
\[
\left|\left\{ \nu \in [n]: \left|\sN_\nu^{B,\xi} - \frac{|B|}{2}\right| > \wt{C} |B| n^{\vep(\delta-1)/2}\right\}\right| \le n (\log_2\hspace{-2pt}n)^{2 \bar{C}_0}\bar{\Upsilon}_n(|B|,\delta),
\]
where $\bar{\Upsilon}_n(x,\delta):=\frac{x^{-2}+n^{\delta-1}}{2}$.
\end{lem}
Note the difference between $\bar{\Upsilon}_n(\cdot,\cdot)$ of Lemma \ref{lem:good_B_improve} and ${\Upsilon}_n(\cdot,\cdot)$ of Lemma \ref{lem:good_B}. Presence of $x^{-2}$ in $\bar{\Upsilon}_n$ is the key to the improvement of Theorem \ref{thm:main}. Using  Lemma \ref{lem:good_B_improve} we now complete the proof of Theorem \ref{thm:main_improve}.

\begin{proof}[Proof of Theorem \ref{thm:main_improve}]
We begin by noting that Lemma \ref{common-nbr-allgood} and Lemma \ref{lem:f_r_bar_bound} continue to hold as long as $r \le \log_2\hspace{-2pt} n$. Thus to improve Theorem \ref{thm:main} the key step is to establish that the number of ``bad'' vertices is negligible for all $r \le ((1-\delta) \wedge \f{2}{3}) \log_2\hspace{-2pt}n- C_5' \log_2\log_2\hspace{-2pt}n$, when $p=1/2$. That is, we need to improve  Lemma \ref{lem:Bad_bound} under the current set-up. To this end, we note that 
\[
\log_2\hspace{-2pt} \bar{\Upsilon}_n(n/2^r,\delta) \le \max\{2(r-\log_2\hspace{-2pt}n) , (\delta-1) \log_2 \hspace{-2pt}n\}.
\]
Therefore, if $r \le ((1-\delta)\wedge \f{2}{3})\log_2\hspace{-2pt}n - 2(\bar{C}_0+6) \log_2\hspace{-1pt}\hspace{-2pt}\log_2 \hspace{-2pt}n$, we have $\log_2\hspace{-2pt}\bar{\Upsilon}_n(n/2^r,\delta) \le -r-2(\bar{C}_0+6)\log_2\hspace{-1pt}\log_2\hspace{-2pt}n$. Repeating the remaining steps of the proof of Lemma \ref{lem:Bad_bound} one can extend its conclusion for all $r \le ((1-\delta) \wedge \f{2}{3}) \log_2\hspace{-2pt}n- C_5' \log_2\log_2\hspace{-2pt}n$. Since The proofs of other lemmas are based on Lemma \ref{lem:Bad_bound} one can improve those as well and  hence repeating the other steps of the proof of Theorem \ref{thm:main} one can obtain Theorem \ref{thm:main_improve}. We omit rest of the details.  
\end{proof}

\begin{rmk}
Adding more assumptions on the graph sequence, e.g.~bounds on the number of common neighbors of five and six vertices, one may potentially improve Theorem \ref{thm:main_improve} further. We do not pursue that direction.
\end{rmk}

\section{Proofs of the applications of Theorem \ref{thm:main}}
\label{sec:potaotm}

\subsection{Proof of Theorem \ref{thm:ergm-impli}}
We start this section with the following concentration result for the \abbr{ERGM} model.

\begin{thm}[Concentration results]
	\label{thm:ergm-cond}
	Under Assumption \ref{ass:ergm}, for any $\alpha>0$, there exists a constant $K_{\mvbeta}:=K_{\mvbeta}(\alpha) < \infty$ such that
	\begin{equation}
	\label{eqn:ergm-deg-con}
		\P\left(\max_{\nu\in [n]} \left||\sN_{\nu}| - np^*\right| \ge K_{\mvbeta} \sqrt{n\log{n}}\right) \le n^{-\alpha},
	\end{equation}
	and 
	\begin{equation}
	\label{eqn:ergm-comm-deg-con}
		\P\left(\max_{\nu\neq \nu^\prime\in [n]} \left||\sN_{\nu}\cap \sN_{\nu^{\prime}}| - n(p^*)^2\right| \ge K_{\mvbeta} \sqrt{n\log{n}}\right) \le n^{-\alpha},
	\end{equation}
where $p^*$ is as in Assumption \ref{ass:ergm}.
	\end{thm}

We note that Theorem \ref{thm:ergm-impli} follows immediately from the above concentration result and Theorem \ref{thm:main}, upon applying Borel Cantelli lemma. Hence, the rest of the section will be devoted in proving Theorem \ref{thm:ergm-cond}. 

Proof of \eqref{eqn:ergm-comm-deg-con} actually follows from  the proofs of \cite[Lemma 10, Lemma 11]{chatterjee2010applications}. To describe the results of \cite{chatterjee2010applications} we need to introduce the following notation. 

For each $\nu\neq \nu^\prime \in [n]$, in the sequel we write 
\begin{equation}
\label{eqn:lnu-def}
	L_{\nu \nu^\prime} := \frac{1}{n} \sum_{k\neq \nu, \nu^\prime} X_{\nu k} X_{\nu^{\prime} k } = \frac{|\sN_{\nu}\cap \sN_{\nu^{\prime}}|}{n},
\end{equation} 
for the normalized co-degree. Now recalling the definition of $\varphi_{\mvbeta}$ (see \eqref{eqn:phi-beta-def}), we paraphrase two of the main results of \cite{chatterjee2010applications} relevance to us.

\begin{lem}[{\cite[Lemma 10, Lemma 11]{chatterjee2010applications}}]
	\label{lem:cha-lij-conc}
	Under Assumption \ref{ass:ergm}, for any $i\neq j\in [n]$, for all $t\geq 8\gamma/\sqrt{n}$,
	\begin{equation}
	\label{eqn:conc-lij}
		\P\left(\sqrt{n}\left|L_{ij} - \frac{1}{n}\sum_{k\notin \set{i,j}}\varphi_{\mvbeta}(L_{ik})\varphi_{\mvbeta}(L_{jk})  \right| \geq t\right) \leq 2\exp\left(- \frac{t^2}{12(1+\gamma)}\right). 
	\end{equation}
	Further there exists a constant $K_{\mvbeta}$ such that for each $i\neq j \in [n]$,
	\begin{equation}
	\label{eqn:lij-p2-diff}
		\E\left(\left|L_{ij} - (p^*)^2\right|\right) \leq \frac{K_{\mvbeta}}{\sqrt{n}}. 
	\end{equation} 
	\end{lem}
To prove \eqref{eqn:ergm-comm-deg-con} we do not need \eqref{eqn:lij-p2-diff} directly. However, its proof will suffice in this case. More precisely, we will prove the following lemma which is sufficient to deduce \eqref{eqn:ergm-comm-deg-con}.
\begin{lem}\label{lem:proof-ergm-comm-deg-con}
Under Assumption \ref{ass:ergm}, for any $i\neq j\in [n]$, for all $t\geq 8\gamma/\sqrt{n}$,
	\begin{equation}
	\label{eqn:conc-lij-1}
		\P\left(\sqrt{n}\left|L_{ij} - \frac{1}{n}\sum_{k\notin \set{i,j}}\varphi_{\mvbeta}(L_{ik})\varphi_{\mvbeta}(L_{jk})  \right| \geq t\right) \leq 2\exp\left(- \frac{t^2}{12(1+\gamma)}\right). 
	\end{equation}
	Further, given any $\alpha >0$, there exists a positive constant $C_{\mvbeta}$ such that ,
	\begin{equation}
	\label{eqn:lij-p2-diff-1}
		\P\left(\max_{i \ne j \in [n]}\left|L_{ij} - (p^*)^2\right|\ge C_{\mvbeta}\sqrt{\f{\log n}{n}}\right) \leq n^{-\alpha}. 
	\end{equation} 
\end{lem}
Most of the proof of Lemma \ref{lem:proof-ergm-comm-deg-con} has been carried out in \cite{chatterjee2010applications}. Here we will only provide outlines, and we refer the interested readers to \cite{chatterjee2010applications}. To prove \eqref{eqn:ergm-deg-con}, along with Lemma \ref{lem:proof-ergm-comm-deg-con}, we will need the following result.
\begin{lem}\label{lem:conc-deg-ph}
 Let $\bm{X}\sim p_{\mvbeta}$. 	Fix any vertex $\nu\in [n]$ and let $d_{\nu} = \sum_{j\neq \nu} X_{j\nu}$ denote the degree of this vertex. Then for any $t> 0$, 
	\[\P\left(\left|\frac{d_{\nu}}{n-1} - \frac{1}{n-1}\sum_{j\neq \nu}\varphi_{\mvbeta}(L_{\nu j}) \right|> t\right)\leq 2\exp\left(-\f{(n-1)t^2}{1+\gamma}\right).\]
\end{lem}
\subsubsection{Proofs of Lemma \ref{lem:proof-ergm-comm-deg-con} and Lemma \ref{lem:conc-deg-ph}} 
In this section we prove  Lemma \ref{lem:proof-ergm-comm-deg-con} and Lemma \ref{lem:conc-deg-ph}. The key technique here is Chatterjee's method for concentration inequalities via Stein's method \cite{chatterjee-thesis,chatterjee-stein-ptrf}.  We start by quoting the following result which we apply to prove the lemmas. 

\begin{thm}[{\cite[Theorem 1.5]{chatterjee-stein-ptrf}}]
	\label{thm:cha-stein-conc}
	Let $\cX$ denote a separable metric space and suppose $(\bm{X}, \bm{X}^{\prime})$ be an exchangeable pair of $\cX$ valued random variables. Suppose $F:\cX\times \cX \to \R$ be a square integrable anti-symmetric function (i.e.~$F(x,y)=-F(y,x)$) and let $f(\bm{X}) := \E(F(\bm{X}, \bm{X}^{\prime})|\bm{X})$. Write, 
	\[\Delta(\vX) = \frac{1}{2}\E\left(|[f(\vX) - f(\vX^\prime)]F(\vX, \vX^{\prime})|\Big|\vX\right).\]
	Assume that $\E(e^{\theta f(\vX)}|F(\vX, \vX^{\prime})|) < \infty$ for all $\theta\in \R$. Further assume there exist constants $B_1, B_2$ such that almost surely one has $\Delta(\vX) \leq B_1 f(\vX) + B_2$. Then for any $t\geq 0$,
	\[\P(f(\vX) \geq t) \leq \exp\left(-\frac{t^2}{2B_2+2B_1 t}\right), \qquad \P(f(\vX)\leq -t) \leq \exp\left(-\frac{t^2}{2B_2}\right). \]
\end{thm}
\begin{rmk}
Since $F$ is an anti-symmetric function, it is easy to check that the function $f$ as defined above satisfies $\E(f(\vX)) = 0$. Thus the above gives concentration of $f(\vX)$ about it's expectation. 
\end{rmk}
Next we state a simple lemma that describes an equivalence between the high-temperature regime as stated in Assumption \ref{ass:ergm} and a technical condition that arose in \cite{chatterjee2010applications}. Proof follows from elementary calculus and hence omitted. 

\begin{lem}\label{lem:equiv-ass-ergm}
	From Assumption \ref{ass:ergm} recall the condition for the parameter $\mvbeta$ to be in the high temperature regime. This is equivalent to the following: $\mvbeta\in \R\times \R_+$ is such that there exists a unique solution $u^*$ to the fixed point equation $(\varphi_{\mvbeta}(u))^2 = u$ in $[0,1]$ such that $2\varphi_{\mvbeta}(u^*)\varphi_{\mvbeta}^{\prime}(u^*) < 1$.  The $u^*$ so obtained is related to $p^*$ in Assumption \ref{ass:ergm} via $u^* = (p^*)^2$. 
\end{lem}

Now we are ready to prove the lemmas. First we start with the proof of Lemma \ref{lem:conc-deg-ph}.

\begin{proof}[Proof of Lemma \ref{lem:conc-deg-ph}]
Our plan is to apply Theorem \ref{thm:cha-stein-conc}. To do so, we need to construct an exchangeable pair, which is done in the following way: We start from configuration $\vX\sim p_{\mvbeta}$, and choose a vertex $K \in [n]\setminus \set{\nu}$ uniformly at random. Conditional on $K=k$,  sample the edge $X_{\nu k}$ from the conditional distribution given the rest of the edges. We denote $\vX^{\prime}$ to be this new configuration.  It is easy to see that $(\vX, \vX')$ form an exchangeable pair, and from the definition of $p_{\mvbeta}$ we also note that, conditional on $K=k$,
\begin{equation}
\label{eqn:cond-dist-one}
	\P(X_{\nu k}^\prime = x| \vX) \propto \exp(x(\beta+\gamma L_{\nu k}) ), \qquad x\in \set{0,1}. 
\end{equation}
In particular this implies that $\E(X_{\nu k}^{\prime}|\vX) = \varphi_{\mvbeta}(L_{\nu k})$ (see \eqref{eqn:phi-beta-def} for the definition of $\varphi_{\mvbeta}(\cdot)$). Now let $d_\nu^\prime = \sum_{k\neq \nu} X_{k\nu}^\prime$ denote the degree of $\nu$ in the new configuration $\vX'$. For later use, let $L_{\nu k}^{\prime}$ denote the normalized co-degree between $\nu$ and $k$ in $\vX^\prime$, and note that conditional on $K=k$, $L_{\nu k} = L_{\nu k}^{\prime}$. Now define the anti-symmetric function,
\begin{equation}
\label{eqn:F-deg-def}
	F(\vX, \vX^{\prime}):= d_\nu - d_\nu^{\prime} = \sum_{j\neq \nu} (X_{j\nu} - X_{j\nu}^\prime). 
\end{equation}
Using \eqref{eqn:cond-dist-one} it is easy to check that,
\begin{equation}
\label{eqn:f-deg-def}
	f(\vX)= \E(F(\vX, \vX^{\prime})|\vX) = \frac{d_{\nu} - \sum_{j\neq \nu} \varphi_{\mvbeta}(L_{\nu,j})}{n-1}.
\end{equation}
To apply Theorem \ref{thm:cha-stein-conc}, we next need to find an upper bound on $|(f(\vX)-f(\vX'))F(\vX,\vX')|$. To this end, it is easy to note that $|F(\vX, \vX^\prime)|\leq 1$. Next, note that $|d_v-d_v'| \le 1$. Upon observing that $\|\varphi_{\mvbeta}'\|_\infty \le \gamma$, we further deduce
\[
\left|\varphi_{\mvbeta}(L_{v,j}) -\varphi_{\mvbeta}(L_{v,j}')\right| \le \gamma \left|L_{v,j} - L_{v,j}'\right|= \f{\gamma}{n}\left|X_{vk} -X'_{vk}\right| \le \f{\gamma}{n}, \qquad \text{ on the event } \{K=k\}.
\]
Therefore recalling the definition of $\Delta(\vX)$ defined in Theorem \ref{thm:cha-stein-conc}, we have that
\[
\Delta(\vX)\leq \f{1+\gamma}{2(n-1)},
\]
which upon applying  Theorem \ref{thm:cha-stein-conc} with $B_1=0$ and $B_2 = \f{1+\gamma}{2(n-1)}$ completes the proof.
\end{proof}
We now prove Lemma \ref{lem:proof-ergm-comm-deg-con}. Since, most of the work is already done in \cite{chatterjee2010applications} we provide an outline here for completeness.
\begin{proof}[Proof of Lemma \ref{lem:proof-ergm-comm-deg-con}]
Once again the key idea is to use Theorem \ref{thm:cha-stein-conc}. In this case, fixing $i,j\neq [n]$, we construct  the exchangeable pair as follows: Given a graph $\vX$ we choose a vertex $K\in [n]\setminus \set{i,j}$ uniformly at random and replace the edges $(X_{ik}, X_{jk})$ with $(X_{ik}^\prime, X_{jk}^\prime)$ using the conditional distribution of these edges conditional on the rest of the edges. This gives us the new graph $\vX'$. Let us write $L_{ij}^\prime$ for the normalized co-degree of $i,j$ in $\vX^\prime$.  Similar to \eqref{eqn:cond-dist-one},  it is easy to verify using the form of the Hamiltonian that for $x,y\in \set{0,1}$,
\begin{align*}
	\P(X_{ik}^\prime & = x, X_{jk}^\prime =y|\vX) \propto\\
	&\exp\left(\gamma x L_{ik} + \gamma yL_{jk} + \beta x + \beta y -\frac{\gamma}{n}x X_{ij}X_{jk} - \frac{\gamma}{n}y X_{ij}X_{ik} + \frac{\gamma}{n}xy X_{ij} \right)
\end{align*} 
Now defining the anti-symmetric function $F(\vX, \vX^{\prime}) = L_{ij} - L_{ij}^\prime$, a careful analysis similar to the proof of Lemma \ref{lem:conc-deg-ph} completes the proof of \eqref{eqn:conc-lij-1}. Note that \eqref{eqn:conc-lij-1} holds only when $t \ge 8 \gamma/\sqrt{n}$. This condition on $t$ is needed because the quantity for which concentration is desired, is not exactly $f(\vX)$  in this case. So one needs to apply triangle inequality, and hence we require the aforementioned lower bound on $t$ for \eqref{eqn:conc-lij-1} to hold. 

Next building on \eqref{eqn:conc-lij-1}, we now proceed to prove \eqref{eqn:lij-p2-diff-1}. To this end, let 
\[\Xi:= \max_{i\neq j\in [n]}\left|L_{ij} - \frac{1}{n} \sum_{k\neq \set{i,j}} \varphi_{\mvbeta}(L_{ik})\varphi_{\mvbeta}(L_{jk})\right|\]
Then using union bound, from  \eqref{eqn:conc-lij}, we deduce that given any $\alpha>0$, there exists  $\ol{C}_{\mvbeta}:=\ol{C}_{\mvbeta}(\alpha)$ such that,
\begin{equation}
\label{eqn:xi-bound}
	\P\left(\Xi \leq \ol{C}_{\mvbeta} \sqrt{\frac{\log{n}}{n}} \right) \le n^{-\alpha}. 
\end{equation}
 Let $L_{\min} = \min_{i\neq j} L_{ij}$ and abusing notation let $L_{\min} =  L_{I_{\sss \min}J_{\sss \min}}$. Similarly let $L_{\max} = \max_{i\neq j} L_{ij} $.   Since we are in the ferromagnetic regime, i.e.~$\gamma >0$, note that $\varphi_{\mvbeta}(\cdot)$ is an increasing non-negative function.  Thus,
\begin{align*}
	-\Xi \leq L_{I_{\sss \min}J_{\sss \min}} - \frac{1}{n}\sum_{k\neq \set{I_{\sss \min} , J_{\sss \min}}} \varphi_{\mvbeta}(L_{I_{\sss \min} k}) \varphi_{\mvbeta}(L_{J_{\sss \min} k}) \leq L_{I_{\sss \min}J_{\sss \min}} - \left(\varphi_{\mvbeta}(L_{J_{\sss \min}I_{\sss \min}})\right)^2 +\f{2}{n}.
\end{align*} 
Rearranging and carrying out a similar analysis for $L_{\max}$ we get, 
\begin{equation}
\label{eqn:lmax-lmin}
	\left(\varphi_{\mvbeta}(L_{\min})\right)^2 - \Xi-\f{2}{n} \;  \leq \; L_{\min} \; \leq \; L_{\max} \;\leq \; \left(\varphi_{\mvbeta}(L_{\max})\right)^2 + \Xi +\f{2}{n}. 
\end{equation} 
This motivates defining the function 
\begin{equation}
\label{eqn:varpsi-def}
	\psi_{\mvbeta}(u) := \left(\varphi_{\mvbeta}(u)\right)^2 - u, \qquad u\in [0,1]. 
\end{equation}
Using Lemma \ref{lem:equiv-ass-ergm} it is very easy to see that 
	  $u^*$ is the unique solution of $\psi_{\mvbeta}(u) = 0$ and further $\psi_{\mvbeta}(u) >  0$ for $u < u^*$, whereas $\psi_{\mvbeta}(u) < 0$ for $u > u^*$.  One also has that the derivative $\psi_{\mvbeta}^\prime(u^*) < 0$. These observations, upon applying Inverse Function Theorem,  imply that there exist $\vep, \delta >0$ such that if $|u-u^*|> \delta$ then $|\psi_{\mvbeta}(u)| > \vep$ and further 
	\[\sup_{0 < |u-u^*| \leq \delta} \left[\frac{u-u^*}{-\psi_{\mvbeta}(u)}\right]:= c > 0. \]
Now by \eqref{eqn:lmax-lmin}, $\psi_{\mvbeta}(L_{\max}) \geq -\left(\Xi+\f{2}{n}\right)$ and $\psi_{\mvbeta}(L_{\min}) \leq \Xi+\f{2}{n}$. Thus on the event $\{\Xi\leq \f{\vep}{2}\}$, considering all the cases, we have
\[u^* - \delta < L_{\min} \leq L_{\max} \leq u^* +\delta.\]
However the latter implies that 
\[|L_{\min} -u^*| \leq c \left(\Xi+\f{2}{n}\right), \qquad |L_{\max} - u^*| \leq c\left(\Xi+\f{2}{n}\right),\]
which upon applying yields. This completes the proof. 
\end{proof}
Equipped with Lemma  \ref{lem:proof-ergm-comm-deg-con} and Lemma \ref{lem:conc-deg-ph} we are now ready to prove the concentration result Theorem \ref{thm:ergm-cond}.

We start with the proof for the degree of vertices namely \eqref{eqn:ergm-deg-con}. We will see that this result follows directly if we are able to prove \eqref{eqn:ergm-comm-deg-con}. We start with the following concentration inequality. 

\begin{proof}[Proof of Theorem \ref{thm:ergm-cond}]
We begin by noting that \eqref{eqn:ergm-comm-deg-con} is immediate from Lemma \ref{lem:proof-ergm-comm-deg-con}. Thus it only remains to establish \eqref{eqn:ergm-deg-con}. Noting that $\varphi_{\mvbeta}(\cdot)$ is $\gamma$-Lipschitz on $\R$ and using that by definition, $p^* = \varphi_{\mvbeta}((p^*)^2)$, from Lemma \ref{lem:proof-ergm-comm-deg-con}, we get that,
\begin{equation}
\label{eqn:sup-phi-p-st}
\P\left(\max_{i\neq j\in [n]}\left|\varphi_{\mvbeta}(L_{ij}) - p^*\right| \ge \gamma C_{\mvbeta}\sqrt{\frac{\log{n}}{n}}\right) \le n^{-\alpha}. 
\end{equation} 
Further using union bound, by Lemma \ref{lem:conc-deg-ph}, and enlarging  $C_{\mvbeta}$ if needed, we also have that, 
\begin{equation}
\label{eqn:sup-deg-phi}
	\P\left(\max_{i\in [n]} \left| d_i - \sum_{j\neq i} \varphi_{\mvbeta}(L_{ij}) \right| \geq C_{\mvbeta} \sqrt{n\log{n}}\right) \le n^{-\alpha}. 
\end{equation}
Combining the above with \eqref{eqn:sup-phi-p-st} completes the proof.
\end{proof}

\subsection{Proof of Theorem \ref{thm:rgeom-impli}} In this section our goal is to provide the proof  of Theorem \ref{thm:rgeom-impli}. Hence, we first need to show that assumptions {\bf (A1)}-{\bf (A2)}  hold for random geometric graphs. This is derived in the following concentration result.
\begin{thm}[Concentration result]
	\label{thm:rgeom-cond}
	Consider random geometric graphs $\Graph(d,n,p)$ satisfying Assumption \ref{ass:rgeom}. Then for any $\delta > 1/2$, the following holds almost surely, for all large $n$:
\begin{enumerate}[(a)]
\item
	\begin{equation}
	\label{eqn:rgeom-deg-con}
		\max_{\nu\in [n]} \left||\sN_{\nu}| - np\right| < n^{\delta}.
	\end{equation}
\item For every pair $i \ne j \in [n]$, there exists $p_n^{i,j}$, possibly random, such that 	
	\begin{equation}
	\label{eqn:rgeom-comm-deg-con}
		\max_{\nu\neq \nu^\prime\in [n]} \left||\sN_{\nu}\cap \sN_{\nu^{\prime}}| - n(p_n^{v,v'})^2\right| \leq n^{\delta}.
	\end{equation}
\item Moreover
	\begin{equation}
	\label{eqn:epsn-def}
\max_{v \ne v' \in [n]} \left|(p_n^{v,v'})^2 - p^2\right| \le \tau_n, \text{ where } \tau_n :=  \kappa_p\left[\sqrt{\frac{\max\{\log{n},\log{d}\}}{d}} + \frac{1}{n^2}\right]
	\end{equation}
and $\kappa_p$ is some constant depending only on $p$.	
\end{enumerate}	
\end{thm}

\bigskip
Note that \eqref{eqn:rgeom-deg-con}-\eqref{eqn:rgeom-comm-deg-con} does not establish that the random geometric graphs satisfy {\bf (A1)}-{\bf (A2)}. In fact, when the dimension of the space is poly-logarithmic of the number of points one cannot expect to establish {\bf (A1)}-{\bf (A2)} for any $\delta \le 1-\vep$ for an arbitrarily small $\vep >0$. However, using \eqref{eqn:epsn-def} we see that one can establish {\bf (A1)}-{\bf (A2)} for random geometric graphs with 
\[
\delta_n=1 - \f{\log (1/\tau_n)}{\log n}.
\]
We use this key observation to carefully adapt the proof of Theorem \ref{thm:main} to establish Theorem \ref{thm:rgeom-impli}.

\begin{proof}[Proof of Theorem \ref{thm:rgeom-impli}]
We start by redefining the notion of ``good'' vertices. Since in this current set-up we work with $\delta_n$ instead of a fixed $\delta$ it is natural to have the following definition: For any given set $B \subset [n]$, and $\xi \in \{0,1\}$, we here define
\[
\mathrm{Good}^{\xi,p}(B):= \left\{ \nu \in [n]: \left|\sN_\nu^{B,\xi} - {|B|}p^\xi q^{1-\xi}\right| \le \wt{C}  |B| p^\xi q^{1-\xi}n^{\vep_n(\delta_n-1)/2} \right\},
\]
where $\vep_n:= \f{\bar{C}_0 \log \log (1/\tau_n)}{(1-\delta_n)\log (1/\tau_n)}$ and $\wt{C}$ is some large constant. Equipped with this definition we can easily make necessary changes in the proof of Theorem \ref{thm:main} to complete the proof in the current set-up. 

In summary, the roles of $\delta$ and $n$ from the proof of Theorem \ref{thm:main} will be replaced by $\delta_n$ and $1/\tau_n$ in this proof. Keeping this in mind, we proceed below.

Proceeding exactly same as in Lemma \ref{common-nbr-allgood_p} we deduce that 
if $r < \log (1/\tau_n)$, $\ul{v} \in [n]^r$, $\ul{\xi} \in \{0,1\}^r$, and $\nu_j \in \Good^{\ul{\xi}^{j},p}(\nu_1, \nu_2,...,\nu_{j-1})$,  for all $j =3,4,\ldots,r$, then 
\beq\label{eq:good_set_estimates_p}
\left|\left|\sN_{\nu_1}^{\xi_1} \cap \sN_{\nu_2}^{\xi_2} \cap \cdots \cap \sN_{\nu_r}^{\xi_r}\right|- {n}p^{n(\ul{\xi})}q^{r-n(\ul{\xi})}\right|  \le 3\wt{C}  (\log (1/\tau_n))^{-(\bar{C}_0/2-1)}  {n}p^{n(\ul{\xi})}q^{r-n(\ul{\xi})},
\eeq
for all large $n$, where $n(\ul{\xi}):=|\{i \in [r]: \xi_i =1\}|$. Next we need to extend Lemma \ref{lem:Bad_bound_p} in the current set-up. Recall a key to the proof of Lemma \ref{lem:Bad_bound_p} is the variance bound of Lemma \ref{lem:var_N_nu_p}. Lemma \ref{lem:var_N_nu_p} can be extended in the context of random geometric graph to yield
\beq\label{eq:var_bd_rgg}
\mathrm{Var} \left(\sN_\nu^{B,\xi}\right) \le |B| pq + \bar{C} {|B|^2 n^{\delta_n-1}}.
\eeq
From \eqref{eq:var_bd_rgg} it also follows that 
\[
\left|\left\{ \nu \in [n]: \left|\sN_\nu^{B,\xi} - {|B|}p^\xi q^{1-\xi}\right| > \wt{C} |B| p^\xi q^{1-\xi}n^{\vep_n(\delta_n-1)/2}\right\}\right| \le n (\log (1/\tau_n))^{\bar{C}_0}\Upsilon_n^p(|B|,\delta_n),
\]
where we recall $\Upsilon_n^p(x,\delta)=\frac{2 x^{-1}+q^{-2}n^{\delta-1}}{2}$. We note that if $\log(1/q) r \le ((1-\delta_n) \wedge \f{1}{2})\log n - (\bar{C}_0+13)\log \log (1/\tau_n)$, then $\log \Upsilon_n^p((n/2)q^r,\delta_n) \le - \log(1/q) r -(\bar{C}_0+12)\log \log \tau_n$. Now repeating the remaining steps of the proof of Lemma \ref{lem:Bad_bound_p} we obtain the following result:
For any two positive integers $j< r$ let ${\sf H}_r$ be a graph on $r$ vertices. Fix ${\sf H}_j$ any one of the sub-graphs of ${\sf H}_r$ induced by $j$ vertices. Assume that 
\[
n_\Graph({\sf H_r}) \ge \f{1}{2} \times \f{(n)_r}{|\mathrm{Aut}({\sf H}_r)|}\left(\f{p}{q}\right)^{|E({\sf H}_r)|}q^{{r \choose 2}} \quad \text{ and } \quad n_\Graph({\sf H_j}) \le 2 \times \f{(n)_j}{|\mathrm{Aut}({\sf H}_j)|} \left(\f{p}{q}\right)^{|E({\sf H}_j)|}q^{{j \choose 2}}. 
\]
There exists a large positive constant $C_0$, depending only on $p$ such that, for any given $\ul{\xi}=\{\xi_1,\xi_2,\ldots,\xi_j\} \in \{0,1\}^j$, and $r \le (1-\delta_n)(\log n/\log (1/q))- C_0 \log \log (1/\tau_n)$, we have 
\beq\label{eq:bad_bound_p_rgg}
\left|\Bad^{\ul{\xi}}({\sf H}_{j}, {\sf H}_r)\right| \le \f{n_\Graph({\sf H_r})}{\left(\log (1/\tau_n) \right)^{9(r-j)}},
\eeq
for all large $n$. 

To complete the proof we now need to extend Lemma \ref{goodbad_p} and Lemma \ref{lem:f_r_bar_bound_p}. Since Lemma \ref{goodbad_p} is built upon Lemma  \ref{lem:Bad_bound_p} we obtain necessary modifications of Lemma \ref{goodbad_p} using \eqref{eq:bad_bound_p_rgg}. The main difference here is the rate of convergence. More precisely,  $\log n$ appearing in the  rate of convergence in  Lemma \ref{goodbad_p} should be replaced by $\log(1/\tau_n)$ in the current set-up. Next recalling the proof of Lemma \ref{lem:f_r_bar_bound_p} we see that the conclusion of Lemma \ref{lem:f_r_bar_bound_p} continues to hold for all $r \le \log (1/\tau_n)$. Combining these ingredients one can now easily finish the proof. We omit the tedious details.
\end{proof}

\bigskip

Thus it only remains to establish Theorem \ref{thm:rgeom-cond}. We break the proof of Theorem \ref{thm:rgeom-cond} into two parts. In Section \ref{sec:proof_rgg_part_1} we prove \eqref{eqn:rgeom-deg-con}-\eqref{eqn:rgeom-comm-deg-con} and in Section \ref{sec:proof_rgg_part_2}  we establish \eqref{eqn:epsn-def}. 

Before going to the proof, let us observe that  $n$ uniform points on $\sS_{d-1}$ can be generated as follows: Let $\bm{Z}_1, \bm{Z}_2, \ldots, \bm{Z}_n$ be i.i.d.~standard Normal random vectors namely $\bm{Z}_i\sim N_d(\mvzero, \vI_d)$ where $\mvzero := (0,0,\ldots, 0)$ is the origin in $\R^d$ and $\vI_d$ is the $d\times d$ identity matrix. Then, setting 
\begin{equation}
\label{eqn:vx-vz-const}
	\vX_1 := \frac{\bm{Z}_1}{\|\bm{Z}_1\|_2}, \vX_2 := \frac{\bm{Z}_2}{\|\bm{Z}_2\|_2}, \ldots ,\vX_n := \frac{\bm{Z}_n}{\|\bm{Z}_n\|_2}, 
\end{equation}
where $\|\cdot \|_2$ denotes the Euclidean norm in $\R^d$, we see that $\vX_1,\vX_2,\ldots, \vX_n$ are $n$ i.i.d. uniform points on $\sS_{d-1}$. Here, we will this representation of $\{\vX_i\}_{i \in [n]}$, because it allows us to use properties of Gaussian random vectors. For future reference let us denote $\bm{Z}_i := (Z_{i1}, Z_{i2}, \ldots Z_{id})$.

\subsubsection{{\bf Proof of \eqref{eqn:rgeom-deg-con}-\eqref{eqn:rgeom-comm-deg-con}.}}\label{sec:proof_rgg_part_1}
In this section we prove \eqref{eqn:rgeom-deg-con}-\eqref{eqn:rgeom-comm-deg-con}. We start with the proof of \eqref{eqn:rgeom-deg-con}.

\begin{proof}[Proof of \eqref{eqn:rgeom-deg-con}]
For ease of writing, without loss of generality, let us consider the vertex $1$. Note that conditional on $\bm{Z}_1$, one can construct an orthogonal transformation $\vP$ such that in the new coordinates, $\bm{X}_1 = (1,0,\ldots, 0)$ whilst $\bm{Z}_k^{\prime} = \vP \bm{Z}_k$ for $2\leq k\leq n$ are i.i.d.~$N_d(\mvzero, \vI_d)$. To ease notation we will continue to refer to these as $(\bm{Z}_2, \ldots, \bm{Z}_n)$. Thus note that $\langle \bm{X}_1, \bm{X}_i \rangle = Z_{i1}/\|\bm{Z}_i\|_2$.  In particular the constant $t_{p,d}$ is such that, 
\begin{equation}
\label{eqn:tpd-alternate}
\P\left(\frac{Z_{i1}}{\|\bm{Z}_i\|_2} \geq t_{p,d}\right) = p	.
\end{equation}
Further for each $i\neq j\in [n]$ we write $\set{i\sim j}$ for the event that vertex $i$ and $j$ are connected by an edge in $\Graph(n,d,p)$. Then we have,
\[\left\{1\sim 2\right\} = \left\{\frac{Z_{21}}{\|\bm{Z}_2\|_2} \geq t_{p,d}\right\},\quad \left\{1\sim 3\right\} = \left\{\frac{Z_{31}}{\|\bm{Z}_3\|_2} \geq t_{p,d}\right\},\,  \ldots, \quad \left\{1\sim n\right\} = \left\{\frac{Z_{n1}}{\|\bm{Z}_n\|_2} \geq t_{p,d}\right\}.  \]
This implies that the distribution of the degree of vertex $1$ is $\mbox{Bin}(n-1, p)$. Now, applying  Chernoff's inequality we obtain concentration inequality for $\sN_1$. For any other vertex one can derive the same bound proceeding same as above. Thus the union bound and Borel Cantelli lemma completes the proof of \eqref{eqn:rgeom-deg-con}.  
\end{proof}

\begin{proof}[Proof of \eqref{eqn:rgeom-comm-deg-con}]
We only establish \eqref{eqn:rgeom-comm-deg-con} for $i=1, j=2$. Proof for any other pair of vertices is exactly same. One can then complete the proof of \eqref{eqn:rgeom-comm-deg-con} by taking a  union bound over every pair $i \ne j \in [n]$. 

As in the proof of \eqref{eqn:rgeom-deg-con}, without loss of generality, let $\bm{X}_1 = \ve_1$ whist all other points are constructed as in \eqref{eqn:vx-vz-const}. We start by conditioning on $\bm{X}_2 = \bm{Z}_2/\|\bm{Z}_2\|_2 = (a_1,a_2, \ldots a_d)$. For the rest of this section, write $\tilde\P$ and $\tilde\E$ for the corresponding conditional probability and conditional expectation. Then note that, 
\[|\sN_1\cap \sN_2| = \sum_{v \neq \set{1,2}} \mathbf{1} \left\{\frac{Z_{v1}}{\|\bm{Z}_v\|_2} > t_{p,d}, \frac{\sum_{i=1}^d a_i Z_{vi}}{\|\bm{Z}_v\|_2}> t_{p,d}\right\}. \]
In particular, under $\tilde\P$ we have that $|\sN_1\cap \sN_2|$ is a $\mbox{Bin}(n-2, \tilde{\gamma}_n)$ where, 
\beq\label{eq:gamma_n_define}
\tilde{\gamma}_n :=  \tilde\P\left({\frac{Z_{v1}}{\|\bm{Z}_v\|_2} > t_{p,d}, \frac{\sum_{i=1}^d a_i Z_{vi}}{\|\bm{Z}_v\|_2}> t_{p,d}}\right). 
\eeq
Now using Chernoff's inequality and taking expectation over $\bm{Z}_2$, we obtain
\[
\P\left(|\sN_1 \cap \sN_2| - n(p_n^{1,2})^2| \ge t\right) \le 2 \exp\left(-\f{t^2}{2n}\right),
\]
where $p_n^{1,2}:=\sqrt{\tilde{\gamma}_n}$, which upon taking a union over $i \ne j \in [n]$ completes the proof.
\end{proof}

\subsubsection{{\bf Proof of \eqref{eqn:epsn-def}.}}\label{sec:proof_rgg_part_2} 
Here again we establish \eqref{eqn:epsn-def} for $i=1, \, j=2$. The proof for arbitrary $i \ne j \in [n]$ is same, and therefore one can complete the proof by taking a union bound. Thus it is enough to prove \eqref{eqn:epsn-def} only for $i=1, \, j=2$. 
To simplify notation, we let 
\begin{equation}
\label{eqn:za-zb-def}
	Z_a := Z_{v1}, \qquad Z_b: =  \sum_{i=1}^d a_i Z_{vi}, \qquad \rho := a_1 = \frac{Z_{21}}{\|\bm{Z}_2\|_2}. 
\end{equation}
Note that $(Z_a,Z_b)$ has a standard Bivariate Normal distribution with correlation $\rho$. If $\rho$ were zero, then we would have immediately deduced that \eqref{eqn:epsn-def} holds. Here, we show that $\rho \approx 0$ with large probability, and therefore $\wt{\gamma}_n$ should not differ much from $p^2$. Below, we make this idea precise. We begin with two concentration results on Gaussian random vectors.

\begin{lem}\label{lem:norm-conc}
	Let $Z_1, Z_2, \ldots Z_d$ be i.i.d.~standard Normal random variables and let $\bm{Z} := (Z_1, \ldots, Z_d)$ be the vector representation of this in $\R^d$. Then there exists $d_0$ such that for all $d> d_0$ and all $\vep> 0$,
	\[\P\left(\left|\|\bm{Z}\|_2 - \sqrt{d}\right|> \vep +\frac{1}{2\sqrt{d}}\right)\leq 2e^{-\frac{\vep^2}{2}}.\]
\end{lem}
\begin{proof} First note that since $\|\cdot\|_2$ is a one-Lipschitz function, standard Gaussian concentration (for example, see \cite{boucheron2013concentration}) implies that
\begin{equation}
\label{eqn:lip-gauss-conc}
	\P\left(\left|\|\bm{Z}\|_2 - \E(\|\bm{Z}\|_2)\right| > \vep\right) \leq 2\exp(-\frac{\vep^2}{2}).
\end{equation}
Since $\|\bm{Z}\|_2$ has a chi-distribution, this implies that 
\begin{equation}
\label{eqn:expec-norm}
	\E(\|\bm{Z}\|_2) = \sqrt{2} \frac{\Gamma(\frac{d+1}{2})}{\Gamma(\frac{d}{2})},
\end{equation}
where $\Gamma(\cdot)$ is the Gamma function. Standard results on asymptotics  for the Gamma function (see, for example  \cite{tricomi1951asymptotic}) imply 
\[\frac{\Gamma(\frac{d+1}{2})}{\Gamma(\frac{d}{2})} = \sqrt{\frac{d}{2}}\left[ 1 - \frac{1}{4d} + O\left(\frac{1}{d^2}\right)\right]\]
Using this in \eqref{eqn:expec-norm}, upon combining with \eqref{eqn:lip-gauss-conc} completes the proof. 
\end{proof}

\bigskip
Note that the first co-ordinate of $\vX_2$ is given by $a_1 = Z_{21}/\|\bm{Z}_2\|$. Lemma \ref{lem:norm-conc} implies that $\sqrt{d} a_1 = O(1)$ with large probability. The next simple Lemma establishes concentration rates about the order of magnitude. 
\begin{lem}\label{lem:rho-conc}
	Under Assumption \ref{ass:rgeom}, $\exists n_0$ such that for $n\geq n_0$, 
	\[\P\left(\frac{|Z_{21}|}{\|\bm{Z}_2\|_2} > \frac{\sqrt{64\log{n}}}{\sqrt{d}}\right)\leq \frac{4}{n^8}. \]
\end{lem}
\begin{proof} To ease notation, let $t := \frac{\sqrt{64\log{n}}}{\sqrt{d}}$. Then note that, 
\[\P\left(\frac{|Z_{21}|}{\|\bm{Z}_2\|_2} > t \right) \leq \P\left( \left|  \|\bm{Z}_2\|_2 - \sqrt{d}\right| > \frac{\sqrt{d}}{2}\right) + \P\left(|Z_{21}| > \frac{t}{2}\sqrt{d}\right).\]
The upper bound on the first term in the \abbr{RHS} above follows from Lemma \ref{lem:norm-conc}, whereas standard tail bounds for the Normal distribution gives upper bound on the second term. Combining them together completes the proof. 
\end{proof}

\bigskip
As already noted $(Z_a,Z_b)$ has a bivariate Normal distribution. Therefore we will need some estimates on the distribution function of a bivariate Normal distribution. To this end, throughout the sequel, let $\Phi$ denote the standard Normal distribution function and let $\bar{\Phi}(h,k,\rho):=\P(Z \ge h, Z_\rho \ge k)$, where $(Z,Z_\rho)$ has a standard bivariate Normal distribution with correlation $\rho$. Then we have the following result. 
\begin{lem}[{\cite{willink2005bounds}}]
\label{lem:tail-bd-biv-norm}
	Fix $0\leq \rho\leq 1$ and $h > 0$.  Let $(Z, Z_{\rho})$ has a standard bivariate Normal distribution with correlation $\rho$. Then denoting~$\theta:= \sqrt{\frac{1-\rho}{1+\rho}}$, one has 
	\[\Phi(-h)\Phi(-\theta h) \leq \bar{\Phi}(h,h,\rho) \leq (1+\rho) \Phi(-h)\Phi(-\theta h). \]
\end{lem}

\bigskip

Finally to evaluate \eqref{eq:gamma_n_define}, we need the following asymptotic behavior of $t_{p,d}$.
\begin{lem}[{\cite[Lemma 1]{devroye2011high}}]
\label{lem:DGLU}
	Fix $0 < p \leq 1/2$ and assume $d\geq \max\set{4/p^2,27}$. Then there exists a constant $\kappa_p^*<\infty$, depending only on $p$, such that one has 
	\[\bigg|t_{p,d}\sqrt{d}\; -\; \Phi^{-1}(1-p)\bigg| \leq \kappa_p^*\sqrt{\frac{\log{d}}{d}}.\]
\end{lem}

\begin{proof}[Proof of \eqref{eqn:epsn-def}]
First let us define the event
\begin{equation}
\label{eqn:good-event}
	\sG_n^{1,2}:= \left\{\frac{|Z_{21}|}{\|\bm{Z}_2\|_2} \leq  \frac{\sqrt{64\log{n}}}{\sqrt{d}} \right\}.
\end{equation}
We will show that on $\sG_n^{1,2}$, we have $|(p_n^{1,2})^2 - p^2| \le \tau_n$. By the same argument one can extend the result for all pairs $i \ne j \in [n]$. Therefore, the proof completes by setting $\sG_n:=\cap_{i \ne j \in [n]} \sG_n^{i,j}$, applying Lemma \ref{lem:rho-conc},  taking a  union bound, and applying Borel Cantelli lemma. 

We break the proofs into two parts.  First we show that on $\sG_n^{1,2}$, we have $\tilde{\gamma}_n \le p^2 +\tau_n$. Denoting $\Delta_n  :=  \sqrt{8\log{n}}$, we note
\begin{align}
\label{eqn:1416}
	\tilde{\gamma}_n  & \leq \tilde\P\left(Z_a > t_{p,d}(\sqrt{d} - \Delta_n), Z_b > t_{p,d}(\sqrt{d} - \Delta_n)\right) + \tilde\P \left(| \norm{\bm{Z}_v}_2  - \sqrt{d} | > \Delta_n \right)\notag\\
	 & =: \text{Term A} + \text{Term B}.
\end{align} 
	By Lemma \ref{lem:norm-conc}, we have $\text{Term B} \le 2/n^{2}$. To bound $\text{Term A}$ we use the fact that under $\tilde\P$, $(Z_a, Z_b)$ has a standard bivariate Normal distribution with correlation $\rho$ as defined in \eqref{eqn:za-zb-def}. First let us consider the case $\rho >0$. Using Lemma \ref{lem:tail-bd-biv-norm} with $h = t_{p,d}(\sqrt{d} - \Delta_n)$, we obtain
\beq\label{eq:TermA}
\text{Term A} \le (1+\rho) \Phi(-t_{p,d}(\sqrt{d} - \Delta_n)) \Phi(-\theta t_{p,d}(\sqrt{d} - \Delta_n)).\notag
\eeq
Next, applying 	Lemma \ref{lem:DGLU} we note
\begin{align}\label{eq:t_pd}
\left|t_{p,d}(\sqrt{d}-\Delta_n) -  \Phi^{-1}(1-p)\right| \leq \kappa'_p\sqrt{\frac{\max\{\log{n},\log{d}\}}{d}},
\end{align}
for another constant $\kappa_p'$, depending only on $p$. Therefore using the Mean-Value Theorem we further obtain that
\beq\label{eq:termA_1}
\left| \Phi(-t_{p,d}(\sqrt{d} - \Delta_n)) - p \right| \le \bar{\kappa}_p \sqrt{\frac{\max\{\log{n},\log{d}\}}{d}}.
\eeq
Moreover, note that on the event $\sG_n^{1,2}$ we have $\rho \leq \sqrt{64\log{n}/d}$ and by Assumption \ref{ass:rgeom} we have $d/\log n \ra \infty$. Hence, using the Mean-Value Theorem again
\beq\label{eq:termA_2}
\left|\Phi(-t_{p,d}(\sqrt{d} - \Delta_n)) - \Phi(-\theta t_{p,d}(\sqrt{d} - \Delta_n))\right|  \le \f{1}{\sqrt{2 \pi}}\left|\theta -1\right| |t_{p,d}(\sqrt{d} - \Delta_n)| \le \kappa_p'' \sqrt{\frac{\max\{\log{n},\log{d}\}}{d}},
\eeq
for some other constant $\kappa_p''$. Combining \eqref{eq:termA_1}-\eqref{eq:termA_2} we obtain the desired bound for \text{Term A} when $\rho >0$. 

When $\rho <0$, we cannot directly use Lemma \ref{lem:tail-bd-biv-norm}. Instead we use the following result: 
\[
\bar{\Phi}(h,h,\rho) = 2\Phi(-h)\Phi(-\theta h) - \bar{\Phi}(\theta h, \theta h, -\rho),
\]
for any $h \in \R$ and $ \rho \in (-1,1)$ (see \cite[Eqn.~(C)]{SO}, \cite[pp.~2294]{willink2005bounds}). Now using the lower bound from Lemma \ref{lem:tail-bd-biv-norm}, we obtain $\bar{\Phi}(\theta h, \theta h, -\rho) \ge \Phi(-\theta h) \Phi(-\theta^2 h)$. We have already seen above that $\theta \approx 1$. Therefore proceeding as above we can argue that $ \Phi(-\theta^2  h) \approx \Phi(-\theta h)$. Then proceeding similarly as above we obtain the desired upper bound for $\text{Term A}$ when $\rho <0$. The details are omitted.

Now it remains to find a lower bound on $\t{\gamma}_n$. To this end, it is easy to note that
\begin{align}
	\tilde{\gamma}_n &\geq \tilde\P\left(\frac{Z_a}{\norm{\bm{Z}_v}_2} > t_{p,d}, \frac{Z_b}{\norm{\bm{Z}_v}_2} > t_{p,d}, \bigg|\norm{\bm{Z}_v}_2 -\sqrt{d} \bigg| \leq \Delta_n \right) \label{eqn:2018}\notag\\ 
	&\geq \text{Term A}' - \text{Term B}, \notag
\end{align} 
where 
\[
\text{Term A}':=\tilde\P\left(Z_a > t_{p,d}(\sqrt{d} + \Delta_n), Z_b > t_{p,d}(\sqrt{d} + \Delta_n)\right).
\]
Proceeding as in the case of $\rho >0$, one can complete the proof. We omit the details. 	
\end{proof}








\subsection{Proof of Theorem \ref{lem:large_clique}.}

To prove Theorem \ref{lem:large_clique} we need a good estimate on $\P(\sC_r)$. It is well known that the number of copies of any subgraph ${\sf H}$ in an Erd\H{o}s-R\'{e}nyi random graph is well approximated in total variation distance, by a Poisson distribution with appropriate mean (see \cite[Section 5.1]{BHJ}). Combining \cite[Theorem 5.A]{BHJ} and \cite[Lemma 5.1.1(a)]{BHJ} results in the following proposition. Here for any two probability measures $\cP_1$ and $\cP_2$ defined on $\mathbb{N}$, write $d_{\mathrm{TV}}(\cP_1,\cP_2)$ for  the total variation distance between these measures.  
Abusing notation write $d_{\mathrm{TV}}(X_1,X_2):= d_{\mathrm{TV}}(\cP_1,\cP_2)$, when $X_1 \sim \cP_1$, and $X_2 \sim \cP_2$.

\begin{prop}\label{prop:bhj}
Let $\Graph(n,1/2)$ be the Ed\H{o}s-R\'{e}nyi random graph with connectivity probability $\f{1}{2}$. For any graph ${\sf H}$ let $v({\sf H}) \le n$ denote the cardinality of the vertex set of ${\sf H}$, and $e({\sf H})$ be the same for the edge set. Further write
\[
\mu={n \choose v({\sf H})} \f{v({\sf H})!}{a({\sf H})} \left(\f{1}{2}\right)^{e({\sf H})},
\]
where $a({\sf H})$ denote the number elements in the automorphism groups of ${\sf H}$. Then
\beq\label{eq:d_tv_bd}
d_{\mathrm{TV}}\left(n_{\Graph(n,1/2)}({\sf H}), \dPois(\mu)\right) \le \left(1- e^{-\mu}\right)\left(\f{\mathrm{Var}(n_{\Graph(n,1/2)}({\sf H}))}{\mu} -1 + 2\left(\f{1}{2}\right)^{e({\sf H})}\right).
\eeq
For $\wt{\sf H}$ any isomorphic copy of ${\sf H}$, let $\Gamma_{\wt{\sf H}}^t$ be the collections of all subgraphs of the complete graph on $n$ vertices that are isomorphic to ${\sf H}$ with exactly $t$ edges not in $\wt{\sf H}$. Then
\beq\label{eq:var_W}
\f{\mathrm{Var}(n_{\Graph(n,1/2)}({\sf H}))}{\mu}= 1- \left(\f{1}{2}\right)^{e({\sf H})} + \sum_{t=1}^{e({\sf H})-1}|\Gamma_{\wt{\sf H}}^t| \left(\left(\f{1}{2}\right)^t- \left(\f{1}{2}\right)^{e({\sf H})}\right).
\eeq
\end{prop}
Equipped with Proposition \ref{prop:bhj} we are now ready to prove Lemma \ref{lem:large_clique}.

\begin{proof}[Proof of Theorem \ref{lem:large_clique}]
First note that $\sC_r= \{ n_{\Graph(n,1/2)}({\sf C}_r) >0\}$. Thus it is enough to prove that
\beq\label{eq:large_clique_up_bd}
\P\left(\sA_{\mathrm{d}}^\delta \cup \sA_{\mathrm{cod}}^\delta \Big| n_{\Graph(n,1/2)}({\sf C}_r) >0 \right) \le \f{\P\left(\sA_{\mathrm{d}}^\delta \cup \sA_{\mathrm{cod}}^\delta\right)}{\P\left( n_{\Graph(n,1/2)}({\sf C}_r) >0 \right)} \ra 0.
\eeq
Using Hoeffding's inequality we obtain
\[
\sup_{\nu \in [n]}\P\left(\left|\left|\sN_\nu\right| - \frac{n}{2}\right|  \ge   Cn^{\delta}\right), \sup_{\nu \ne \nu' \in [n]} \P\left(\left|\left|\sN_\nu\cap \sN_{\nu'}\right| - \frac{n}{4}\right|  \ge   Cn^{\delta}\right) \le 2 \exp\left(-2 C n^{2\delta -1}\right).
\]
Thus taking a union bound 
\beq\label{eq:num_bd}
\P\left(\sA_{\mathrm{d}}^\delta \cup \sA_{\mathrm{cod}}^\delta\right) \le  2 \exp\left(- C n^{2\delta -1}\right).
\eeq
Now to control the denominator of \eqref{eq:large_clique_up_bd} we use Proposition \ref{prop:bhj}. Note that, if we are able to show that the \abbr{RHS} of \eqref{eq:d_tv_bd}, excluding the term $(1- e^{-\mu})$, is $o(1)$, then from \eqref{eq:d_tv_bd} we deduce that
\[
\left|\P(n_{\Graph(n,1/2)}({\sf C}_r) >0) - (1-e^{-\mu})\right| \le  (1-e^{-\mu}) o(1).
\]
This therefore implies that
\beq\label{eq:nG_bd}
\P(n_{\Graph(n,1/2)} ({\sf C}_r)>0) \ge \f{1-e^{-\mu}}{2}.
\eeq
Next note
\[
\mu= \E\left[n_{\Graph(n,1/2)}({\sf C}_r)\right]= {n \choose r} \left(\f{1}{2}\right)^{r \choose 2}
= \f{n(n-1)\cdots (n-r+1)}{r!} \left(\f{1}{2}\right)^{r \choose 2}.
\]
Since $r= c n^{1/2-\vep}$, using Stirling's approximation, we have
\begin{align*}
\mu \le \f{1}{\sqrt{2\pi r}}\left(\f{en}{r}\right)^r  \left(\f{1}{2}\right)^{r \choose 2}= \f{1}{\sqrt{2 \pi c} n^{1/4-\vep/2}}\left(\f{e}{c} n^{1/2+\vep}\right)^{c n^{1/2-\vep}} \left(\f{1}{2}\right)^{r \choose 2}.
\end{align*}
As $r=cn^{1/2-\vep}$ it is easy to check that
\[
 \left(\f{e}{c} n^{1/2+\vep}\right)^{c n^{1/2-\vep}} \ll 2^{r \choose 2},
\]
and hence $\mu \ra 0$, as $n \ra \infty$. Therefore $1- e^{-\mu} \approx \mu$, for large $n$,  and so from \eqref{eq:nG_bd} we further have,
\beq\label{eq:denom_bd}
\P(n_{\Graph(n,1/2)} ({\sf C}_r)>0) \ge \f{\mu}{4}, \quad \text{ for all large } n.
\eeq
Next we observe
\begin{align*}
\limsup_n \log_e\left(\f{8 \exp\left(- C n^{2\delta -1}\right)}{{n \choose r} \left(\f{1}{2}\right)^{r \choose 2}} \right) \le \log_e\hspace{-2pt}8 + \limsup_n \left[ {r \choose 2} \log_e\hspace{-2pt}2- C n^{2\delta-1}\right].
\end{align*}
Since $r=cn^{1/2-\vep}$ with $\vep+\delta >1$, we have
\[
r^2= c^2 n^{1-2\vep} \ll n^{2\delta-1},
\]
proving
\begin{align*}
\limsup_n \log_e\left(\f{8 \exp\left(- C n^{2\delta -1}\right)}{{n \choose r} \left(\f{1}{2}\right)^{r \choose 2}} \right)=-\infty,
\end{align*}
which upon combining with \eqref{eq:num_bd} and \eqref{eq:denom_bd}, yield \eqref{eq:large_clique_up_bd}.

Thus to complete the proof we only need to show that for $r=cn^{1/2-\vep}$, where $c$ and $\vep$ are some positive constants,
\beq\label{eq:Var_small}
\f{\mathrm{Var}(n_{\Graph(n,1/2)}({\sf C}_r))}{\mu} -1 + 2\left(\f{1}{2}\right)^{r \choose 2}\ra 0, \quad \text{ as } n\ra \infty.
\eeq
Using \eqref{eq:var_W} we note that
\beq
\f{\mathrm{Var}(n_{\Graph(n,1/2)}({\sf C}_r))}{\mu} -1 + 2\left(\f{1}{2}\right)^{r \choose 2}= \left(\f{1}{2}\right)^{r \choose 2} + \sum_{t=1}^{{r \choose 2}-1}|\Gamma_{\wt{\sf H}}^t| \left(\left(\f{1}{2}\right)^t- \left(\f{1}{2}\right)^{r \choose 2}\right), \notag
\eeq
for any isomorphic copy $\wt{\sf H}$ of the complete graph on $r$ vertices. If $\wt{\sf H}_1$ and $\wt{\sf H}_2$ are  two different isomorphic copies of the complete graph on $r$ vertices, with  $s$ common vertices between them, then there are ${r \choose 2} - {s \choose 2}$ edges of $e(\wt{\sf H}_2)$ that are not part of the edge set of $\wt{\sf H}_1$. Thus the above expression simplifies to the following:
\beq\label{eq:var_pois_bd}
\f{\mathrm{Var}(n_{\Graph(n,1/2)}({\sf C}_r))}{\mu} -1 + 2\left(\f{1}{2}\right)^{r \choose 2} \le  \left(\f{1}{2}\right)^{r \choose 2} + \sum_{s=2}^{r-1}{r \choose s} {n-r \choose r-s}\left(\f{1}{2}\right)^{{r \choose 2}- {s \choose 2}}. 
\eeq
Now we need to control the summation appearing in the \abbr{RHS} of \eqref{eq:var_pois_bd}. Denoting
\[
a_s := {r \choose s} {n-r \choose r-s}\left(\f{1}{2}\right)^{{r \choose 2}- {s \choose 2}},
\]
we note
\begin{align}
\varrho_s:=\f{a_{s+1}}{a_s} = \f{{r \choose s+1} {n-r \choose r-s-1}}{{r \choose s} {n-r \choose r-s}} \left(\f{1}{2}\right)^{{s \choose 2}- {s+1 \choose 2}}
 &= \f{s!\left((r-s)!\right)^2(n-2r+s)!}{(s+1)!\left((r-s-1)!\right)^2(n-2r+s+1)!}\left(\f{1}{2}\right)^{- s}\notag\\
& = \f{(r-s)^2}{(s+1)(n-2r+s+1)}2^s.\label{eq:cons_ratio}
\end{align}
Thus
\beq\label{eq:ratio_condition}
\varrho_s \ge 1 \Leftrightarrow  2^s \ge \f{(s+1)(n-2r+s+1)}{(r-s)^2}.
\eeq
Using the above equivalent condition of $\varrho_s\ge 1$ we show below that the sequence $\{a_s\}$ monotonically decreases initially, achieves a minimum for $s=O(\log n)$, and then monotonically increases afterwards for all $s$ except possibly for $s=r-3,r-2$. This observation helps us to compute $\sup_{2 \le s \le r-1} a_s$, and therefore we can easily obtain an upper bound on summation appearing in the \abbr{RHS} of \eqref{eq:var_pois_bd}. To this end,  note that for any fixed $K$ if $s \le K$, then
\[
r^2 2^s \le  (cn^{1/2-\vep})^2 2^K < \eta n,
\]
for any $\eta >0$, proving that $\{a_s\}$ is strictly decreasing for $s=2,3,\ldots,K$. We next try to find a necessary condition on $\varrho_s \ge 1$. If $\varrho_s \ge1$, then using \eqref{eq:ratio_condition} we note that we must have 
\[
2^s \ge \f{(s+1)(n-2r+s+1)}{(r-s)^2} \ge \f{3(n-2r+3)}{r^2}\ge \f{2n}{r^2}= \f{2}{c^2} n^{2\vep},
\]
where in the last step we use the fact that $r=c n^{1/2-\vep}$. The above expression therefore implies that $s \ge \vep \log_2\hspace{-2pt}n$. Hence, if a minimum of $\{a_s\}$ occurs at $s^*$, then we should have that $s^* \ge \vep \log_2\hspace{-2pt}n$. Next we show that for any $s >s^*$, $\varrho_s >1$ for all $s$, except possibly for $s=r-3,r-2$. To see this, consider
\beq\label{eq:ratio_rho}
\f{\varrho_{s+1}}{\varrho_s}= \f{a_{s+2}/a_{s+1}}{a_{s+1}/a_{s}}=2 \f{(r-s-1)^2}{(r-s)^2}\f{(s+1)(n-2r+s+1)}{(s+2)(n-2r+s+2)}.
\eeq
Since $s \ge \vep \log_2\hspace{-2pt}n$, and $r=cn^{1/2-\vep}$, the product of the ratio
\[
\f{(s+1)(n-2r+s+1)}{(s+2)(n-2r+s+2)} 
\]
can be made as close to $1$ as needed for all large $n$. Also note that if $s \le r-4$, then
\[
\f{(r-s-1)^2}{(r-s)^2} = \left(1-\f{1}{r-s}\right)^2\ge \f{9}{16}.
\]
Hence, from \eqref{eq:ratio_rho} we conclude that $\{a_s\}$ is monotonically increasing for all integer $s$ between $s^*$, and $r-4$. Hence, from \eqref{eq:var_pois_bd} we have
\begin{align*}
\f{\mathrm{Var}(n_{\Graph(n,1/2)}({\sf C}_r))}{\mu} -1 + 2\left(\f{1}{2}\right)^{r \choose 2} \le  \left(\f{1}{2}\right)^{r \choose 2} + r \max_{s=2,r-3,r-2,r-1}\left\{{r \choose s} {n-r \choose r-s}\left(\f{1}{2}\right)^{{r \choose 2}- {s \choose 2}}\right\}. 
\end{align*}
It is easy to note that
\[
r \max_{s=2,r-3,r-2,r-1}\left\{{r \choose s} {n-r \choose r-s}\left(\f{1}{2}\right)^{{r \choose 2}- {s \choose 2}}\right\}= O\left(r^2 n \left(\f{1}{2}\right)^r\right)= o(1),
\]
and thus we obtain \eqref{eq:Var_small} completing the proof.
\end{proof}
\begin{rmk} In \cite[Example 5.1.4]{BHJ} the number of copies of large complete subgraphs of an Erd\H{o}s-R\'{e}yni graph on $n$ vertices is also shown to converge to a Poisson distribution in total variation distance. However, their set-up significantly differs from ours. In particular, they assume that the connectivity probability $p$ of the Erd\H{o}s-R\'{e}yni graph satisfies the inequality $p \ge n^{-2/(r-1)}$, and $r \ll n^{1/3}$, neither of which holds here. Both in \cite[Example 5.1.4]{BHJ}, and Lemma \ref{lem:large_clique} the starting point is \eqref{eq:d_tv_bd}-\eqref{eq:var_W}. However, the analysis of the terms appearing in \eqref{eq:var_W} are done differently.
\end{rmk}


\subsection{Proof of Theorem \ref{thm:binary_graph}}
Recall that $\Graph_n^{b}$ is a graph on $n$ vertices where its vertices are binary tuples of length $k$ containing an odd number of ones, excluding the vector of all ones, where $n=2^{k-1}-1$. A pair of vertices $u$ and $v$ are connected if $\langle u, v \rangle =1$, where additions and multiplications are the corresponding operations in the binary field. It is easy to see that $\Graph_n^b$ is a $d$-regular graph with $d=2^{k-2}-2$. Also it can be shown that the number of common neighbors for any pair of adjacent vertices is $2^{k-3}-3$ and the same for a pair of non-adjacent vertices is $2^{k-3}-1$ (See \cite[Section 3, Example 4]{prg}). Therefore, it is immediate that $\Graph_n^b$ satisfies {\bf (A1)}-{\bf (A2)} with $p=1/2, \, \delta =0$, and $C=3$. Thus the proof of Theorem \ref{thm:binary_graph}(i) is a direct application of Theorem \ref{thm:main}. In the rest of this section we prove Theorem \ref{thm:binary_graph}(ii)-(iii). To complete the proof we need the following elementary result.

\begin{lem}\label{lem:lin_algeb}
Let $B$ be a binary $m \times k$ matrix with $\mathrm{rank}(B)=\ell$ such that $\ell \le m \le k$. Then given any binary vector $b$ of length $m$ such that it belongs to the column space of $B$, the number of solutions of the linear system of equation $Bx=b$ is $2^{k-\ell}$. 
\end{lem} 

\bigskip
The proof of Lemma \ref{lem:lin_algeb} is easy and follows from standard linear algebra argument and an induction argument. We omit the details. Now we are ready to prove Theorem \ref{thm:binary_graph}(ii)-(iii). First we begin with the proof of Theorem \ref{thm:binary_graph}(ii).

\begin{proof}[Proof of Theorem \ref{thm:binary_graph}(ii)]
First let us find the size of the largest independent set in $\Graph_n^b$. To this end, let us assume that $\{v_1,v_2,\ldots,v_r\}$ forms an independent set of size $r$. Then, we have that $\langle v_i, v_j \rangle =0$ for all $i \ne j \in [r]$. We further claim that $\{v_1,v_2,\ldots,v_r\}$ is a set of mutually independent vectors when viewed as vectors of length $k$ over the binary field. To see this, if possible let us assume that there exists coefficients $c_j \in \{0,1\}$, $j \in [r]$, not all of them zero, such that 
\beq\label{eq:max_indep_set}
\sum_{j=1}^r c_j v_j=0. 
\eeq
Note that $\langle v_j, v_j \rangle=\langle v_j , {\bf 1} \rangle =1$ for all $j \in [r]$, where ${\bf 1}$ is the vector of all ones. Therefore, taking an inner product with $v_i$ on both sides of \eqref{eq:max_indep_set} we have that $c_i=0$ for all $i \in [r]$ which  is a contradiction. Therefore $\{v_1,v_2,\ldots,v_r\}$ forms an independent set of size $r$. Since the number of independent vectors of length $k$ over any field is at most $k$, recalling that $k=\log_2(n+1) +1$, the assertion about that size of the maximal independent set in $\Graph_n^b$ is established.

Now let us show that the number of independent sets of size $(1-\eta)\log_2\hspace{-2pt}n$ is roughly same as that of a $\Graph(n,\f{1}{2})$. We begin by noting that if $\{v_1,v_2,\ldots,v_r\}$ is an independent set then so is $\{v_1,v_2,\ldots,v_{r-1}\}$ and $v_r$ must be a common non-neighbor of $v_1,v_2,\ldots,v_{r-1}$. To find the number of common non-neighbors we define $B_r$ to be the $r\times k$ matrix whose rows are $v_1,v_2,\ldots,v_{r-1}, {\bf 1}$. Then we note that $v_r$ must be a solution of 
\beq\label{eq:lin_eq}
B_r x = \begin{pmatrix} 0  \\ \vdots \\ 0 \\ 1 \end{pmatrix}.
\eeq	
The first $(r-1)$ rows of \eqref{eq:lin_eq} ensures that $v_r$ is a common non-neighbor of $\{v_1,v_2,\ldots,v_{r-1}\}$ and the last row ensures that the number of ones in $v_r$ is odd, which is necessary for $v_r$ to be a vertex of $\Graph_n^b$. From Lemma \ref{lem:lin_algeb} it follows that, if $\mathrm{rank}(B_r)=r$, then the number of such solutions is $2^{k-r}$. We also observe that any of $\{v_1,v_2,\ldots,v_{r-1},{\bf 1}\}$ cannot be a solution of \eqref{eq:lin_eq}. Because
\[
\langle v_i,v_i \rangle = \langle v_i, 1 \rangle =1, \quad i=1,2,\ldots,r-1.
\]
Thus we deduce that given any collection of $(r-1)$ vertices $\{v_1,v_2,\ldots,v_{r-1}\}$, if $\mathrm{rank}(B_r)=r$ then the number of common non-neighbors is $2^{k-r}$. Now let us try to understand when one can have $\mathrm{rank}(B_r) \ne r$. If $\mathrm{rank}(B_r) < r$ then we must have coefficients $c_1,c_2,\ldots, c_r \in \{0,1\}$, all of them not zero, such that
\beq\label{eq:lin_combin}
\sum_{j=1}^{r-1} c_j v_j + c_r {\bf 1}=0. 
\eeq
Since $\langle v_j, {\bf 1} \rangle =1$ for $j\in [r-1]$, taking an inner product with ${\bf 1}$ from \eqref{eq:lin_combin} we deduce that $\sum_{j=1}^r c_j=0$. Since $\{v_1,v_2,\ldots,v_{r-1}\}$ forms an independent set of size $(r-1)$, we have that $\langle v_i, v_j \rangle =0$ for $i \ne j \in [r-1]$. Thus taking an inner product with $v_j$ for $j \in [r-1]$, from \eqref{eq:lin_combin} we further deduce that $c_j +c_r=0$ for all $j \in [r-1]$. Combining the last two observations we deduce that $r$ must be even and $c_j=1$ for all $j$. Therefore we can only have $\mathrm{rank}(B_r) \in \{r-1,r\}$ when $\{v_1,v_2,\ldots,v_{r-1}\}$ is an independent set.  This further implies that
\beq\label{eq:lin_combin_1}
v_{r-1}=v_1+v_2+ \cdots+v_{r-2}+{\bf 1},
\eeq 
whenever $\mathrm{rank}(B_r)=r-1$. It is also easy to check that in this case there are no solution to the linear system of equations \eqref{eq:lin_eq}. In summary we have the following situation. Given any independent set $\{v_1,v_2,\ldots,v_{r-1}\}$, one of the following two conditions hold:
\begin{itemize}
\item $\mathrm{rank}(B_r)=r$ and there are $2^{k-r}$ common non-neighbors of $\{v_1,v_2,\ldots,v_{r-1}\}$.
\item $\mathrm{rank}(B_r)=r-1$ (only when $r$ is even).  Then \eqref{eq:lin_combin_1} holds, and there are no common non-neighbors of $\{v_1,v_2,\ldots,v_{r-1}\}$.
\end{itemize}
Using these observations we are now ready to prove our claim for the independent sets.

We start by choosing $v_1$, which we can obviously do in $(2^{k-1}-1)$ many ways. For any such choice of $v_1$ there are $2^{k-2}$ many non-neighbors of $v_1$. This is the number of choices of $v_2$. For any such choice of $\{v_1,v_2\}$, by our above argument there are $2^{k-3}$ choices of $v_3$ such that $v_3$ is a common non-neighbor of $v_1$ and $v_2$. Out of these $2^{k-3}$ many choices if $v_3=v_1+v_2+{\bf 1}$ then we cannot extend $\{v_1,v_2,v_3\}$ to a larger independent set as we have already seen that in this case there no common non-neighbor of these three vertices. Thus the number of choices of $v_3$ such that it can be extended to form a larger independent set is $(2^{k-3}-1)$. For each of these choices we again have already established that $\mathrm{rank}(B_4)=4$. Therefore, we obtain that the number of choices of $v_4$ is $2^{k-4}$. Continuing the above we see that the number of independent sets of size $r=(1-\eta) k$ is
\begin{align}
\f{1}{r!} (2^{k-1}-1)2^{k-2} (2^{k-3}-1)2^{k-4} \cdots  & = \f{1}{r!}\prod_{\ell=1}^r 2^{k-\ell} \prod_{\substack{1 \le \ell \le r \\ \ell=\text{odd}}} \left(1 - 2^{-(k-\ell)}\right)\notag\\
& = \f{1}{r!}2^{{k \choose 2}-{k-r \choose 2}} \prod_{\ell=\lceil \f{k-r}{2}\rceil}^{\f{k-1}{2}} \left(1 - 2^{-2\ell}\right)\notag\\
& \sim \f{1}{r!}2^{{k \choose 2}-{k-r \choose 2}}  \exp\left(-\sum_{\ell=\lceil \f{k-r}{2}\rceil}^{\f{k-1}{2}} 2^{-2\ell}\right)\sim \f{1}{r!}2^{{k \choose 2}-{k-r \choose 2}}, \notag
\end{align}
as $k \ra \infty$. In the last two steps we use the fact that $\eta >0$. An easy calculation yields that 
\begin{align}
\E\left[n_{\Graph(n,\f{1}{2})}({\sf I}_r)\right] \sim \f{1}{r!}n^r \left(\f{1}{2}\right)^{r \choose 2} \sim \f{1}{r!}2^{(k-1)r} \left(\f{1}{2}\right)^{r \choose 2}. \notag
\end{align}
Since ${k \choose 2} + {r \choose 2} - {k-r \choose 2} = r(k-1)$ the proof completes.
\end{proof}


\bigskip

To prove Theorem \ref{thm:binary_graph}(iii) it will be easier to consider the following alternate representation of the vertices of $\Graph_n^b$. Note that for any vertex $v$ in $\Graph_n^{b}$ we can define $\bar{v}$ such that $v+\bar{v}={\bf 1}$. Instead of considering $v$ as its vertices, equivalently we can view $\Graph_n^{b}$ as a graph with vertices $\bar{v}={\bf 1}+v$. Note that in this representation $\bar{u}$ and $\bar{v}$ are connected if and only if $\langle \bar{u}, \bar{v}\rangle =0$. We work with this representation to prove Theorem \ref{thm:binary_graph}(iii). Theorem \ref{thm:binary_graph}(iii) actually follows from \cite{thom2}. We include the proof for completeness. It can be noted that the same proof shows the existence of a clique of size $\sqrt{n+1}-1$ in $\Graph_n^b$.

\begin{proof}[Proof of Theorem \ref{thm:binary_graph}(iii)]
Denote $t:= \f{k-1}{2}$ and fix any positive integer $t'$. We create a clique of size $(t+t')$ as follows: First we consider $t$ mutually independent vectors $\{\bar{v}_1,\bar{v}_2,\ldots,\bar{v}_t\}$ which form a clique of size $t$ and then we pick $t'$ more vectors from $\text{Span}(\bar{v}_1,\bar{v}_2,\ldots, \bar{v}_t)$. Since $\langle \bar{v}_j, {\bm 1}\rangle =0$, it is easy to check that any vector belonging to $\text{Span}(\bar{v}_1,\bar{v}_2,\ldots, \bar{v}_t)$ is a common neighbor of $\{\bar{v}_1,\bar{v}_2,\ldots,\bar{v}_t\}$. Therefore we obtain a clique of size $(t+t')$. Thus we only need to count in how many ways one can construct $t$ mutually independent vectors which themselves form a clique and in how many ways one can choose $t'$ more vectors from the span of those $t$ mutually independent vectors.

Fix $\ell <t$ and let $\{\bar{v}_1,\bar{v}_2,\ldots,\bar{v}_\ell\}$ be a collection $\ell$ mutually independent vectors. Then $\bar{v}_{\ell+1}$ is a common neighbor of $\{\bar{v}_1,\bar{v}_2,\ldots,\bar{v}_\ell\}$ if $\wt{B}_{\ell+1} \bar{v}_{\ell+1}={\bm 0}$, where $\wt{B}_{\ell+1}$ is a $(\ell+1)\times k$ matrix with rows $\bar{v}_1,\bar{v}_2,\ldots,\bar{v}_\ell, {\bm 1}$, and ${\bm 0}$ is the zero vector. Since $\langle \bar{v}_j , {\bf 1} \rangle =0$ for $j \in [\ell]$ and $\{\bar{v}_j\}_{j \in [\ell]}$ is a collection of mutually independent vectors, we have $\mathrm{rank}(\wt{B}_{\ell+1})=\ell+1$. Therefore, number of solutions of $\wt{B}_{\ell+1} x = {\bf 0}$ is $2^{k-\ell-1}$. Removing $2^{\ell}$ many vectors  that belong to $\text{Span}(\bar{v}_1,\bar{v}_2,\ldots, \bar{v}_\ell)$ we obtain that the number of choices of common neighbors of $\{\bar{v}_1,\bar{v}_2,\ldots,\bar{v}_\ell\}$ that are independent from the latter collection of vectors is $2^{k-\ell-1} - 2^\ell$. Continuing by induction we then deduce that the number of mutually independent vectors $\{\bar{v}_1,\bar{v}_2,\ldots,\bar{v}_t\}$ such that they themselves form a clique of size $t$ is 
\[
(2^{k-1}-1)(2^{k-2} - 2)\cdots(2^{t+1}-2^{t-1}).
\]
Recalling that $k=2t+1$, we now have that the number of cliques of size $(t+t')$ in $\Graph_n^{b}$ is at least
\begin{align*}
&\f{1}{(t+t')!}(2^{k-1}-1)(2^{k-2} - 2)\cdots(2^{t+1}-2^{t-1})(2^t-(t+1))(2^t-(t+2))\cdots (2^t- (t+t'))\\
\sim & \f{1}{(t+t')!}\prod_{\ell=1}^{t} 2^{k-\ell} \prod_{\ell \ge 1} \left(1-\f{1}{4^\ell}\right) 2^{t t'}\\
\sim  & \frac{1}{(t+t')!}n^t \prod_{\ell=0}^{t-1}2^{-\ell}  \prod_{\ell \ge 1} \left(1-\f{1}{4^\ell}\right)  n^{t'} 2^{-tt'} \sim\f{1}{(t+t')!} n^{t+t'}  \left(\f{1}{2}\right)^{t+t' \choose 2} 2^{t' \choose 2} \prod_{\ell \ge 1} \left(1-\f{1}{4^\ell}\right).
\end{align*}
Since for $x \in \f{1}{2}\log 2$, one has $1-x \ge e^{-2x}$, we obtain that $\prod_{\ell \ge 1} \left(1-\f{1}{4^\ell}\right) \ge \exp(-2/3)$. Finally noting that in $\Graph(n,\f{1}{2})$ the number of cliques of size $(t+t')$ is approximately $\f{1}{(t+t')!} n^{t+t'}  \left(\f{1}{2}\right)^{t+t' \choose 2} $ the proof completes.
\end{proof}

\section*{Acknowledgements} 
Much of the work was done while AB was a visiting assistant professor at Department of Mathematics, Duke University. SB has been partially supported by NSF-DMS grants  1310002, 160683, 161307 and SES grant 1357622. SC has been partially supported by SES grant 1357622 and by a research assistantship from SAMSI. AN has been partially supported by NSF-DMS grants 1310002 and 1613072.

\bibliographystyle{plain}
\bibliography{pseudo}



\end{document}